\documentclass[a4paper,10pt]{amsart}
\usepackage[T1]{fontenc}
\usepackage[utf8]{inputenc}
\usepackage{amssymb,amsmath,amsthm,amsfonts,mathrsfs,latexsym, hyperref, color,caption,stmaryrd}
\usepackage{graphicx,enumerate}
\usepackage{epstopdf}

\usepackage{mathptmx}
\usepackage[scaled=.95]{helvet}
\usepackage{courier}

\renewcommand\theequation{\thesection.\arabic{equation}}
\numberwithin{equation}{section}

\newtheorem{thm}{Theorem}[section]
 
 \newtheorem{rem}[thm]{Remark}
 \newtheorem{prop}[thm]{Proposition}
 \newtheorem{lem}[thm]{Lemma}
 \newtheorem{cor}[thm]{Corollary}

\newcommand{\R}{\mathbb{R}}

\newcommand{\N}{\mathbb{N}}
\newcommand{\U}{\mathcal{U}}

\newcommand{\ra}{\rightarrow}
\newcommand{\vp}{\varphi}
\newcommand{\e}{\varepsilon}

\title{Infinite horizon sparse optimal control}
\author{Dante Kalise}
\address{Johann Radon Institute of Computational and Applied Mathematics, Austrian Academy of Sciences, Altenbergerstra{\ss}e 69, A-4040 Linz, Austria}
\email{dante.kalise@oeaw.ac.at,}

\author{Karl Kunisch}
\address{Institute for Mathematics and Scientific Computing, University of Graz, Heinrichstra{\ss}e 36, A-8010 Graz, Austria, and Johann Radon Institute of Computational and Applied Mathematics, Austrian Academy of Sciences}
\email{karl.kunisch@uni-graz.at}
\thanks{The authors acknowledge the support of the ERC-Advanced Grant OCLOC “From Open to Closed Loop Optimal Control of PDEs”.}

\author{Zhiping Rao}
\address{Johann Radon Institute of Computational and Applied Mathematics, Austrian Academy of Sciences, Altenbergerstra{\ss}e 69, A-4040 Linz, Austria}
\email{zhiping.rao@oeaw.ac.at}

\begin{document}

\maketitle

\begin{abstract}
A class of infinite horizon optimal control problems involving $L^p$-type cost functionals with $0<p\leq 1$ is discussed. The existence of optimal controls is studied for both the convex case with $p=1$ and the nonconvex case with $0<p<1$, and the sparsity structure of the optimal controls promoted by the $L^p$-type penalties is analyzed. A dynamic programming approach is proposed to numerically approximate the corresponding sparse optimal controllers. 
\end{abstract}


\section{Introduction}
In this paper, we investigate the following infinite horizon optimal control problem: given $\lambda>0$, and $U$ a convex compact subset of $\R^m$, solve for any $x\in\R^d$ 
\begin{equation}\label{ocinfini}
\inf_{u\in L^{\infty}(0,\infty;U)} \int^{\infty}_0e^{-\lambda s}\ell(y(s),u(s))ds
\end{equation}
subject to the dynamical constraint
\begin{equation}\label{ds}
\left\{
\begin{array}{ll}
\dot y(s)=f(y(s),u(s)) & \quad\text{a.e.}\;\; s>0,\\
y(0)=x.
\end{array}
\right.
\end{equation}
Here $f:\R^d\times U\ra \R^d$ and $\ell:\R^d\times U\ra \R$ are continuous functions. Given $\gamma>0$ and $p\in(0,1]$, $\ell(\cdot,\cdot)$ is defined as follows: for any $x\in\R^d$ and $u=(u^1,\ldots,u^m)\in U$,
\[
\ell(x,u)=\ell_1(x)+\gamma\|u\|_{p}^p,
\]
with $\ell_1$ a nonnegative and strictly convex function, and
\[
\|u\|_{p}=\left(\sum^m_{k=1}|u^k|^p\right)^{1/p}.
\]

The most remarkable issue of the problem \eqref{ocinfini} is the
presence of the nonsmooth term $\|\cdot\|_p^p$ in the cost functional.
Moreover, the problem is nonconvex when $0<p<1$, which induces new
properties of the optimal controls and also makes the analysis of the
problem more complicated. The use of these functionals implies that
optimal controls can be identically $0$ on subsets of positive measure.
This is referred to as sparsity. Intuitively, the controls are switched
off completely in intervals where, in the case smooth cost functionals
were used, they would be small but nonzero. Finite-dimensional, finite
horizon, optimal control problems with nonsmooth penalizations in
the control variable were originally studied in the context of the
so-called \textsl{minimum fuel optimal control problem},
\cite{NEU65,ROS07}. In the linear case, necessary conditions for $L^1$
optimal control problems were derived in \cite{HAJ79}. Linear-quadratic
optimal control problems with an additional $L^1$-cost on the control were recently discussed in \cite{ALT15}. For the nonlinear
case, first order necessary and second order sufficient optimality
conditions for control problems involving an $L^1$-term in the cost
functional were obtained in \cite{VM06}. Applications to sparse optimal
control for multi-agent systems were addressed in \cite{CFPT15}.

For infinite-dimensional dynamics, one of the areas of application for
sparsity functionals is optimal actuator timing and placement \cite{S09,CZ13}. Vanishing of the control in temporal or spatial regions indicates that it is not worthwhile to
assert control force on the system there. Such a type of information is
not available from quadratic control penalties. Open loop, finite
horizon, optimal control problems with sparsity enhancing functionals
in the context of partial differential equations were addressed in
several papers in the recent past. Here we can just mention only a few
of them, see \cite{HSW12,CCK12,PV13}.

The analysis in the present work is focused on the existence of solutions
and the sparsity properties of infinite horizon optimal controls invoked
by $L^p$ control penalties with $0<p\leq 1$. The problem is formulated in
Section 2. To derive the existence result, we discuss the convex case
$p=1$ in Section 3, and the nonconvex case $0<p<1$ in Section 4. For the
convex case, the existence result is established for fully nonlinear
dynamics $f$. A similar result has been obtained in \cite{CS04} for
finite horizon problems. For the nonconvex case, the existence of
solutions is not guaranteed in general. We consider the case when $u$
arises linearly in the dynamics, and pose it as a time-discretized
nonconvex problem in the infinite dimensional sequence space $\ell^p$. The
existence result for the reformulated problem is then derived following
the method introduced in \cite{IK14}. Subsequently, first order
optimality conditions are derived and the sparsity structure of optimal
controls is investigated for both convex and nonconvex cases in Section
5. In \cite{VM06,ALT15} control constraints of $L^\infty$ type are
used, whereas we allow for $L^q$-constraints with $1\leq q\leq \infty$.
An example with Eikonal dynamics is analyzed in Section 6 and the
sparse region of the optimal control is explicitly given.

Turning to Section 7, we acknowledge that the numerical approximation
of infinite horizon optimal control problems is a challenging task. In
the open-loop context, the infinite horizon problem can be treated
either by a sequential approximation of finite horizon control problems
or via pseudospectral collocation methods \cite{FR08,GP09}. In the case
of closed-loop optimal controls, there exists a solid computational
framework based on the solution of the algebraic Riccati equations for
the linear quadratic case. For the more general case with nonlinear
dynamics and nonquadratic costs, computations can be carried out via
dynamic programming. More precisely, we compute the value function
associated to the control problem by solving a Hamilton-Jacobi-Bellman
(HJB) equation, and then the optimal control and the associated optimal
trajectories are reconstructed through an online feedback mapping which requires the solution of a nonlinear optimization problem. Comparing the HJB approach to
open-loop methods, it has the advantage of being in feedback form,
yielding robust controllers in the presence of perturbations . In the
HJB approach we are, of course, confronted with the so-called \textsl{curse of
dimensionality}. However, for low-dimensional dynamics, the design of
numerical schemes is well-established (we refer to \cite[Chapter 8]{FFSIAM} for an
updated introduction to this topic). In Section 7, numerical simulations
are carried out based on the algorithms introduced in
\cite{AFK15,KKK15}. For the example with Eikonal dynamics, the numerical
results confirm our analysis on the sparsity properties of the optimal
controls.

\section{The optimal control problem}
Recalling that the infinite horizon, optimal control problem is given by
\[
\inf_{u\in L^{\infty}(0,\infty;U)} \left\{J(x,u):=\int^{\infty}_0e^{-\lambda s}\ell(y(s),u(s))ds:\ \dot y(s)=f(y(s),u(s))\ \text{for}\ s>0,\ y(0)=x\right\},
\]
we make the following assumptions.
\begin{itemize}
\item[{\bf (H1)}] 
There exists $L>0$ such that 
\[
\|f(x_1,u)-f(x_2,u)\|_2\leq L\|x_1-x_2\|_2\\, \quad\text{for all}\ u\in U.
\]
\item[{\bf (H2)}]
For each $x\in\R^d$, there exists $(y^*(\cdot),u^*(\cdot))$ satisfying \eqref{ds} such that
\[
J(x,u^*)<+\infty.
\]
\end{itemize}
\begin{rem}
Assumption {\bf (H2)} is a condition on the dynamics $f$, in combination with the factor $\lambda$. It is also related to controllability assumptions. For example, consider the linear-quadratic case with
\[
f(x,u)=Ax+Bu,\ \ell_1(x)=\|x\|^2_2,
\]
where $A\in\R^{n\times n}$ and $B\in\R^{n\times m}$. Then {\bf (H2)} is satisfied if $\lambda>2\rho(A)$, where $\rho(A)$ is the spectral radius of $A$. On the other hand, {\bf (H2)} also holds if the Kalman's controllability rank condition is satisfied. In this case, for each initial point $x\in\R^d$, there exists a trajectory leading $x$ to the origin in finite time, and then the control is switched off so that the running cost stays zero from then on.
\end{rem}
The value function $v:\R^d\ra \R$ is introduced as follows:
\[
v(x):=\inf_{u\in L^{\infty}(0,\infty;U)} J(x,u),\quad \forall\,x\in\R^d.
\]
The value function satisfies the following dynamic programming principle: for any $x\in\R^d$ and $h\geq 0$,
\[
v(x)=\inf_{u\in L^{\infty}(0,h;U)}\int^{h}_0e^{-\lambda s}\ell(y(s),u(s))ds+e^{-\lambda h}v(y(h)).
\]
By standard arguments, it is deduced that $v$ is the unique viscosity solution to the stationary HJB equation:
\begin{equation}\label{HJB}
\lambda v(x)+H(x,Dv(x))=0\ \text{for}\ x\in\R^d,
\end{equation}
where the Hamiltonian $H:\R^d\times\R^d\ra\R$ is given by
\begin{equation}\label{Hamiltonian}
H(x,p)=\sup_{u\in U}\left\{-f(x,u)\cdot p-\ell(x,u)\right\}.
\end{equation}
Note that for $0<p<1$, $\ell$ is not a convex function. Therefore, the existence of a minimizer for the problem \eqref{ocinfini} is not guaranteed. In the following, we will discuss the convex case with $p=1$ and the nonconvex case with $0<p<1$ separately.

\section{The convex case: $p=1$}
In this section, the following general convexity condition is assumed.
\begin{itemize}
\item[{\bf (H3)}]
For each $x\in\R^d$, the following subset of $\R^d\times\R$ is convex:
\[
\{(f(x,u),\xi)\,:\, \xi\geq\ell(x,u),\ u\in U\}.
\]
\end{itemize}
\begin{rem}
If $f$ is affine in $u$, then {\bf (H3)} is satisfied.
\end{rem}

In \cite{AKT01}, existence of a minimizer for the problem \eqref{ocinfini} is obtained through an approximation approach in the case of control-affine dynamics. The idea is to approximate the infinite horizon problem \eqref{ocinfini} by a family of finite horizon problems. We extend this approach to the nonlinear case. For any fixed $T>0$ and $x\in\R^d$, consider the problem:
\begin{equation}\label{ocT}
 v_T(x):=\inf_{u\in L^{\infty}(0,T;U)} J_T(x,u)\,,
\end{equation}
where
\[
J_T(x,u)=\int^T_0 e^{-\lambda s}\ell(y(s),u(s))ds\,,
\]
and $(y(\cdot),u(\cdot))$ satisfies the dynamical system 
\[
\left\{
\begin{array}{ll}
\dot y(s)=f(y(s),u(s)) & \text{a.e.}\ s>0,\\
y(0)=x.
\end{array}
\right.
\]
Let $G:\R^d\times\R\ra\R^d\times\R$ be the set-valued multifunction defined by
\[
G(x,\eta):=\{(f(x,u),\lambda \eta+\xi)\,:\,\xi\geq\ell(x,u),\ u\in U\},\ \forall\,x\in\R^d,\ \eta\in\R.
\]
Given $T>0$, consider the following problem: for any $x\in\R^d$
\begin{equation}\label{ocAugmented}
 w_T(x)=\inf_{\eta} e^{-\lambda T}\eta(T),
\end{equation}
where $\eta(\cdot)$ together with a corresponding $y(\cdot)$ satisfies the following differential inclusion
\begin{equation}\label{dsAugmented}
\left\{
\begin{array}{ll}
(\dot y(s),\dot \eta(s))\in G(y(s),\eta(s)),\ \text{a.e.}\ s\in (0,T),\\
(y(0),\eta(0))=(x,0).
\end{array}
\right.
\end{equation}

\begin{lem}\label{ExistenceAugmented}
There exists a minimizer $\eta_T$ for problem \eqref{ocAugmented}.
\end{lem}

\begin{proof}
For any $x\in\R^d$, let $(\eta_n)_{n\in\N}$ be a minimizing sequence for \eqref{ocAugmented}, i.e.
\[
w_T(x)=\lim_{n\ra\infty} e^{-\lambda T}\eta_n(T).
\]
For each $n\in\N$, let $y_n$ be the corresponding trajectory such that $(y_n,\eta_n)$ satisfies \eqref{dsAugmented}. By the definition of $G$ and the selection theorem \cite[Corollary 1, pp. 91]{AC84}, there exists a measurable $u_n$ such that
\[
\left\{
\begin{array}{lll}
 \dot y_n(s)=f(y_n(s),u_n(s)) & \text{a.e.}\ s\in (0,T),\\
 \dot \eta_n(s)\geq \lambda \eta_n(s)+\ell(y_n(s),u_n(s)) & \text{a.e.}\ s\in (0,T),\\
 (y_n(0),\eta_n(0))=(x,0).
\end{array}
\right.
\]
We introduce
\[
\tilde \eta_n(s)=\int^s_0 e^{-\lambda(s'-s)}\ell(y_n(s'),u_n(s'))ds',\ \forall\,s\in [0,T],\ n\in\N,
\]
which satisfies that
\[
\frac{d}{ds}\left[e^{-\lambda s}\tilde \eta_n(s)\right]=e^{-\lambda s}\ell(y_n(s),u_n(s))\leq \frac{d}{ds}\left[e^{-\lambda s} \eta_n(s)\right],
\]
and thus,
\[
e^{-\lambda s}\tilde \eta_n(s)\leq e^{-\lambda s} \eta_n(s),\ \forall\,s\in[0,T],\ n\in\N.
\]
Note that
\[
w_T(x)\leq e^{-\lambda T}\tilde \eta_n(T)\leq e^{-\lambda T} \eta_n(T),\ n\in\N,
\]
and hence
\[
w_T(x)=\lim_{n\ra\infty} e^{-\lambda T}\tilde \eta_n(T).
\]
By setting $\xi_n(s):=e^{-\lambda s}\tilde \eta_n(s)$ for $s\in[0,T]$, we have
\[
\dot \xi_n(s)=e^{-\lambda s}\ell_n(y_n(s),u_n(s)),\ \text{a.e.}\ s\in (0,T).
\]
Since $f(x,u)$ is $L$-Lipschitz continuous w.r.t $x$ uniformly on $u$, by Gronwall inequality there exists a constant $C>0$, such that
\[
\|y_n(s)\|_{\infty}\leq e^{LT}\sqrt{\|x\|^2_2+CT},\ \forall\,s\in[0,T].
\]
Due to the continuity of $f$ and $\ell$, we deduce that $\|\dot y_n(\cdot)\|_{\infty}$ and $\|\dot \xi_n(\cdot)\|_{\infty}$ are uniformly bounded in $[0,T]$. Therefore, the Arzel\`a-Ascoli Theorem and the Dunford-Pettis Theorem imply that there exists $y_T,z\in L^1(0,T;\R^d)$ and $\xi_T,\tau\in L^1(0,T;\R)$ such that, after possibly passing to a subsequence,
\[
(y_n,\xi_n)\ra (y_T,\xi_T)  \text{ uniformly in } [0,T],\ \text{as}\ n\ra\infty,
\]
\[
(\dot y_n,\dot \xi_n)\ra (z,\tau)  \text{ weakly in } L^1(0,T;\R^{d+1}),\ \text{as}\ n\ra\infty.
\]
By definition of $\xi_n$, the following holds:
\[
(y_n,\tilde \eta_n)\ra (y_T,\tilde \eta_T)  \text{ uniformly in } [0,T],\ \text{as}\ n\ra\infty,
\]
\[
(\dot y_n,\dot{\tilde \eta}_n)\ra (z,\tilde \tau)  \text{ weakly in } L^1(0,T;\R^{d+1}),\ \text{as}\ n\ra\infty,
\]
where
\[
\tilde \eta_T(s)=e^{\lambda s}\xi_T(s),\ \tilde \tau(s)=e^{\lambda s}(\tau(s)+\lambda \xi_T(s)).
\]
Recall that $(y_n(\cdot),\tilde \eta(\cdot))$ satisfies \eqref{dsAugmented} and $G$ is locally Lipschitz continuous with convex images. By \cite[Theorem 1, pp. 60]{AC84}, we deduce that
\[
(\dot y_n,\dot{\tilde \eta}_n)\ra (\dot y_T,\dot{\tilde \eta}_T)  \text{ weakly in } L^1(0,T;\R^{d+1}),\ \text{as}\ n\ra\infty,
\]
and $(y_T,\tilde \eta_T)$ satisfies \eqref{dsAugmented}. Moreover,
\[
w_T(x)=e^{-\lambda T}\tilde \eta_T(T),
\]
which implies that $\eta_T$ is a minimizer for \eqref{ocAugmented} with corresponding trajectory $y_T$.
\end{proof}

\begin{lem}\label{LEMocT}
For any $x\in\R^d$, we have $v_T(x)=w_T(x)$ and there exists a minimizer $\bar u$ for problem \eqref{ocT}.
\end{lem}
\begin{proof}
For any $(y(\cdot),u(\cdot))$ satisfying \eqref{ds}, we define
\[
\eta(s)=\int^s_0 e^{-\lambda(s'-s)}\ell(y(s'),u(s'))ds',\ \forall\,s\in [0,T].
\]
Then
\[
\dot \eta(s)=\lambda \eta(s)+\ell(y(s),u(s)),\ \text{a.e.}\ s\in (0,T),
\]
which implies that $(y(\cdot),\eta(\cdot))$ satisfies \eqref{dsAugmented}. It follows that
\begin{eqnarray*}
w_T(x)\leq e^{-\lambda T}\eta(T)=\int^T_0 e^{-\lambda s}\ell(y(s),u(s))ds,
\end{eqnarray*}
and thus,
\[
w_T(x)\leq v_T(x).
\]
On the other hand, let $(\bar y,\bar \eta)$ be optimal for the problem \eqref{ocAugmented} by Lemma \ref{ExistenceAugmented}. Then there exists $\bar u\in L^{\infty}(0,T;U)$ such that
\begin{eqnarray}\label{baru}
w_T(x)= e^{-\lambda T}\bar \eta(T)=\int^T_0 \frac{d}{ds}\left[e^{-\lambda s}\bar \eta(s)\right]ds\geq\int^T_0e^{-\lambda s}\ell(\bar y(s),\bar u(s))ds,
\end{eqnarray}
where the last inequality holds due to the definition of $G$. Thus,
\[
w_T(x)\geq v_T(x).
\]
It follows that $w_T(x)=v_T(x)$. Together with \eqref{baru}, we obtain that $\bar u$ is a minimizer for \eqref{ocT}.
\end{proof}

\begin{thm}\label{THMexistence}
Assume {\bf (H1)-(H3)}. Then there exists a minimizer $\bar u\in L^{\infty}(0,\infty;U)$ for problem \eqref{ocinfini}.
\end{thm}
For the sake of conciseness, we defer this technical proof to the Appendix.  

\section{The non-convex case: $0<p<1$}

In this section, we consider the particular case where $f$ is affine in $u$, i.e. there exists a Lipschitz continuous functions $f_k:\R^d\ra\R^d$ for $k=1,\ldots,m$ such that
 \begin{equation}\label{DSaffine}
 f(x,u)=f_0(x)+\sum^m_{k=1} f_k(x)u^k,\ \forall\,x\in\R^d,\ u=(u^1,\ldots,u^m)\in U.
 \end{equation}
Note that when $0<p<1$, the convexity assumption {\bf (H3)} is not satisfied by $f,\ell$. Therefore, the existence of an optimal control needs special attention.

It is well known that if there exists a minimizer for \eqref{ocinfini}, the dynamic programming approach will provide an optimal feedback control which is a measurable function in general. Since in numerical practice the HJB-based feedback is typically piecewise constant in time, the idea here is to consider a subspace of piecewise constant functions instead of considering the whole space of measurable functions on $(0,\infty)$. For this case existence can be derived.

For the time sequence:
\[
0=t_0<t_1<\ldots<t_i<t_{i+1}<\ldots<\infty,\ i\in\N,
\]
the set of piecewise constant controls $\U$ is defined by
\[
\U=\{u=(u^1,\ldots,u^m):[0,\infty)\ra U\,:\,u^k(s)=u^k_i\ \text{for}\ s\in[t_i,t_{i+1}),\ u^k_i\in U,\ i\in\N,\ k=1,\ldots,m\}.
\]
For any $u\in\U$,
\[
\int^{\infty}_0 e^{-\lambda s}\|u(s)\|^p_p ds=
\sum^m_{k=1}\sum^{\infty}_{i=0}\int^{t_{i+1}}_{t_i}e^{-\lambda s}|u^k_i|^p ds
=\sum^m_{k=1}\sum^{\infty}_{i=0}c_i|u^k_i|^p,
\]
where
\[
c_i=\frac{1}{\lambda} (e^{-\lambda t_i}-e^{-\lambda t_{i+1}}).
\]
Then, the optimization problem \eqref{ocinfini} over $\U$ can be expressed as for any $x\in\R^d$
\begin{equation}\label{ocdiscret}
\inf_{u\in\U}\int^{\infty}_0e^{-\lambda s}\ell_1(y(s))ds+ \gamma\sum^m_{k=1}\sum^{\infty}_{i=0}c_i|u^k_i|^p,
\end{equation}
subject to
\begin{equation}\label{dsDelta}
\left\{
\begin{array}{ll}
\dot y(s)=f_0(y(s))+\sum^m_{k=1}f_k(y(s))u^k_i & \text{for}\ s\in (t_i,t_{i+1}),\ i=0,1,\ldots\\
y(0)=x.
\end{array}
\right.
\end{equation}
We denote by $v^{\Delta}(x)$, $x\in\R^d$ the associated value function.
For any $p>0$ we define the space $\ell^p=\left\{(u_i)_{i\in\N}\,:\,\sum^{\infty}_{i=1}|u_i|^p<\infty\right\}$ endowed with
\[
|u|_{\ell^p}=\left(\sum^{\infty}_{i=1}|u_i|^p\right)^{1/p},\ \text{for}\ u\in\ell^p.
\]
It is a norm if $p\geq1$ and a quasi-norm if $0<p<1$. We recall the following result.
\begin{lem}\label{LEMembedding}
For $1\leq r<s\leq\infty$, $\ell^r\subseteq\ell^s$. 
\end{lem}
\begin{proof}
The case when $s=\infty$ is trivial. Consider $s\neq \infty$.
For any $(u_i)_{i\in\N}\in\ell^r$, we have that $|u_i|\ra 0$ as $i\ra \infty$. Then, there exists $i_0\in\N$ such that
\[
\forall\, i> i_0,\ |u_i|\leq 1.
\]
We set
\[
M:=\max\left\{1,\max_{i=1,\ldots,i_0}\{ |u_i|\}\right\}.
\]
Thus, for $s\in (r,\infty)$
\begin{eqnarray*}
\sum_{i=1}^{\infty}|u_i|^s = M^s\sum_{i=1}^{\infty} \left|\frac{u_i}{M}\right|^s\leq M^s\sum_{i=1}^{\infty} \left|\frac{u_i}{M}\right|^r =M^{s-r}\sum_{i=1}^{\infty} |u_i|^r,
\end{eqnarray*}
which concludes the proof.
%
\end{proof}
Due to the $\ell^p$-penalty in the distributed cost, problem \eqref{ocdiscret} turns into
\begin{equation}\label{ocellcp}
v^{\Delta}(x)=\inf_{u=(u^1,\ldots,u^m)\in \U,(c_i^{1/p}u^k_i)\in\ell^p,k=1,\ldots,m}\int^{\infty}_0e^{-\lambda s}\ell_1(y(s))ds+ \gamma\sum^m_{k=1}\sum^{\infty}_{i=0}\left|c_i^{1/p}u^k_i\right|^p.
\end{equation}

\smallskip

To establish an existence result for problem \eqref{ocellcp}, we follow the idea in \cite{IK14} by introducing the reparametrization $\psi:\ell^2\ra\ell^p$ with
\[
\psi(w)_i=|w_i|^{\frac{2}{p}}\mathop{sgn}(w_i),\ \text{for}\ w\in\ell^2,\ i\in\N.
\]
Using the fact that $\psi$ is an isomorphism, \eqref{ocellcp} is equivalent to
\begin{equation}\label{ocellc2}
v^{\Delta}(x)=\inf_{w=(w^1,\ldots,w^m),\psi(w)\in \U,w^k\in\ell^2,k=1,\ldots,m}\int^{\infty}_0e^{-\lambda s}\ell_1(y(s))ds+ \gamma\sum^m_{k=1}\sum^{\infty}_{i=0}|w^k_i|^2,
\end{equation}
where $(y,w)$ satisfies
\begin{equation}\label{dsDeltaw}
\left\{
\begin{array}{ll}
\dot y(s)=f_0(y(s))+\sum^m_{k=1}f_k(y(s))c_i^{-1/p}\psi(w^k)_i & \text{for}\ s\in (t_i,t_{i+1}),\ i=0,1,\ldots\\
y(0)=x.
\end{array}
\right.
\end{equation}
Let us recall \cite[Lemma 2.1]{IK14} as follows.
\begin{lem}\label{lemell2c}
The mapping $\psi:\ell^2\ra\ell^2$ is weakly (sequentially) continuous, i.e. $w^n\ra \bar w$ weakly in $\ell^2$ implies that $\psi(w^n)\ra\psi(\bar w)$ weakly in $\ell^2$.
\end{lem}
\begin{thm}\label{thmexistence1}
If {\bf (H1)-(H2)} hold, there exists a minimizer $\bar w\in\left(\ell^2\right)^m$ to \eqref{ocellc2}, and hence a solution $\bar u\in\U$ to \eqref{ocellcp}.
\end{thm}
\begin{proof}
Let $(w^{1,n},\ldots,w^{m,n})$ be a minimizing sequence of \eqref{ocellc2}. For $k=1,\ldots,m$, we set $u^{k,n}=(u^{k,n}_i)_{i\in\N}$ such that
\[
u_i^{k,n}=c_i^{-1/p}\psi(w^{k,n})_i,\ \forall\,i\in\N.
\]
Let $y^n$ be the solution of \eqref{ds} associated to $u^n=(u^{1,n},\ldots,u^{m,n})$. Note that for $k=1,\ldots,m$,
\[
\sum^{\infty}_{i=0}c_i^{2/p}|u^{k,n}_i|^2=\sum^{\infty}_{i=0}|w^{k,n}_i|^{4/p},
\]
and $2<\frac{4}{p}$. Since $\ell^2\subseteq\ell^{4/p}$, we deduce that 
\[
w^{k,n}\in \ell^{4/p}\ \text{and}\ (c_i^{1/p}u^{k,n}_i)_{i\in\N}\in\ell^2.
\] 
It follows that $\{((c_i^{1/p}u^{k,n}_i)_{i\in\N},w^{k,n})\}_{n\in\N}$ is a bounded sequence in $\ell^2\times\ell^2$. Hence there exists a subsequence such that $((c_i^{1/p}u^{k,n}_i)_{i\in\N},w^{k,n})$ converges weakly to some $((c_i^{1/p}\bar u^k_i)_{i\in\N},\bar w^k)\in\ell^2\times\ell^2$. From Lemma \ref{lemell2c} we have that 
\begin{equation}\label{Eqexist1}
\bar u^k_i=c_i^{-1/p}\psi(\bar w^k)_i,\ \forall\,i\in\N.
\end{equation}
Weak convergence of $w^{k,n}$ to $\bar w^k$ in $\ell^2$ implies that 
\[
w^{k,n}_i\ra \bar w^k_i\ \text{for each}\ i\in\N.
\]
Let $y_n$ be the solution to \eqref{dsDeltaw} with the control $(w^{1,n},\ldots,w^{m,n})$. Then on each interval $[t_i,t_{i+1}]$, $i=0,1,\ldots$, it is deduced by the same arguments as in Lemma \ref{LEMocT} that there exists $\bar y_i:[t_i,t_{i+1}]\ra \R^d$ such that 
\[
y_n \ra \bar y_i\ \text{uniformly in}\ [t_i,t_{i+1}],\ \text{as}\ n\ra\infty.
\]
For $\bar y:[0,\infty)\ra\R^d$ defined by
\[
\bar y(s)=\bar y_i(s),\ \text{for}\ s\in [t_i,t_{i+1}],\ i=0,1,\ldots,
\]
it follows that
\[
y_n \ra \bar y\ \text{uniformly in}\ [0,\infty),\ \text{as}\ n\ra\infty,
\]
and hence $\bar y$ is the solution to \eqref{dsDeltaw} corresponding to $\bar w:=(\bar w^1,\ldots,\bar w^m)$. Here we use that $f$ is affine in $\psi(w^k)$, $k=1,\ldots,m$. By convexity of $\ell_1$ and the lower semi-continuity of the $\ell^2$ norm, we deduce that $\bar w$ is a minimizer for the problem \eqref{ocellc2}. Hence $\bar u:=(\bar u^1,\ldots,\bar u^m)$ satisfying \eqref{Eqexist1} is a minimizer for the problem \eqref{ocellcp}.
\end{proof}

\section{Sparsity properties}
In this section, the control set $U$ is given by $\ell^q$-type constraints of the form
\begin{equation}\label{DefU}
U=\left\{u=(u_1,\ldots,u_m)\in\R^m\,:\,\sum^m_{i=1}|u_i|^q\leq \rho^q\right\},
\end{equation}
where $q\geq 1$ and $\rho>0$ are fixed. We focus on control-affine dynamics as in \eqref{DSaffine} which are recalled here
\[
f(x,u)=f_0(x)+\sum^m_{i=1} f_i(x)u_i,\ \forall\,x\in\R^d,\ u=(u_1,\ldots,u_m)\in U.
\]
Let us also recall the running cost: given $p\in (0,1]$,
\[
\ell(x,u)=\ell_1(x)+\gamma\|u\|^p_p,\ \forall\,x\in\R^d,\ u=(u_1,\ldots,u_m)\in U.
\]
In this framework, we investigate the sparsity properties for the optimal controls which can be derived from the optimality condition.  In general,
the first order necessary optimality conditions for the infinite horizon problem \eqref{ocinfini} are the following (\cite[Remark III.2.55]{BC97}). 
\begin{lem}\label{Optimality}
Assume that {\bf (H1)-(H2)} hold and suppose in addition that $f,\ell$ are $C^1$ with respect to the first variable. Given $x\in\R^d$, let $\bar u\in L^{\infty}(0,\infty;U)$ be a locally optimal control for problem \eqref{ocinfini} with the initial point $x$ and corresponding optimal trajectory $\bar y$. Then there exists an adjoint state $\vp:[0,\infty)\ra\R^d$ satisfying:
\begin{enumerate}[(i)]
 \item 
 For almost all $s\in(0,\infty)$, the following equality holds in the Carath\'eodory sense
 \begin{equation}\label{dsadjoint}
 \dot \vp(s)=-\left[\frac{\partial}{\partial y}f(\bar y(s),\bar u(s))\right]^T \vp(s)-e^{-\lambda s}\frac{\partial}{\partial y}\ell(\bar y(s),\bar u(s)).
 \end{equation}
 \item 
 \[
 \lim_{T\ra\infty}\vp(T)=0.
 \]
 \item
 For almost all $s\in(0,\infty)$ and all $u\in U$, 
 \[
 -f(\bar y(s),\bar u(s))\cdot \vp(s)-e^{-\lambda s}\ell(\bar y(s),\bar u(s))\geq -f(\bar y(s),u)\cdot \vp(s)-e^{-\lambda s}\ell(\bar y(s),u).
 \]
\end{enumerate}
\end{lem}

The optimality condition shows that for almost all $s>0$, a local optimal control $\bar u$ maximizes
\begin{eqnarray*}
&&-f(\bar y(s),u)\cdot \vp(s)-e^{-\lambda s} \ell(\bar y(s),u) \\
&=&\sum^m_{i=1}\left(-f_i(\bar y(s))\cdot \vp(s) u_i-\gamma e^{-\lambda s}|u_i|^p\right)-f_0(\bar y(s))\cdot \vp(s)-e^{-\lambda s}\ell_1(\bar y(s))\\
&=& \gamma e^{-\lambda s}\sum^m_{i=1}\left(c_i(s) u_i -|u_i|^p\right)-f_0(\bar y(s))\cdot \vp(s)-e^{-\lambda s}\ell_1(\bar y(s)),
\end{eqnarray*}
where
\begin{equation}\label{CiDef}
c_i(s)=-\frac{f_i(\bar y(s))\cdot \vp(s)}{\gamma}e^{\lambda s}
\end{equation}
will be of importance throughout this section. The cases $0<p<1$ and $p=1$ will be treated separately. 

Denote by $(e_i)_{i=1,\ldots,m}$ the Euclidean basis of $\R^m$. The first sparsity result that can be obtained for $0<p<1$ is the following.
\begin{prop}
Given $0<p<1$, $q\geq 1$ and $\rho>0$,
let $\bar u$ be a locally optimal control, $\bar y$ be the corresponding optimal trajectory and $\vp$ be the adjoint state. Then the following holds in the almost everywhere sense: if
\begin{equation}\label{conditionpq}
\frac{c_i(s)(q-1)}{p(q-p)}\rho^{1-p}<1,\ \forall\,i=1,\ldots,m,
\end{equation}
then
\[
\left\{
\begin{array}{ll}
\bar u(s)=0 & \text{if}\ \rho^{1-p}\max_{i=1,\ldots,m}|c_i(s)|<1,\\
\bar u(s)=\rho e_i\mathop{sgn} c_i(s) & \text{if}\ \rho^{1-p}|c_i(s)|>\max_{j\in\{1,\ldots,m,j\neq i\}}\{\rho^{1-p}|c_j(s)|,1\}.
\end{array}
\right.
\]
\end{prop}

\begin{proof} 
The arguments are carried out for an arbitrary $s\in (0,\infty)$, and to simplify the notation the dependence of $c_i$ on $s$ is not indicated. 

By Lemma \ref{Optimality}, $\bar u$ maximizes the following function
\[
g(u):=\sum^m_{i=1}\left(c_i u_i -|u_i|^p\right),\ \text{for}\ u=(u_1,\ldots,u_m)\in\R^m.
\]
At first consider the case when $c_i\geq 0$ for $i=1,\ldots,m$. For any $u=(u_1,\ldots,u_m)\in\R^m$, if there exists some $i\in\{1,\ldots,m\}$ such that $u_i<0$, then due to the fact that $c_i\geq 0$ we have
\[
g(u)<g(\tilde u)\ \text{with}\ \tilde u=(u_1,\ldots,u_{i-1},-u_i,u_{i+1}\ldots,u_m),
\]
and consequently
\[
\bar u\in \left\{u\in\R^m\,:\,u_i\geq 0,\ \sum^m_{i=1}u_i^q\leq \rho^q,\ i=1,\ldots,m\right\}.
\]
For $i=1,\ldots,m$, we set
\[
w_i=u_i^q\ \text{and}\ h(w)=\sum^m_{i=1}\left(c_iw_i^{1/q}-w_i^{p/q}\right),\ \text{for}\ w=(w_1,\ldots,w_m),\ w_i\geq 0.
\]
Then the maximization of $g$ is transformed to
\[
\max\{h(w)\,:\,w\in\Omega\},
\]
where $\Omega$ is the polygon defined as
\[
\Omega=\left\{w\in\R^m\,:\,w_i\geq 0,\ \sum^m_{i=1}w_i\leq \rho^q,\ i=1,\ldots,m\right\}.
\]
For any $w\in\Omega\backslash\partial\Omega$, we have
\[
\frac{\partial^2 h}{\partial w_i^2}=\frac{p(q-p)}{q^2}\cdot\frac{1}{w_i^{2-p/q}}\left(1-\frac{c_i(q-1)}{p(q-p)}w_i^{(1-p)/q}\right)\ \text{and}\ \frac{\partial^2h}{\partial w_i\partial w_j}=0.
\]
The constraints satisfied by $w$ imply that
\[
\frac{\partial^2 h}{\partial w_i^2}\geq\frac{p(q-p)}{q^2}\cdot\frac{1}{\rho^{2q-p}}\left(1-\frac{c_i(q-1)}{p(q-p)}\rho^{1-p}\right)>0.
\]
Thus, $h$ is strongly convex in $\Omega\backslash\partial\Omega$. It is also strongly convex in $\Omega$ since $h$ is continuous. Thus the maximum of $h$ is obtained at the vertices of the polygon $\Omega$. By comparing the values of $h$ at each vertex, we get the following properties for the maximum $\bar w\in\Omega$
\[
\left\{
\begin{array}{ll}
\bar w=0 & \text{if}\ \rho^{1-p}\max_{i=1,\ldots,m}c_i<1,\\
\bar w=\rho^q e_i & \text{if}\ \rho^{1-p}c_i>\max_{j\in\{1,\ldots,m,j\neq i\}}\{\rho^{1-p}c_j,1\}.
\end{array}
\right.
\]
The desired result for $c_i\geq 0$, for $i=1,\ldots,m$, is obtained by replacing $\bar w_i$ by $\bar u_i^{q}$. The other cases when $c_i$ have different signs can be treated analogously.
\end{proof}
\begin{rem}
Condition \eqref{conditionpq} is always satisfied for $q=1$.
\end{rem}
We proceed to give the next sparsity result for the case with $\ell^ \infty$-constraints. For any $a,b\in\R$, the following notation is used:
\[
\llbracket a,b \rrbracket:=[\min\{a,b\},\max\{a,b\}].
\]
\begin{prop}\label{constraintsinf}
Given $0<p\leq 1$ and $q=\infty$, assume that $\bar u$ is a local optimal control for problem \eqref{ocinfini} with
\[
U=\prod^m_{i=1}[-\rho_i,\rho_i]
\]
where $\rho_i\geq 0$. Let $\bar y$ be the corresponding optimal trajectory and $\vp$ be the adjoint state. Then the following holds: for almost all $s\in (0,\infty)$
\[
\left\{
\begin{array}{llll}
\bar u_i(s)=0 & \text{if}\ \rho^{1-p}|c_i(s)|< 1,\\
\bar u_i(s)=-\rho_i\mathop{sgn} c_i(s) & \text{if}\ \rho^{1-p}|c_i(s)|>1,\\
\bar u_i(s)\in\{0,-\rho_i\mathop{sgn} c_i(s) & \text{if}\ \rho^{1-p}|c_i(s)|= 1,\ p\neq 1,\\
\bar u_i(s)\in\llbracket 0,-\rho_i\mathop{sgn} c_i(s)\rrbracket & \text{if}\ \rho^{1-p}|c_i(s)|= 1,\ p= 1.
\end{array}
\right.
\]
\end{prop}

\begin{proof}
Again the dependence of $c_i$ on $s$ is not indicated to simply the notation.

For $i=1,\ldots,m$, let $g_i:\R\ra\R$ be defined by
\[
g_i(v)=c_i v-|v|^p,\ \text{for}\ v\in\R.
\]
It is trivial to see that 
\[
\bar u_i(s)\in \mathop{\arg\,\max}_{v\in [-\rho_i,\rho_i]}\,g_i(v).
\]
At first consider the case with $c_i\leq 0$. Then
\[
g_i(v)< 0=g_i(0),\ \text{for}\ v>0,
\]
and hence $\bar u_i\in (-\infty,0]$.
For $v< 0$, we have
\[
g_i'(v)=c_i+ p(-v)^{p-1}.
\]
Let $v^*<0$ be the solution of $g_i'(v)=0$. We obtain that $g_i'(v)>0$ for $v\in (v^*,0)$. Thus the maximum of $g_i$ on $[-\rho_i,0]$ is either obtained at $v=-\rho_i$ or $v=0$. Due to the fact that
\[
g_i(-\rho_i)=-c_i\rho_i - \rho_i^p< 0 \Leftrightarrow -c_i< \rho_i^{p-1},
\]
it is deduced that
\[
\left\{
\begin{array}{ll}
\bar u_i(s)=0 & \text{if}\ 0\leq -c_i< \rho_i^{p-1},\\
\bar u_i(s)=-\rho_i & \text{if}\ -c_i>\rho_i^{p-1}.
\end{array}
\right.
\]
Besides, if $-c_i=\rho_i^{p-1}$ and $p\neq 1$, then $g_i'(v)=0$ has a unique solution $v^*$ and $g(-\rho_i)=g(0)$. Thus,
\[
\bar u_i(s)\in\{-\rho_i,0\}.
\]
Otherwise, for $p=1$ and $-c_i=\rho_i^{p-1}=1$, we have $g_i(v)=0$ for $v\in [-\rho_i,0]$.  Thus,
\[
\bar u_i(s)\in [-\rho_i,0].
\]
The case $c_i\geq 0$ can be treated analogously. The desired result is obtained by combining the two cases.
\end{proof}

The last sparsity result concerns the case when $p=1$ and $q\in [1,\infty)$.
\begin{prop}\label{constraintsLq}
Given $p=1$, $1\leq q<\infty$ and $\rho>0$, let  $\bar u$ be a locally optimal control, $\bar y$ be the corresponding optimal trajectory and $\vp$ be the adjoint state. For $s>0$, we set
\[
I(s):=\{i\,:\,|c_i(s)|>1,\ i=1,\ldots,m\}.
\]
If $q=1$, then the following holds for $\bar u$ in the almost everywhere sense
\[
\left\{
\begin{array}{ll}
 \bar u(s)\in\overline{\mathop{co}}\,\{0,\ \rho e_i\mathop{sgn} c_i(s),\ \text{for}\ |c_i(s)|=1\} & \text{if}\ I(s)=\emptyset,\\
 \bar u(s)\in\overline{\mathop{co}}\,\{\rho e_i\mathop{sgn} c_i(s),\ \text{for}\ |c_i(s)|=\max_{j=1,\ldots,m}|c_j(s)|\} & \text{if}\ I(s)\neq \emptyset.
\end{array}
\right.
\]
If $q>1$, then the following holds for $\bar u$ in the almost everywhere sense
\begin{equation}\label{sparseLq1p}
\left\{
\begin{array}{lll}
\bar u(s)\in\overline{\mathop{co}}\,\{0,\ \rho e_i\mathop{sgn} c_i(s),\ \text{for}\ |c_i(s)|=1\} & \text{if}\ I(s)=\emptyset,\\
\bar u_i(s)=0 & \text{if}\ I(s)\neq\emptyset\ \text{and}\ |c_i(s)|\leq 1,\\
\bar u_i(s)=\rho\mathop{sgn}\ c_i(s)\frac{(|c_i(s)|-1)^{q'-1}}{\left(\sum_{i:|c_i(s)|>1}(|c_i(s)|-1)^{q'}\right)^{1/q}} & \text{if}\ I(s)\neq\emptyset\ \text{and}\ |c_i(s)|>1.
\end{array}
\right.
\end{equation}
\end{prop}
\begin{proof}
For $p=1$, $\bar u$ maximizes the function
\[
g(u):=\sum^m_{i=1}\left(c_iu_i-|u_i|\right).
\]
We start by proving the results with $q=1$. Consider first the domain
\[
\Omega_1:=\{u\in\R^m\,:\,u_i\geq 0,\ \sum^m_{i=1} u_i\leq \rho,\ i=1,\ldots,m\}.
\]
In $\Omega_1$,
\[
g(u)=\sum^m_{i=1}\left(c_i-1\right)u_i.
\]
Note that $g$ is linear in $\Omega_1$ which is a polyhedron. Hence the maximizer of $g$ in $\Omega_1$ is either a vertex of the polyhedron $\Omega_1$ or a convex combination of some vertices. The value of $g$ at the vertices are the following
\[
g(0)=0\ \text{and}\ g(\rho e_i)=\rho(c_i-1),\ \text{for}\ i=1,\ldots,m.
\]
It is then deduced that
\[
\left\{
\begin{array}{ll}
 \bar u(s)\in\overline{\mathop{co}}\,\{0,\ \rho e_i,\ \text{for}\ c_i=1\} & \text{if}\ \max_{i=1,\ldots,m} c_i\leq 1,\\
 \bar u(s)\in\overline{\mathop{co}}\,\{\rho e_i,\ \text{for}\ c_i=\max_{j=1,\ldots,m} c_j\} & \text{if}\ \max_{i=1,\ldots,m} c_i>1.
\end{array}
\right.
\]
The domain $U$ defined by \eqref{DefU} can be divided into $2^m$ polyhedra of the form $\Omega_1$ but with different signs of the $u_i$, and in each polyhedron analogous results on maximizers can be obtained by the arguments as in $\Omega_1$. The desired result concerning the maximizers in $U$ is then obtained.

\smallskip

We proceed to the case with $q>1$. Analogous to the previous case, let us consider first the maximization of $g$ in
\[
\Omega_2:=\{u\in\R^m\,:\,u_i\geq 0,\ \sum^m_{i=1} u_i^q\leq \rho^q,\ i=1,\ldots,m\}.
\]
We set the index set $I$ and $q'>1$ as follows
\[
I:=\{i\,:\, c_i>1,\ i=1,\ldots,m\},\ \text{and}\ \frac{1}{q}+\frac{1}{q'}=1.
\]

If $I=\emptyset$, we deduce the same result as in the previous case with $q=1$.

\smallskip

If $I\neq\emptyset$, by H\"older's inequality
\[
g(u)\leq \sum_{i\in I}(c_i-1)u_i
\leq \left(\sum_{i\in I}(c_i-1)^{q'}\right)^{1/q'}\left(\sum_{i\in I}u_i^q\right)^{1/q}
\leq \rho\left(\sum_{i\in I}(c_i-1)^{q'}\right)^{1/q'},
\]
where the equality holds when 
\[
\frac{u_i^q}{(c_i-1)^{q'}}\ \text{is constant for}\ i\in I, \ u_j=0\ \text{for}\ j\not\in I,\ \text{and}\ \sum_{i\in I}u_i^q=\rho^q.
\]
Direct computations show that
\[
u_i=\rho\frac{(c_i-1)^{q'-1}}{\left(\sum_{i\in I}(c_i-1)^{q'}\right)^{1/q}}\ \text{for}\ i\in I.
\]
Again the domain $U$ defined by \eqref{DefU} can be divided into $2^m$ parts with the same structure as $\Omega_2$, and in each part analogous results on maximizers can be obtained. This concludes the proof.
\end{proof}

\section{The Eikonal case with $L^1$-cost}
In order to get additional insight into the structure of the optimal controls with $\|u\|^p_p$ cost, the problem with Eikonal dynamics is analyzed in this section. The Eikonal dynamics system is the following: for $x\in\R^d$
\begin{equation}\label{Eikonal}
\left\{
\begin{array}{ll}
 \dot y(s)=u(s) & \text{for}\ s\in (0,\infty),\\
 y(0)=x,
\end{array}
\right.
\end{equation}
where $u$ takes value in
\[
U:=\{u=(u_1,\ldots,u_d)\in\R^d\,:\,\sum^d_{i=1}u_i^2\leq\rho^2\}\ \text{for some}\ \rho>0.
\]
The running cost is
\[
\ell(x,u)=\frac{1}{2}\|x\|^2_2+\gamma\sum^d_{i=1}|u_i|,
\]
where $(x,u)\in\R^d\times U$. Note that the dynamical system is linear and the cost functional is strictly convex in $x$ and convex in $u$, consequently the optimal state is unique, and as a consequence of \eqref{Eikonal} the optimal control is unique as well (in the almost everywhere sense).

Let $\bar u$ be the optimal control and $\bar y$ be the corresponding optimal trajectory. By Lemma \ref{Optimality} there exists $\vp:[0,\infty)\ra\R$ satisfying the adjoint state equation in the Carath\'eodory sense
\[
\left\{
\begin{array}{ll}
\dot \vp(s)=-e^{-\lambda s}\bar y(s) & \text{for}\ s\in (0,\infty),\\
\lim_{s\ra\infty}\vp(s)=0.
\end{array}
\right.
\]
Proposition \ref{constraintsLq} implies that
\[
\mathop{supp}(\bar u)\subset \{s\in [0,\infty)\,:\,|c_i(s)|\geq 1\},
\]
where 
\[
c_i(s)=-\frac{1}{\gamma}e^{\lambda s}\vp_i(s),\ \text{for}\ s\in [0,\infty).
\]
The optimal control has the following property.
\begin{lem}\label{usupport}
 The support of the optimal control $\bar u$ is bounded. 
\end{lem}
\begin{proof}
 At first we consider the case $x=(x_1,\ldots,x_d)$ with $x_i> 0$ for $i=1,\ldots,d$. The proof is given in several steps.
 
 {\bf Step 1}: The optimal state $\bar y$ is nonnegative.\\
If this were not the true, then there exists an interval $(t_1,t_2)$ with $0<t_1<t_2\leq \infty$ and a component of the state, which without loss of generality we assume to be the first one, such that $\bar y_1(t)<0$ on $(t_1,t_2)$, $\bar y_1(t_1)=0$, and $\bar y_1(t_2)=0$ if $t_2<\infty$. Let us define a new control $\tilde u$ such that
 \[
 \tilde u_1=\left\{
 \begin{array}{ll}
  0 & \text{on}\ (t_1,t_2),\\
  \bar u_1 & \text{otherwise},
 \end{array}
 \right.
 \ \text{and}\ \tilde u_i=\bar u_i\ \text{for}\ i=2,\ldots,d.
 \]
 Let $\tilde y$ be the associated trajectory. We note that $\tilde u$ is feasible, $\tilde y_1(t)=0$ on $(t_1,t_2)$, and $\tilde y_i=\bar y_i$ for $i=2,\ldots,d$. Therefore,
 \[
 J(x,\tilde u)<J(x,\bar u),
 \]
 which is a contradiction to the optimality of $\bar u$.
 
 \smallskip
 
 {\bf Step 2}: $\bar y_i$ and $|c_i|$ are monotonically decreasing for each $i=1,\ldots,d$.\\
By the adjoint state equation we have 
 \[
 \vp(t)=\int^\infty_t e^{-\lambda s}\bar y(s)ds.
 \]
 Hence, $\vp_i$ is nonnegative and monotonically decreasing for $i=1,\ldots,d$. Therefore, $c_i$ is nonpositive and we deduce from Proposition \ref{constraintsLq} that $\bar u_i\leq 0$. This implies that $y_i$ is monotonically decreasing.
 Moreover, integration by parts yields that
 \[
 |c_i(t)|=-c_i(t)=\frac{1}{\gamma}e^{\lambda t}\int^\infty_t e^{-\lambda s}\bar y_i(s)ds =\frac{1}{\lambda\gamma}\left(\bar y_i(t)+\int^\infty_t e^{\lambda (t-s)}\bar u_i(s)ds\right),
 \]
 where we use that $t\mapsto \bar y(t)$ has sublinear growth (note that $\bar u$ is bounded) and hence $\lim_{t\ra\infty}e^{-\lambda t}\bar y(t)=0$. We further have
 \begin{equation}\label{ciprime}
 -c'_i(t)=\frac{1}{\gamma}\int^\infty_t e^{\lambda (t-s)}\bar u_i(s)ds\leq 0,
 \end{equation}
 which implies that $|c_i|$ is monotonically decreasing. In particular this implies that if $I(\bar t)=\emptyset$ for some $\bar t\geq 0$, then $I(t)=\emptyset$ for all $t\geq \bar t$.
 
 \smallskip
 
 {\bf Step 3}: Let us assume that $I(t)\neq \emptyset$ for all $t\geq 0$ and that $|c_i(t)|>1$ for all $t\geq 0$ and $i=1,\ldots,d$. Let $\bar c_i:=\lim_{t\ra\infty}|c_i(t)|\geq 1$, and denote
 \[
 \beta_i(t)=\frac{|c_i(t)|-1}{\sqrt{\sum^d_{j=1} \left(|c_j(t)|-1\right)^2}}\geq 0.
 \]
 Since $\lim_{t\ra\infty}\sum^d_{j=1}\beta_i(t)^2=1$ there exists $\bar i\in\{1,\ldots,d\}$ and $\bar \beta$ such that
 \[
 \lim_{t\ra\infty} \beta_{\bar i}(t)\ra \bar \beta>0.
 \]
 Hence there exists $\bar t>0$ such that $\beta_{\bar i}(t)>\frac{\bar \beta}{2}$ for all $t\geq \bar t$, and thus
 \[
 u_{\bar i}(t)\leq -\frac{\rho\bar \beta}{2}\ \text{for all}\ t\geq \bar t.
 \]
 This implies that $\lim_{t\in\infty} \bar y_{\bar i}(t)=-\infty$ which contradicts that $y_i$ is nonnegative.
 
 \smallskip
 
 {\bf Step 4}: Next we consider the case that $I(t)\neq \emptyset$ for all $t\geq 0$, then there exists some coordinate $\hat i$ and a $\bar t_{\hat i}>0$ such that $|c_{\hat i}(\bar t_{\hat i})|\leq 1$. Next let us choose all coordinates with this property and assume that these are the first $k$. Note that $k<d$ since otherwise $I(t)$ is not different from empty set for all $t\geq 0$. Thus,  
 \[
 |c_i(\bar t_i)|\leq 1\ \text{for }\ i\in\{1,\ldots,k\}.
 \]
 Since $t\mapsto |c_i(t)|$ is monotonically decreasing we have that $|c_i|\leq 1$ for all $t\geq \bar t:=\max\{\bar t_i\}^k_{i=1}$ and $i\in\{1,\ldots,k\}$. Thus by \eqref{sparseLq1p} we have
 \[
 \bar u_i(t)=0\ \text{for all}\ t\geq\bar t\ \text{and}\ i=1,\ldots,k.
 \]
 Turning to the remaining coordinates, by the same argument as in the previous step, it then follows that there exists $\bar i\in\{k+1,\ldots,d\}$, $t_{\bar i}$ and $\bar \beta>0$ such that
 \[
 \bar u_{\bar i}(t)\leq -\frac{\rho\bar \beta}{2}\ \text{for all}\ t\geq t_{\bar i}.
 \]
 Thus $\lim_{t\ra\infty} \bar y_{\bar i}(t)=-\infty$ which is a contradiction.
 
 \smallskip
 
 {\bf Step 5}: The only remaining possibility is that there exists $\tilde t$ such that $I(t)=\emptyset$ for all $t\geq \tilde t$. By \eqref{ciprime} and monotone decay of $t \mapsto |c_i(t)|$ we have $u(t)\equiv 0$ on $(\tilde t,\infty)$.
 This concludes that $\bar u$ has a bounded support for $x_i>0$, $i=1,\ldots,d$.
 
 \smallskip
 
 We proceed to prove that $\bar u$ has a bounded support for any $x\in\R^d$.
 
 {\bf Step 6}: If $x_1=0$ and $x_i>0$ for $i=2,\ldots,d$, then $\bar u_1=\bar y_1=0$ by optimality. We can use the above arguments to show that $\bar u_i(t)=0$ for all $t$ sufficiently large and $i=2,\ldots,d$.
 
 \smallskip
 
 {\bf Step 7}: The remaining cases for the signs of the initial conditions now easily follow.
 
\end{proof}

Now we investigate the behavior of the optimal trajectory.
\begin{lem}\label{ylambdagamma}
 For $i=1,\ldots,d$, the following holds.
 \begin{enumerate}[(i)]
  \item If $|x_i|\leq \lambda\gamma$, then $\bar y_i\equiv x_i$ on $[0,\infty)$.
  \item If $|x_i|>\lambda\gamma$, then there exists $T_i>0$ such that $\bar y_i(t)=\mathop{sgn}(x_i)\lambda\gamma$ for all $t\geq T_i$.
 \end{enumerate}
\end{lem}

\begin{proof}
 The proof is given by several steps.
 
 {\bf Step 1}: If $x_i=0$, then $y_i\equiv 0$ is optimal. 
 
 \smallskip
 
 {\bf Step 2}: Consider the case $0<x_i\leq\lambda\gamma$. Arguing by contradiction, if there exists some $\bar t>0$ with $\bar y_i(\bar t)< x_i$, then there exists some $t_\e\in (0,\bar t)$ such that
 \[
 \bar y_i(\bar t)<\bar y_i(t_\e)=x_i-\e,\ \text{for some}\ \e\in(0,x_i-\bar y_i).
 \]
 For any $t\geq t_\e$, $\bar y(t)\leq \bar y_i(t_\e)$ since $\bar y_i$ is monotonically decreasing. Thus,
 \[
 |c_i(t)|\leq \frac{1}{\gamma} e^{\lambda t}\int^{\infty}_t e^{-\lambda s}(x_i-\e)ds\leq \frac{1}{\lambda\gamma}(\lambda\gamma-\e)<1,\ \text{for}\ t\geq t_\e,
 \]
 which by Proposition \ref{constraintsLq} yields that $\bar u_i(t)=0$ for $t\geq t_\e$. Therefore $\bar y_i(\bar t)=\bar y_i(t_\e)$, which is a contradiction.
 
 \smallskip
 
 {\bf Step 3}: Consider the case $x_i>\lambda\gamma$. Since $\bar y_i$ is nonnegative and monotonically decreasing, there exists $z\geq 0$ such that
 \begin{equation}\label{limyi}
 \lim_{t\ra\infty}\bar y_i(t)=z.
 \end{equation}
 Arguing by contradiction, if $z>\lambda\gamma$, then $\bar y_i(t)\geq z>\lambda\gamma$ for any $t\geq 0$. Consequently,
 \[
 \lim_{t\ra\infty}|c_i(t)|=\lim_{t\ra\infty}\frac{1}{\gamma} e^{\lambda t}\int^{\infty}_t e^{-\lambda s}\bar y_i(s) ds=\frac{z}{\lambda\gamma}>1.
 \]
 Using Proposition \ref{constraintsLq}, there exists $\bar t>0$ and $\bar \beta>0$ such that
 \[
 u_i(t)\leq -\bar \beta,\ \text{for all}\ t\geq\bar t.
 \]
 This implies that $\lim_{t\ra\infty} y_i(t)=-\infty$, which contradicts that $y_i(t)\geq 0$ for all $t\geq 0$.
 
 By a similar argument we can exclude the case that $z<\lambda\gamma$. Therefore, we conclude that $\lim_{t\ra\infty} \bar y_i(t)=\lambda\gamma$. Since $u_i$ has a bounded support, there exists $T_i>0$ such that $\bar y_i(t)=\lambda\gamma$ for all $t\geq T_i$.
 
 \smallskip
 
 {\bf Step 4}: The remaining cases for $x_i<0$ now easily follow.
\end{proof}

\begin{rem}\label{inftofini}
Lemma \ref{usupport} and Lemma \ref{ylambdagamma} imply that the original infinite horizon problem can be reduced to a finite horizon problem for the Eikonal case. In fact, there exists $T>0$ sufficiently large such that $\bar u(t)=0$ and $\bar y(t)=\lambda\gamma$ for $t\geq T$. Consider the following finite horizon problem:
\begin{equation}\label{finiteoc}
 \inf_{u\in L^{\infty}(0,T;U)}\left\{\int^T_0 e^{-\lambda s}\ell(y(s),u(s))ds,\ \dot y(s)=u(s)\ \text{in}\ (0,T),\ y(0)=x,\ y(T)=\lambda\gamma\right\}.
\end{equation}
This problem has a unique optimal solution. Since $(\bar y,\bar u,\vp)$ satisfies the optimality conditions in Lemma \ref{Optimality}, $(\bar y,\bar u,\vp)|_{[0,T]}$ satisfies Pontryagin's maximum principle for the problem \eqref{finiteoc} as well. By \cite[Corollary, pp.220]{C90}, $(\bar y,\bar u)|_{[0,T]}$ is the optimal solution of \eqref{finiteoc}. 
\end{rem}

Now let us construct the optimal control $\bar u$ precisely. We start by the $1$d case.
\begin{thm}\label{THMEikonal}
 For any $x\in\R$, the following holds:
 \begin{itemize}
  \item if $|x|\leq \lambda\gamma$, then $\bar u\equiv 0$ is the optimal control.
  \item if $|x|>\lambda\gamma$, let $\tau=\frac{|x|-\lambda\gamma}{\rho}$. Then
  \[
   \bar u(t)=\left\{
\begin{array}{ll}
-\rho\mathop{sgn}(x) & \text{for}\ t\in [0,\tau),\\
0 & \text{for}\ t\in[\tau,\infty),
\end{array}
\right.
\]
is the optimal control.
 \end{itemize}
\end{thm}

\begin{proof}
Assume without loss of generality that $x>0$. 
The case $x\leq \lambda\gamma$ has already been discussed in Lemma \ref{ylambdagamma}.

\smallskip

If $x> \lambda\gamma$, we set $\tau=\frac{x-\lambda\gamma}{\rho}$, and note that $\tau>0$. Define
\[
\bar y(t)=\left\{
\begin{array}{ll}
x-\rho t & \text{for}\ t\in [0,\tau),\\
x-\rho \tau & \text{for}\ t\in[\tau,\infty),
\end{array}
\right.
\bar u(t)=\left\{
\begin{array}{ll}
-\rho & \text{for}\ t\in [0,\tau),\\
0 & \text{for}\ t\in[\tau,\infty).
\end{array}
\right.
\]
By direction computation the adjoint state $\bar \vp$ is the following:
\[
\bar \vp(t)=\left\{
\begin{array}{ll}
\frac{1}{\lambda}(x-\rho t)e^{-\lambda t}+\frac{\rho}{\lambda^2}(e^{-\lambda \tau}-e^{-\lambda t}) & \text{for}\ t\in [0,\tau),\\
\frac{1}{\lambda}(x-\rho \tau)e^{-\lambda t} & \text{for}\ t\in[\tau,\infty).
\end{array}
\right.
\]
We claim that $(\bar y,\bar u,\bar \vp)$ satisfy the optimality conditions.

In fact, for $t\in(\tau,\infty)$,
\[
\bar \vp(t)=\frac{1}{\lambda}\lambda\gamma e^{-\lambda t}=\gamma e^{-\lambda t},
\]
and for $t\in (0,\tau)$,
\[
\frac{d[\bar \vp(t)-\gamma e^{-\lambda t}]}{dt}=(\rho t-x+\lambda\gamma)e^{-\lambda t}<(\rho\tau -x+\lambda\gamma)e^{-\lambda t}= 0.
\]
Thus,
\[
\left\{
\begin{array}{ll}
\bar \vp(t)>\gamma e^{-\lambda t} & \text{for}\ t\in[0,\tau),\\
\bar \vp(t)=\gamma e^{-\lambda t} & \text{for}\ t\in[\tau,\infty).
\end{array}
\right.
\]
Hence, the optimality conditions in Proposition \ref{constraintsLq} and Lemma \ref{Optimality} are satisfied by $(\bar y,\bar u,\bar \vp)$. Therefore, $(\bar y,\bar u)|_{[0,T]}$ is the optimal solution of the problem \ref{finiteoc} for a sufficient large $T>0$ by Remark \ref{inftofini}. If $u^*$ is the optimal control, then it is also optimal for \ref{finiteoc} by Remark \ref{inftofini}. Thus $\bar u=u^*$ on $[0,T]$. Note that $\bar u(t)=u^*(t)=0$ for $t\geq T$, consequently $\bar u=u^*$ is the optimal control.
\end{proof}

We proceed to give the optimal control in the $2$d case.
\begin{thm}
For any $x\in\R^2$, the following holds:
 \begin{itemize}
  \item if $|x_1|,|x_2|\leq \lambda\gamma$, then $\bar u\equiv 0$ is the optimal control.
  \item if $|x_1|>\lambda\gamma$ and $|x_2|\leq\lambda\gamma$, let $\tau=\frac{|x_1|-\lambda\gamma}{\rho}$. Then
  \[
   \bar u_1(t)=\left\{
\begin{array}{ll}
-\rho\mathop{sgn}(x) & \text{for}\ t\in [0,\tau),\\
0 & \text{for}\ t\in[\tau,\infty),
\end{array}
\right.\qquad
\bar u_2(t)\equiv 0\,,
\]
is the optimal control.
\item
if $|x_2|\geq |x_1|>\lambda\gamma$, let
\[
r=\|x\|_2,\ \tau_1=\frac{r}{\rho}\left(1-\frac{\lambda\gamma}{|x_1|}\right),\ \text{and}\ \tau_2=\frac{\lambda\gamma}{\rho}\left(\frac{|x_2|}{|x_1|}-1\right).
\]
Then
\[
\bar u_1(t)=\left\{
\begin{array}{ll}
-\rho\frac{x_1}{r} & \text{for}\ t\in [0,\tau_1),\\
0 & \text{for}\ t\in[\tau_1,\infty),
\end{array}
\right.
\]
\[
\bar u_2(t)=\left\{
\begin{array}{lll}
-\rho\frac{x_2}{r} & \text{for}\ t\in [0,\tau_1),\\
-\rho\mathop{sgn}(x_2) & \text{for}\ t\in [\tau_1,\tau_1+\tau_2],\\
0 & \text{for}\ t\in[\tau_1+\tau_2,\infty)
\end{array}
\right.
\]
is the optimal control.
 \end{itemize}
 In the remaining regions the optimal control is defined by symmetry.
\end{thm}

\begin{proof}
The first two cases can be seen as extensions of the result in $1$d case. Consider the case with $|x_2|\geq |x_1|>\lambda\gamma$, and assume without loss of generality that $x_2\geq x_1>0$. Let $(\bar y_1,\bar y_2)$ be the trajectory corresponding to $(\bar u_1,\bar u_2)$, then
\[
\bar y_1(t)=\left\{
\begin{array}{ll}
\left(1-\frac{\rho}{r}t\right)x_1 & \text{for}\ t\in [0,\tau_1),\\
\lambda\gamma & \text{for}\ t\in[\tau_1,\infty),
\end{array}
\right.
\]
\[
\bar y_2(t)=\left\{
\begin{array}{lll}
\left(1-\frac{\rho}{r}t\right)x_2 & \text{for}\ t\in [0,\tau_1),\\
\lambda\gamma\frac{x_2}{x_1}-\rho t & \text{for}\ t\in[\tau_1,\tau_1+\tau_2],\\
\lambda\gamma & \text{for}\ t\in[\tau_1+\tau_2,\infty).
\end{array}
\right.
\]
We check the optimality of $(\bar y_1,\bar u_1)$. The adjoint state is given by
\[
\bar \vp_1(t)=\left\{
\begin{array}{ll}
\frac{x_1}{\lambda}\left(1-\frac{\rho}{r}t\right)e^{-\lambda t}+\frac{\rho x_1}{\lambda^2 r}(e^{-\lambda \tau_1}-e^{-\lambda t}) & \text{for}\ t\in[0,\tau_1),\\
\frac{x_1}{\lambda}\left(1-\frac{\rho}{r}\tau_1\right)e^{-\lambda t} & \text{for}\ t\in[\tau_1,\infty),
\end{array}
\right.
\]
and for $t>\tau_1$, we have
\[
\bar \vp_1(t)=\frac{x_1}{\lambda}\left(1-\frac{\rho}{r}\tau_1\right)e^{-\lambda t} =\gamma e^{-\lambda t}.
\]
For $t\in [0,\tau_1)$,
\[
\frac{d[\bar \vp_1(t)-\gamma e^{-\lambda t}]}{dt}=\left(\lambda\gamma-x_1+\frac{\rho x_1}{r} t\right)e^{-\lambda t} < \left(\lambda\gamma-x_1+\frac{\rho x_1}{r} \tau_1\right)e^{-\lambda t}=0.
\]
Therefore,
\[
\left\{
\begin{array}{ll}
\bar \vp_1(t)>\gamma e^{-\lambda t} & \text{for}\ t\in[0,\tau_1),\\
\bar \vp_1(t)=\gamma e^{-\lambda t} & \text{for}\ t\in[\tau_1,\infty).
\end{array}
\right.
\]
It is then deduced that $(\bar y_1,\bar u_1,\bar \vp_1)$ satisfies Proposition \ref{constraintsLq}, and the optimality conditions in Lemma \ref{Optimality} consequently. The optimality of $(\bar y_2,\bar u_2)$ can be checked analogously. Finally by the same arguments as in the $1$d case, $(\bar y,\bar u)$ is the optimal control.
\end{proof}

For the general case, the optimal control is given in the following corollary as an extension of the $2$d case. The optimality can be checked analogously as in the $2$d case and we skip the proof.
\begin{cor}
For any $x\in\R^d$, assume without loss of generality that 
\[
\lambda\gamma<|x_1|\leq|x_2|\leq\cdots\leq|x_n|,\ \text{and}\ |x_i|\leq\lambda\gamma\quad \text{for}\quad i=n+1,\ldots,d.
\]
We set $x_0=\lambda\gamma$, $t_0=0$ and for $k=1,\ldots,n$
\[
r_k=\frac{\lambda\gamma}{x_{k-1}}\sqrt{\sum^n_{j=k}x_j^2},\ \tau_k=\frac{r_k}{\rho}\left(1-\frac{|x_{k-1}|}{|x_k|}\right),\ \text{and}\ t_k=\sum^k_{j=1} \tau_j.
\]
Then the following holds.
\begin{itemize}
 \item For $i=1,\ldots,n$,
 \[
 \bar u_i(t)=\left\{
 \begin{array}{ll}
  -\frac{\rho x_i}{r_k x_{k-1}}\lambda\gamma & \quad\text{for}\quad t\in [t_{k-1},t_k),\ k=1,\ldots,i,\\
  0 & \text{for}\quad t\geq t_i.
 \end{array}
 \right.
 \]
 \item For $i=n+1,\ldots,d$, $\bar u_i\equiv 0$.
\end{itemize}
\end{cor}

\section{Numerical experiments}
In the following, we present numerical experiments illustrating the structural properties discussed in the previous sections related to the minimization of 
\[J(x,u):=\int\limits_0^{\infty}\left(\frac{\|x(s)\|_2^2}{2}+\gamma\|u(s)\|_p^p\right)\,e^{-\lambda s}\,ds\,.
\]
In order to obtain approximate controllers for the infinite horizon optimal control problem, we follow the dynamic programming approach. The solution of the corresponding Hamilton-Jacobi-Bellman equation \eqref{HJB} is numerically approximated by a first-order, semi-Lagrangian scheme. In particular, when $p\leq 1$, we apply the approximation scheme presented in \cite{KKK15}, where the minimization of the Hamiltonian \eqref{Hamiltonian} is performed by a semismooth Newton method. When $p=0.5$, we resort to the scheme presented in \cite{AFK15}, where the minimization is carried out by an evaluation of the Hamiltonian over a discrete set of control values. Boundedness and Lipschitz continuity with respect to the state variable of both the dynamics $f (x,u)$ and the running cost $l(x,u)$ are sufficient to guarantee the convergence of the semi-Lagrangian scheme for the value function. However, convergence of the optimal controllers is a more delicate issue, as it requires among other hypotheses, the convexity of the running cost with respect to the
controls \cite[Chapter 8]{FFSIAM}. This assumption is not fulfilled when $p < 1$. Nevertheless, the proposed algorithm converges and provides optimal controls with the expected properties. For all the tests, the computational domain is $\Omega=[-1,1]^2$, which is discretized with a mesh parameter $k=0.025$.  The discount factor is set $\lambda=0.2$, and the control weight $\gamma=1$.

\subsection{Test 0: preliminaries}
We begin by recalling standard results in the control of Eikonal dynamics of the form
\begin{align*}
\dot x_1(s)&=u_1(s)\\
\dot x_2(s)&=u_2(s)\,,
\end{align*}
where $\|(u_1(s),u_2(s))\|_2\leq 1$. We first consider the case of infinite horizon control with a quadratic control penalization, i.e. $p=2$. Results shown in Figure \ref{eikp2} illustrate a simple setting where the control constraint is inactive and, as in the linear quadratic setting, the associated value function is quadratic and the optimal feedback is linear. The presented trajectories for the initial condition $x(0)=(-0.75,-0.6)$ exhibit asymptotic stabilization, which is the standard result for this case.

\begin{figure}[!h]
\includegraphics[width=0.495\textwidth]{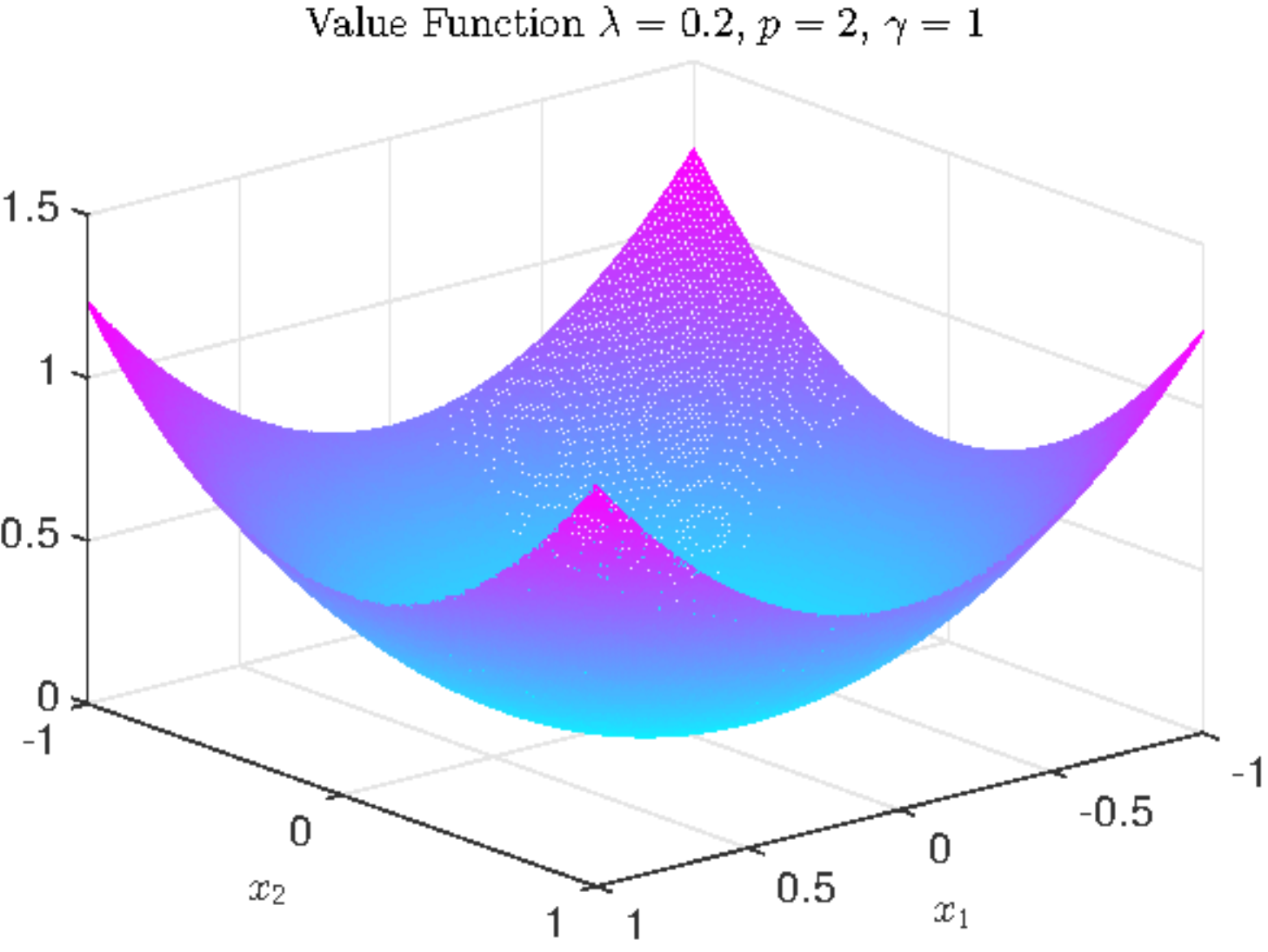}
\includegraphics[width=0.495\textwidth]{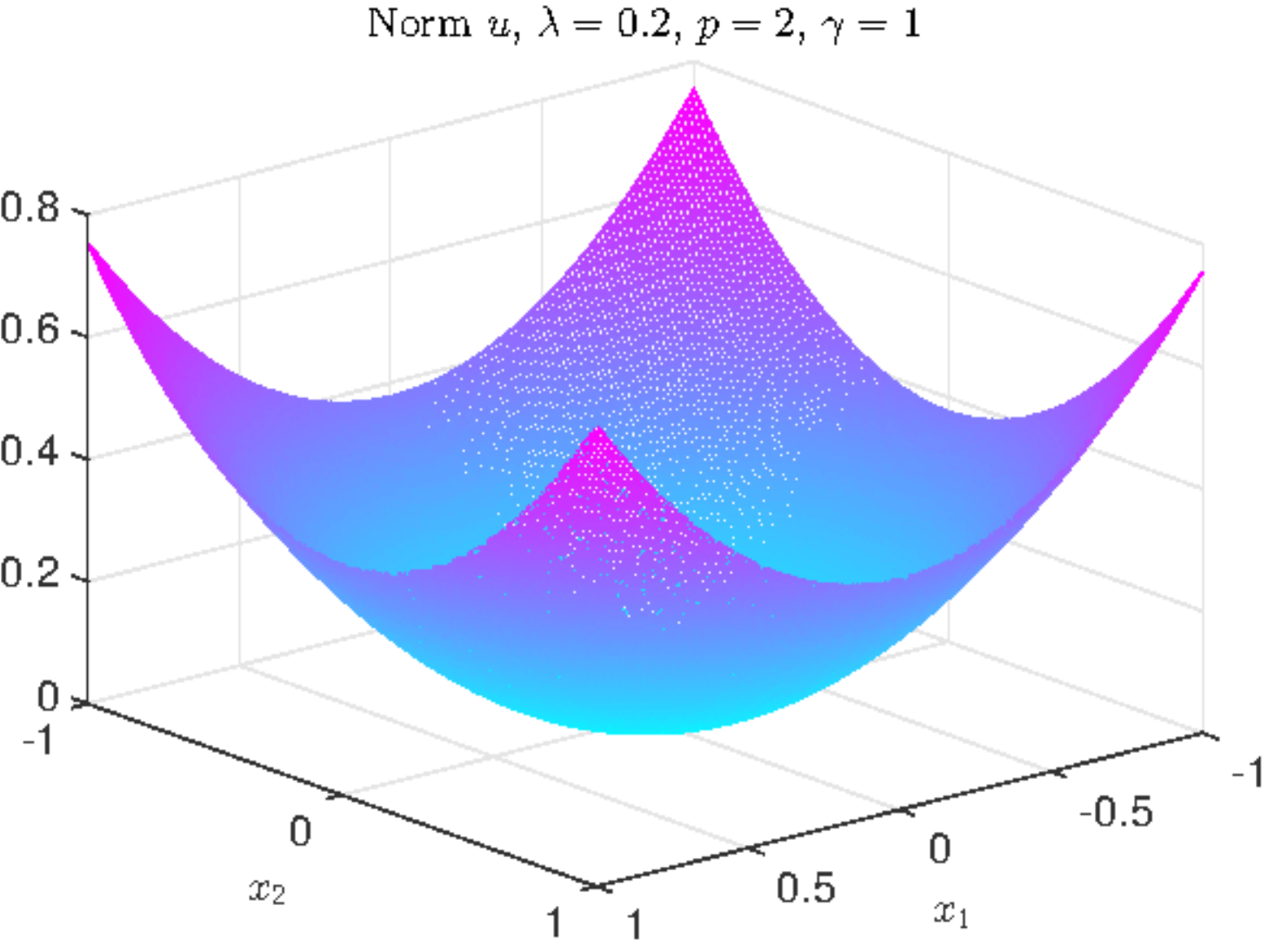}
\vskip 5mm
\includegraphics[width=0.495\textwidth]{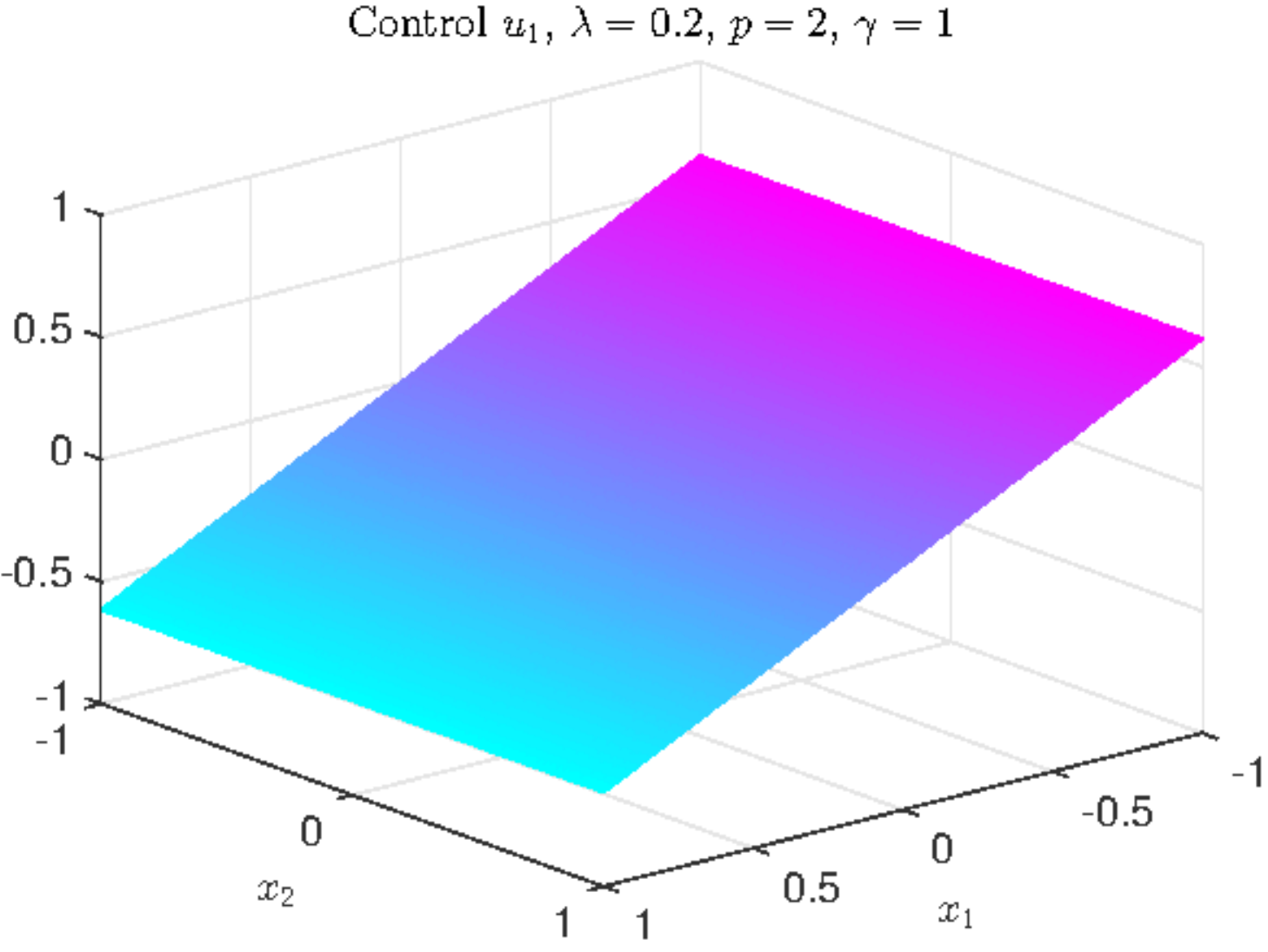}
\includegraphics[width=0.495\textwidth]{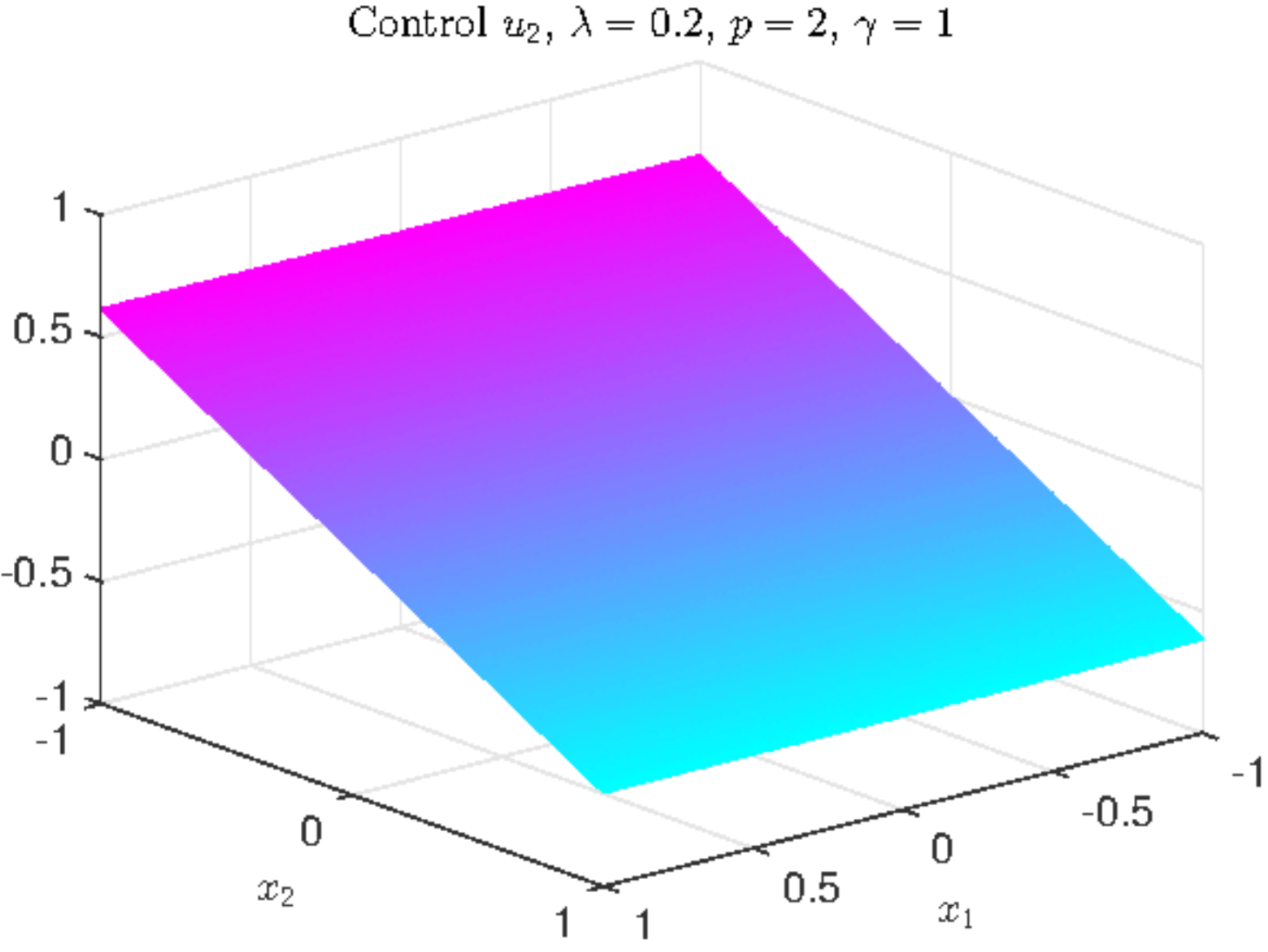}
\vskip 5mm
\centering
\includegraphics[width=0.495\textwidth]{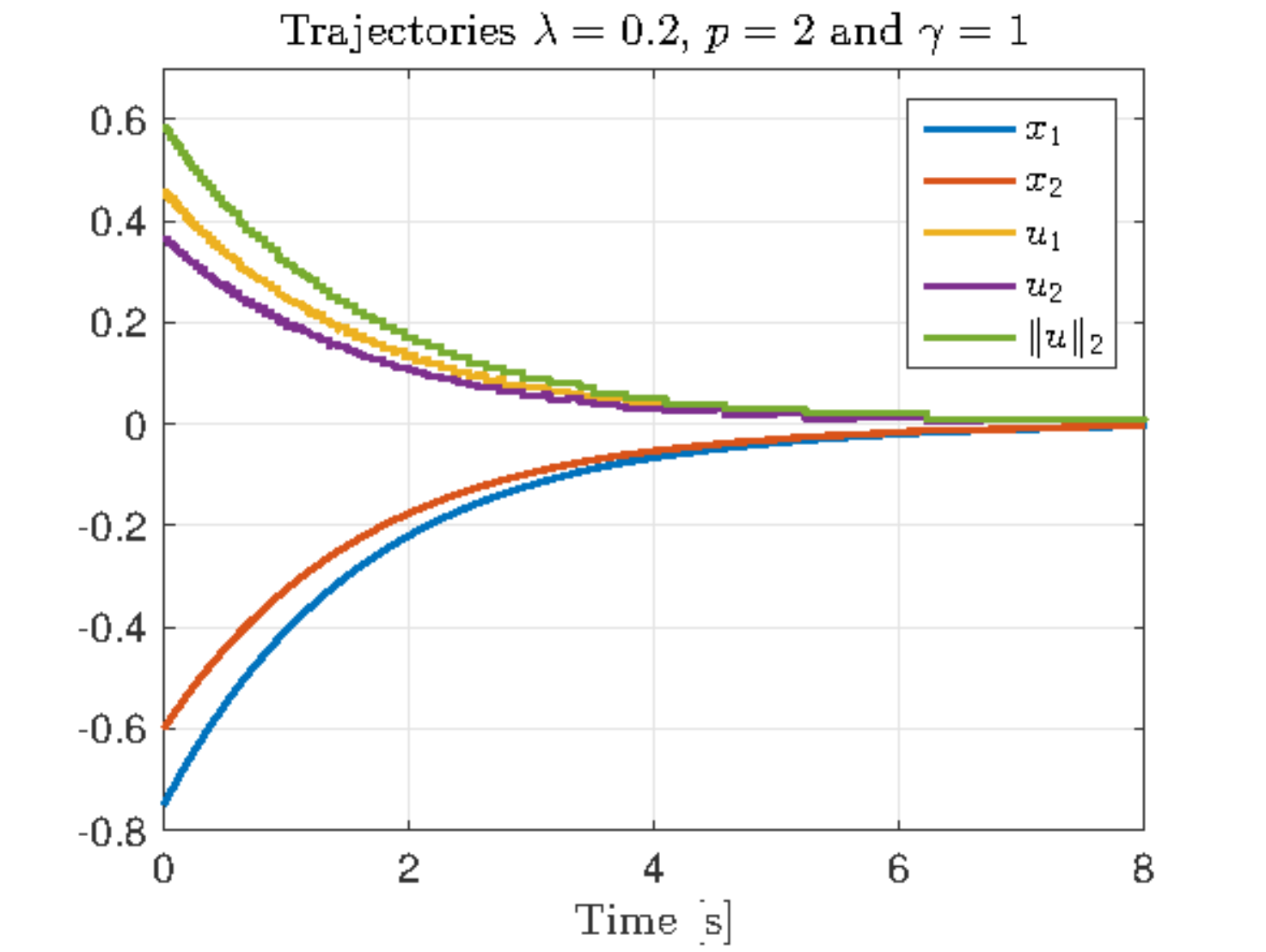}
\caption{Quadratic control problem with Eikonal dynamics and Euclidean norm contraint. Row 1 and 2: value function $V(x_1,x_2)$, $\|u\|_2$ -norm of the optimal control, optimal controls $u_1(x_1,x_2)$ and $u_2(x_1,x_2)$, over the state space  $\Omega=[-1,1]^2$. Row 3: trajectories for the initial condition $(x_1(0),x_2(0))=(-0.75,-0.6).$}\label{eikp2}
\end{figure}

As a counterpart to the quadratic control problem, we present the results related to minimum time control to the origin. Results are shown in Figure \ref{mtp}. In this case, the value function corresponds to the distance function to the origin, and the Euclidean constraint is active in the whole domain (except for the origin), leading to a bang-bang controller, as is expected for minimum time problems and linear dynamics. In contrast to the quadratic control problem, the minimum time controller generates trajectories which arrive to the origin in finite time. The optimal control trajectories exhibit some chattering due to the discrete nature of the approximate value function, however, the bang-bang behavior is respected at all times.

\begin{figure}[!h]
\includegraphics[width=0.495\textwidth]{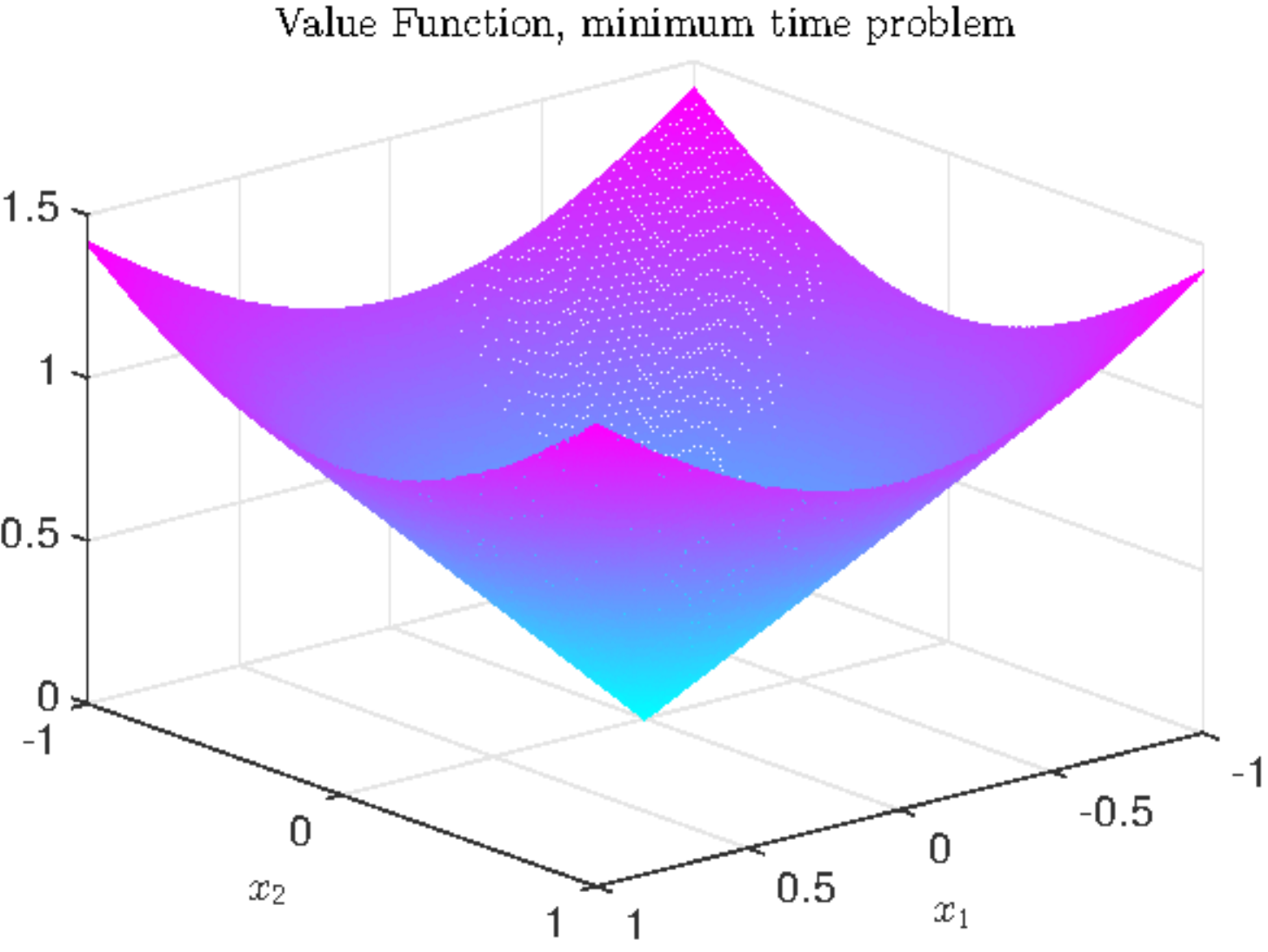}
\includegraphics[width=0.495\textwidth]{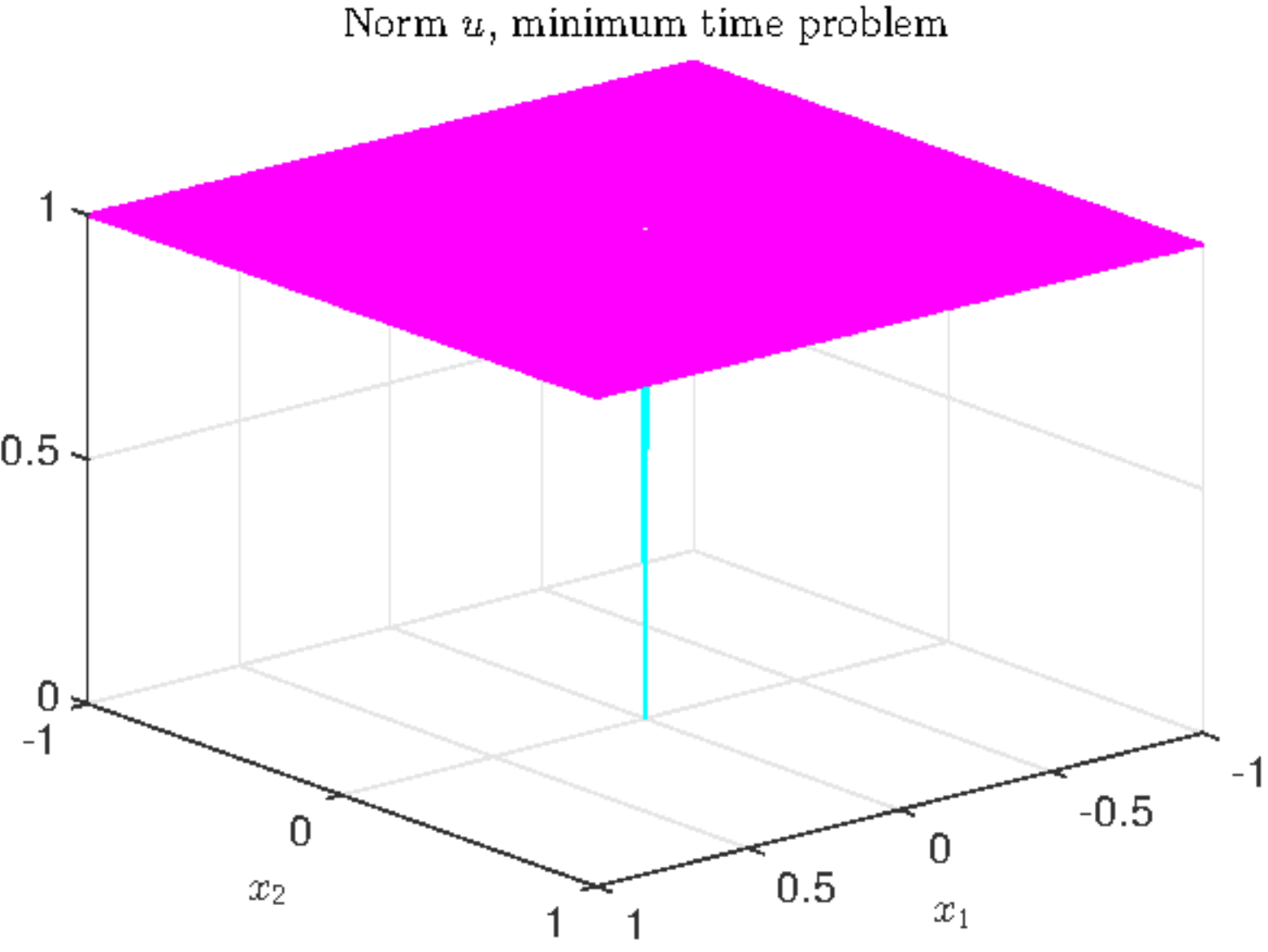}
\vskip 5mm
\includegraphics[width=0.495\textwidth]{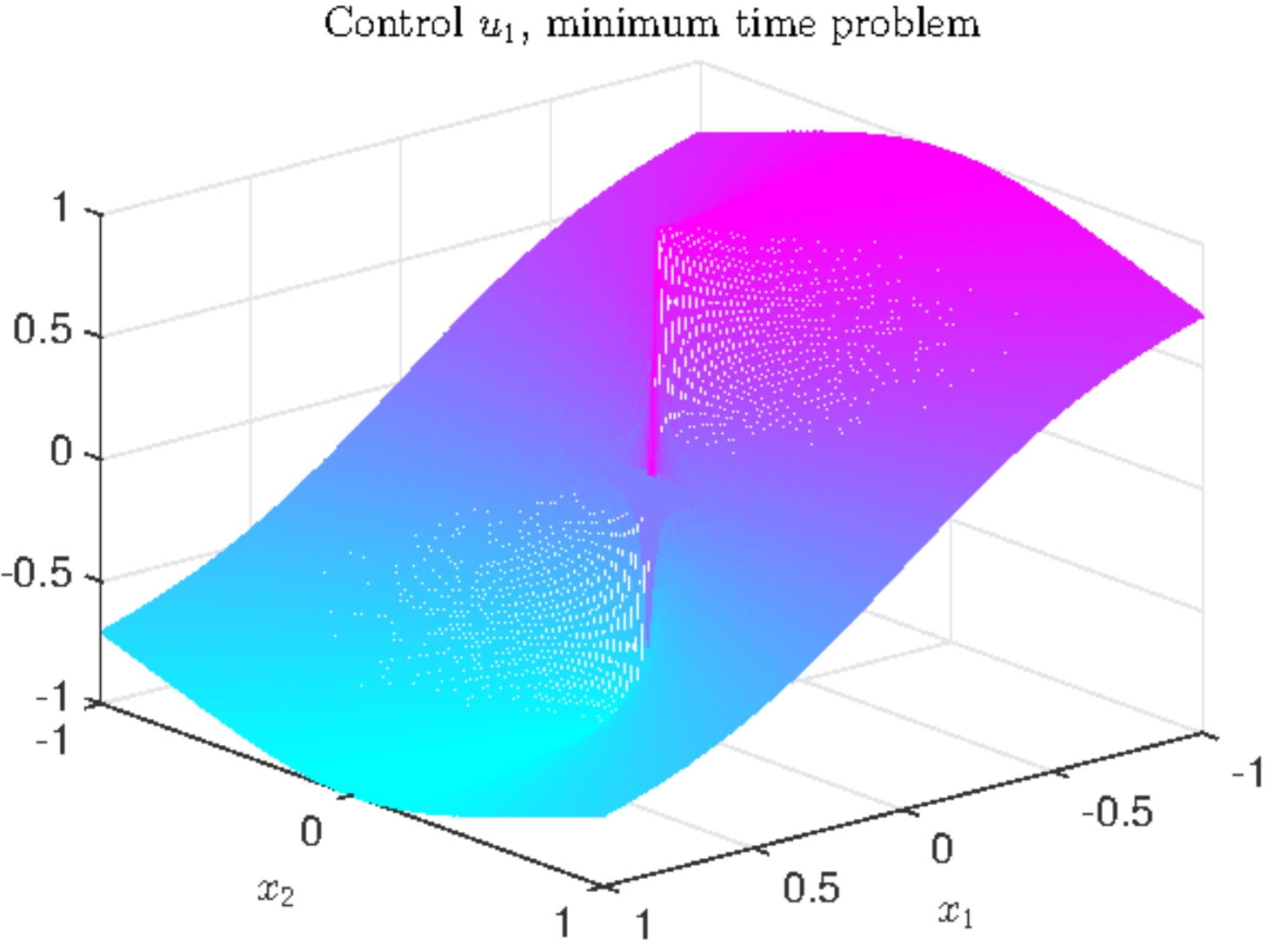}
\includegraphics[width=0.495\textwidth]{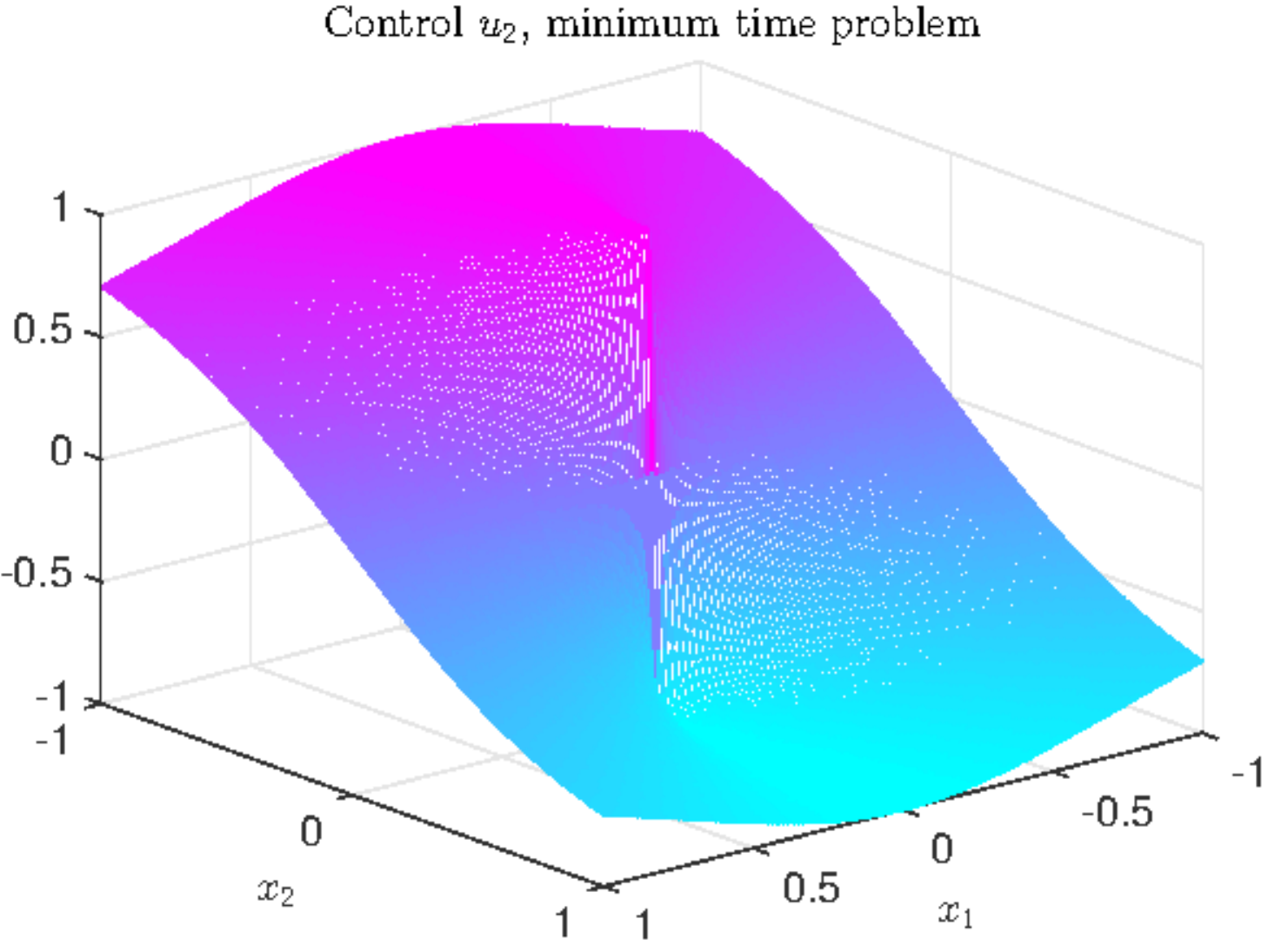}
\vskip 5mm
\centering
\includegraphics[width=0.495\textwidth]{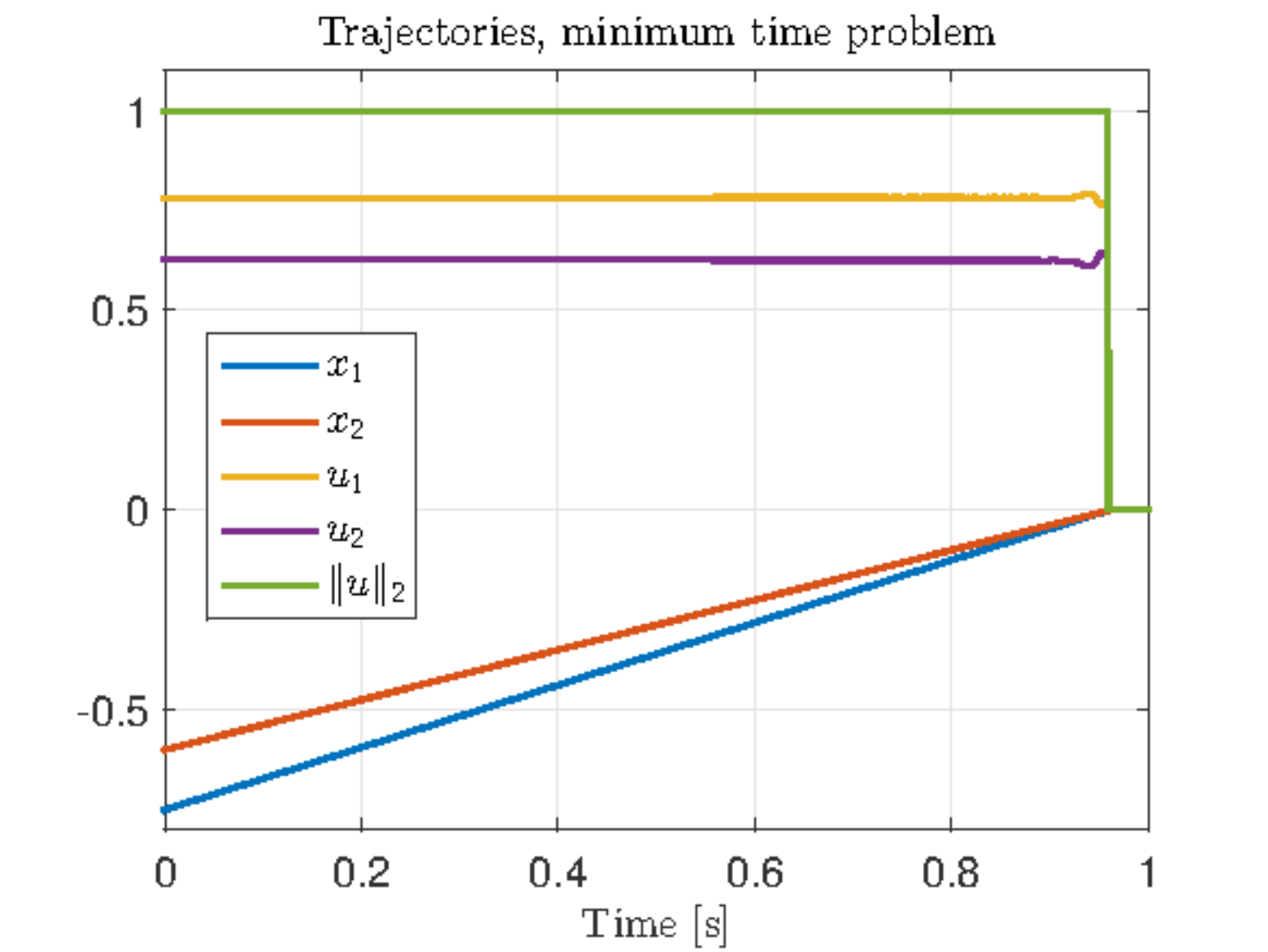}
\caption{Minimum time problem to the origin with Eikonal dynamics. Row 1 and 2: value function $V(x_1,x_2)$, $\|u\|_2$ -norm of the optimal control, optimal controls $u_1(x_1,x_2)$ and $u_2(x_1,x_2)$, over the state space  $\Omega=[-1,1]^2$. Row 3: trajectories for the initial condition $(x_1(0),x_2(0))=(-0.75,-0.6).$}\label{mtp}
\end{figure}

\subsection{Test 1: Eikonal dynamics}
In this first case related to sparse controls, we set $p=1$. Results presented in Figure \ref{eikp1}, illustrate the effect of the control penalization norm on the optimal control problem. The value function loses its quadratic nature, and more importantly, a sparsity region where the optimal control is $(0,0)$ is created around the origin. As predicted in section 6,  the width of this region is equal to $\gamma\lambda$. The optimal control is bang-bang (w.r.t. the Euclidean norm constraint), or zero (sparse), with the Euclidean norm constraint active outside of the sparsity region. We also observe two bands outside the sparsity region where one coordinate is zero while the other makes use of the total control constraint (directional sparsity).  All these features are illustrated in the trajectory plot at the end of Figure \ref{eikp1}. The structure of the optimal control confirms the results obtained in the third case of Theorem 6.3.

\noindent Comparing Figures 2 and 4 we note that in the Eikonal case for the limit $\lambda \to 0$, the optimal control of the $\ell^1$ penalized control problem converges to the optimal control of the minimum time problem. This implies a finite time arrival to the origin, which differs from the asymptotic stabilization properties that can be derived for the $\ell^2$ control penalization. Furthermore, it can be verified that this result is independent of the parameter $\gamma$.

\begin{figure}[!h]
\includegraphics[width=0.495\textwidth]{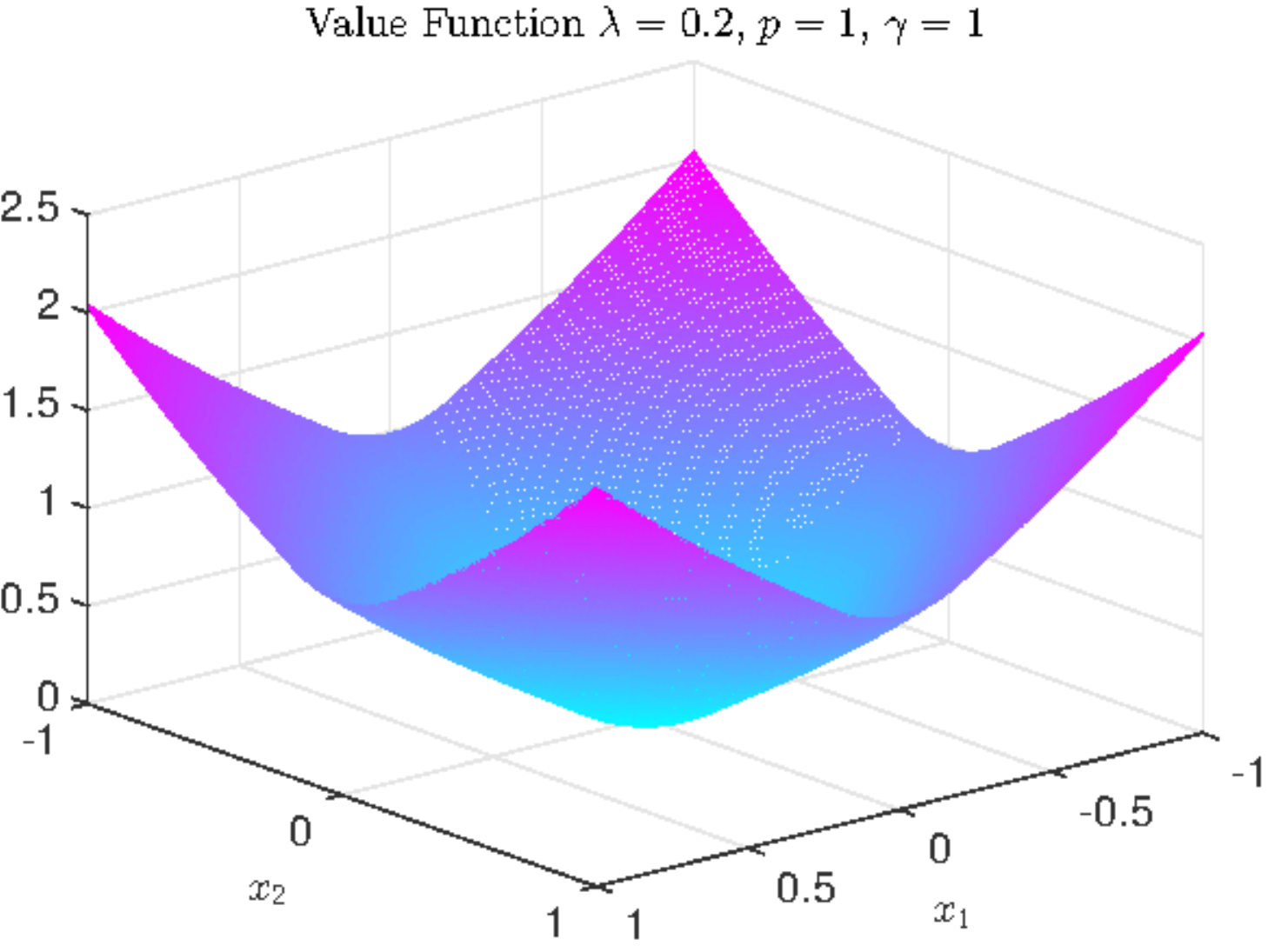}
\includegraphics[width=0.495\textwidth]{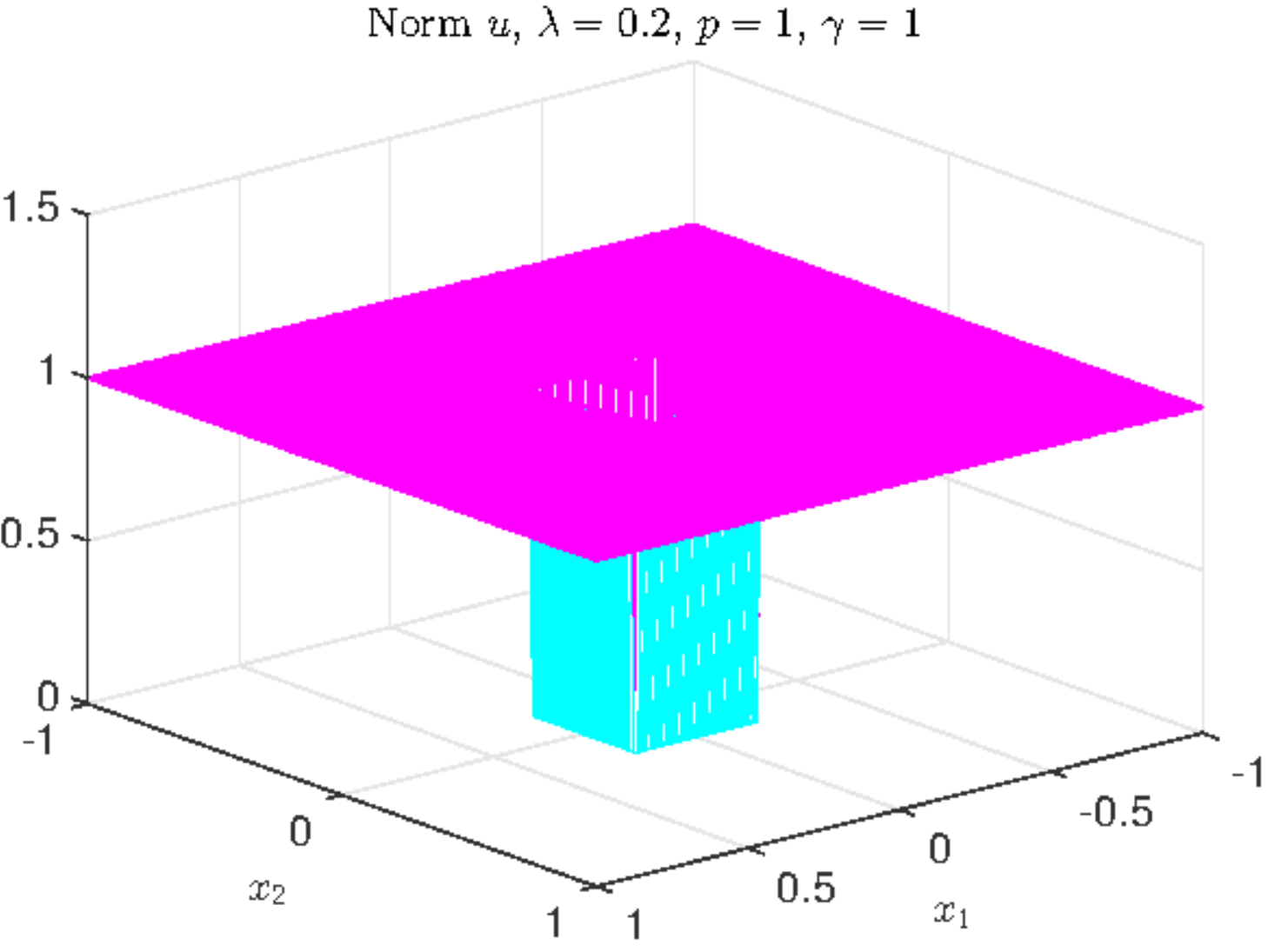}
\vskip 5mm
\includegraphics[width=0.495\textwidth]{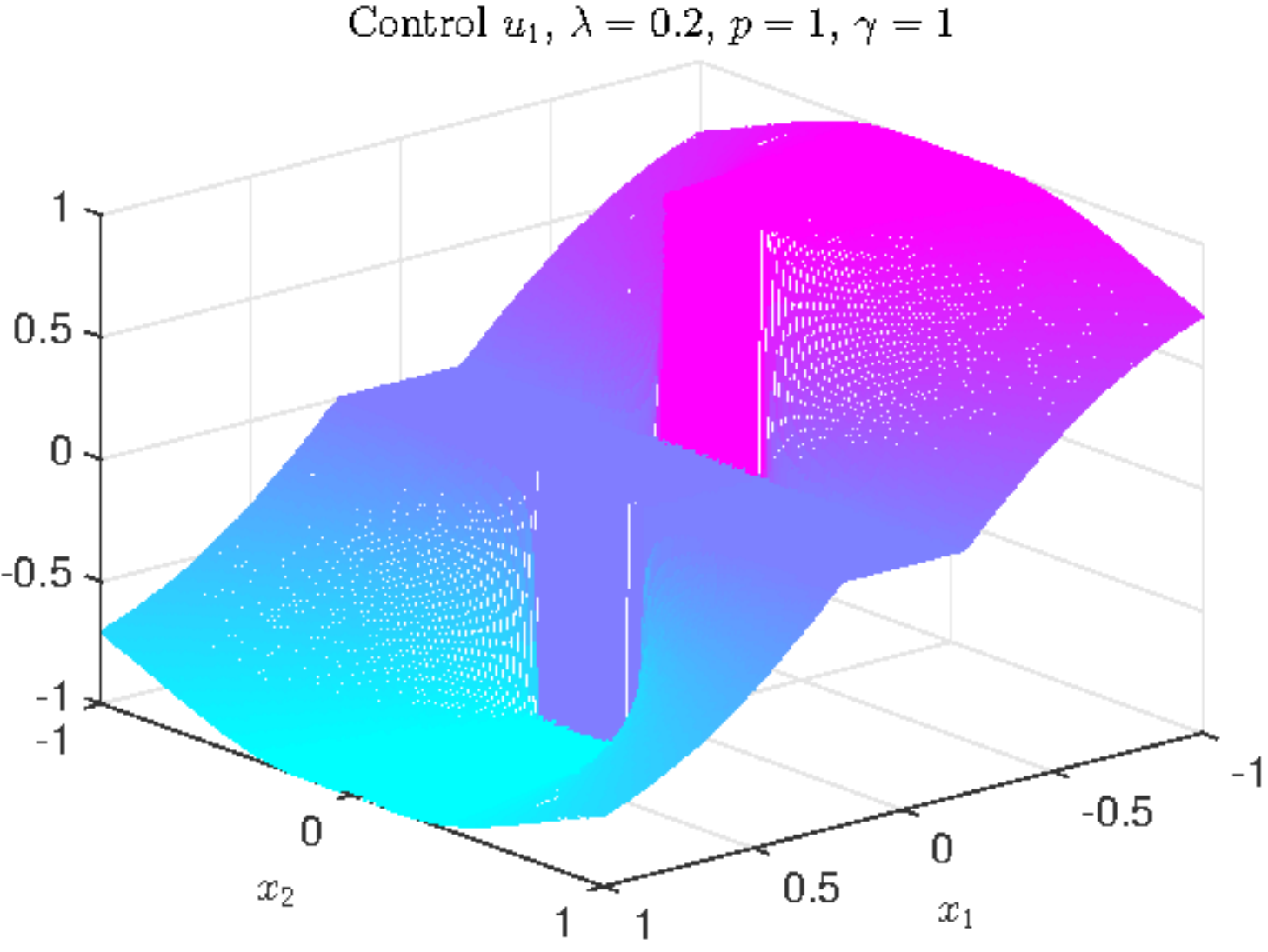}
\includegraphics[width=0.495\textwidth]{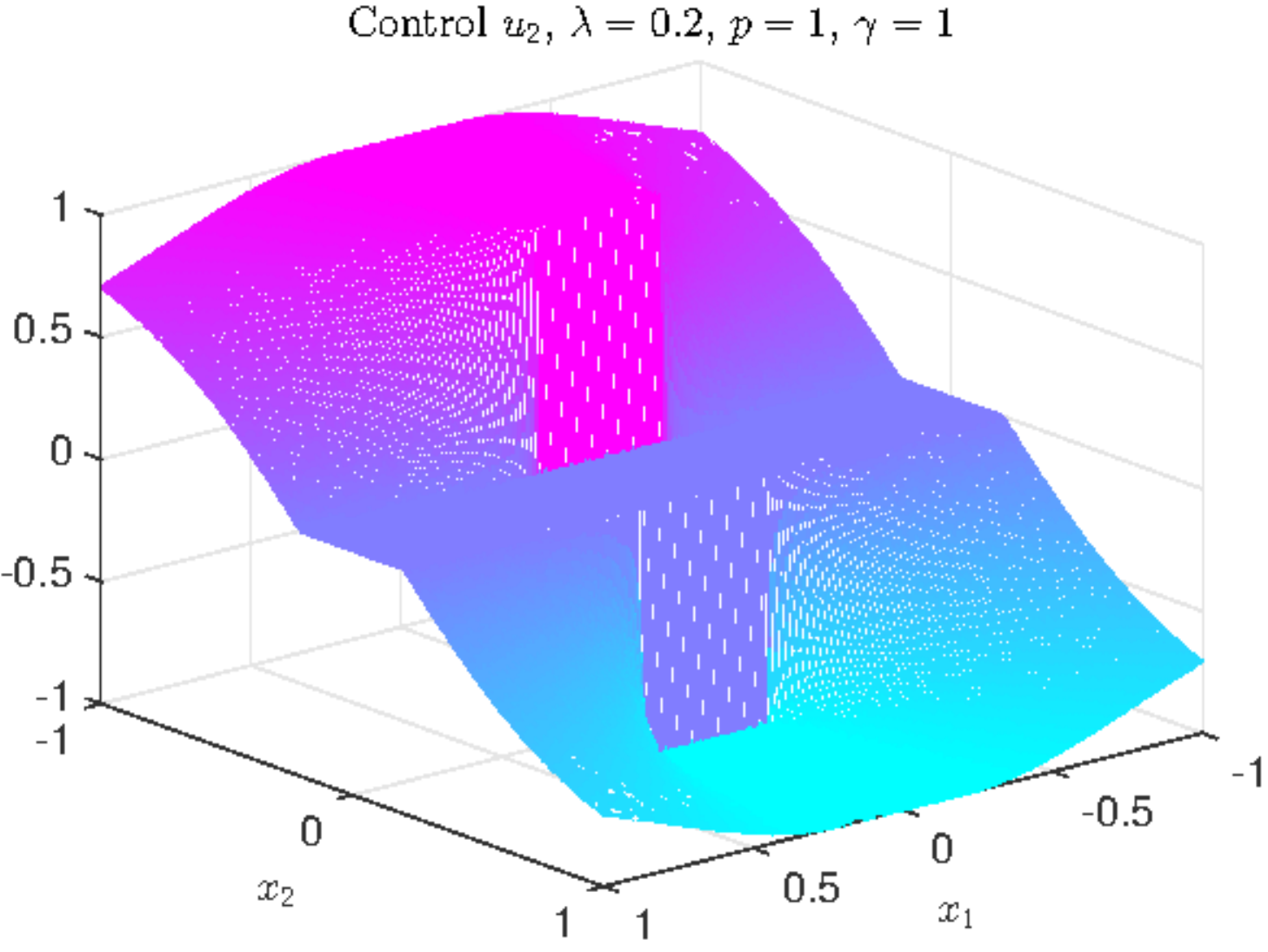}
\vskip 5mm
\centering
\includegraphics[width=0.495\textwidth]{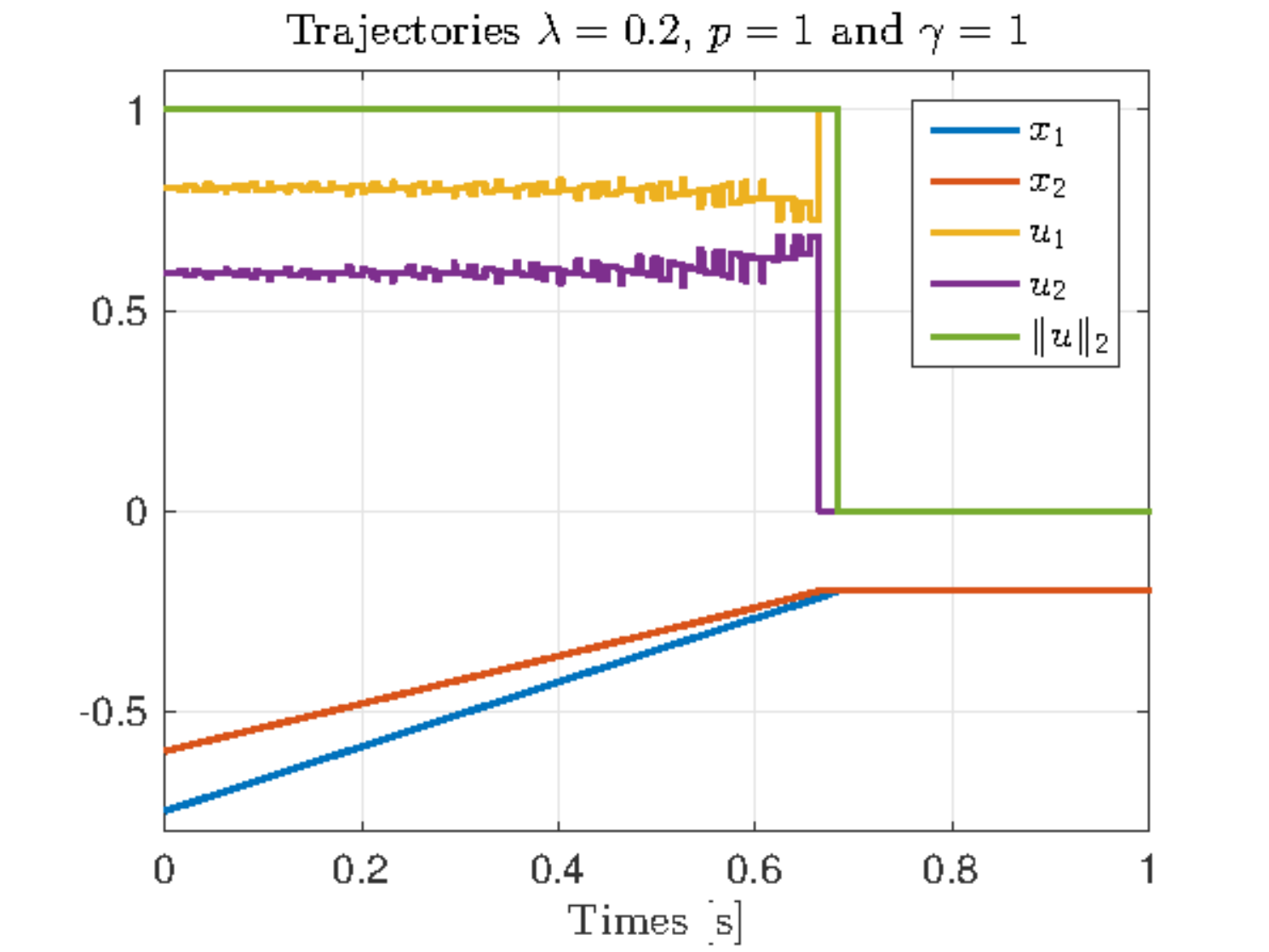}
\caption{Inifinite horizon control of the Eikonal dynamics with $p=1$. Row 1 and 2: value function $V(x_1,x_2)$, $\|u\|_2$ -norm of the optimal control, optimal controls $u_1(x_1,x_2)$ and $u_2(x_1,x_2)$, over the state space  $\Omega=[-1,1]^2$. Row 3: trajectories for the initial condition $(x_1(0),x_2(0))=(-0.75,-0.6).$}\label{eikp1}
\end{figure}

\begin{figure}[!h]
\includegraphics[width=0.495\textwidth]{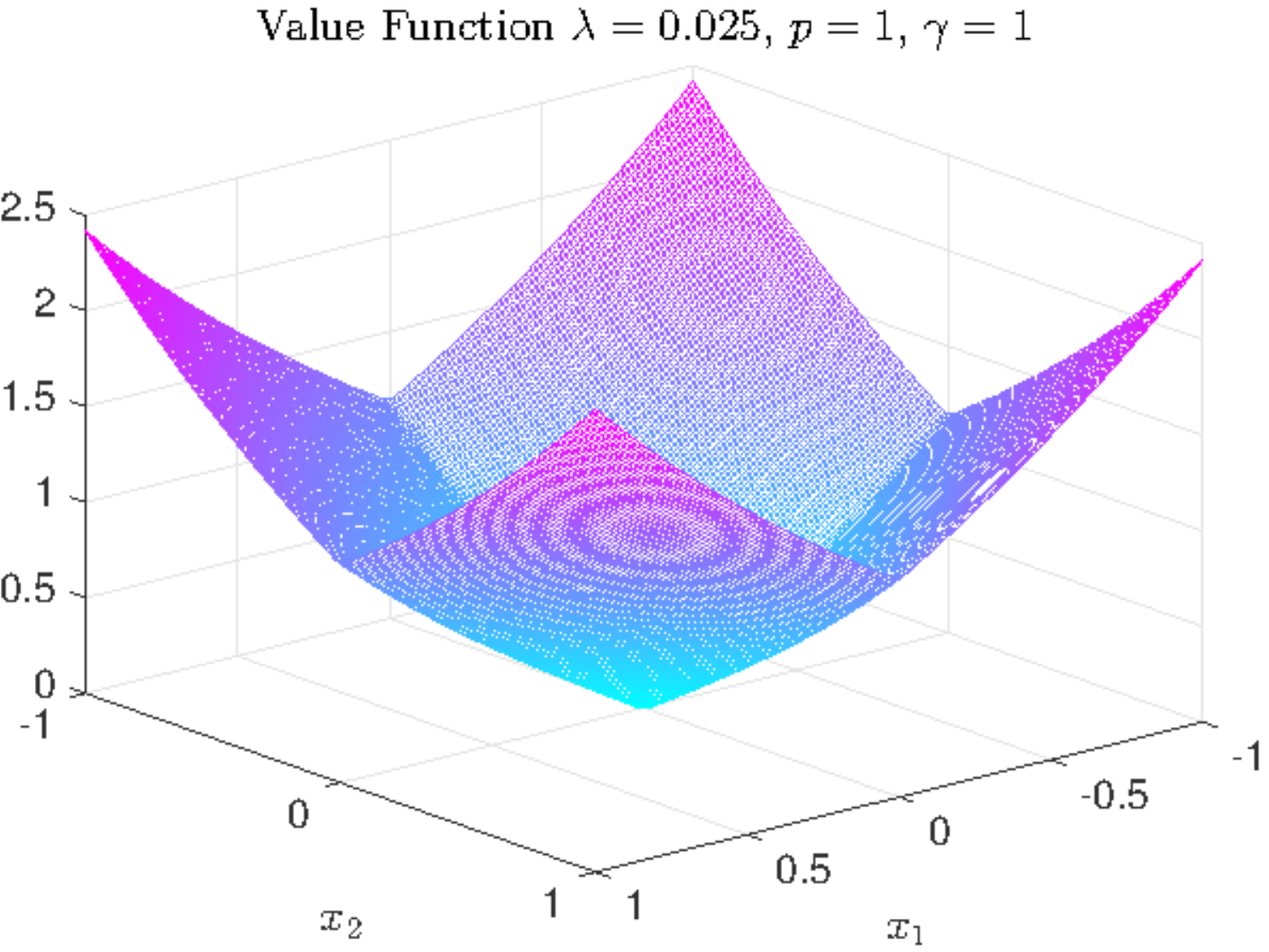}
\includegraphics[width=0.495\textwidth]{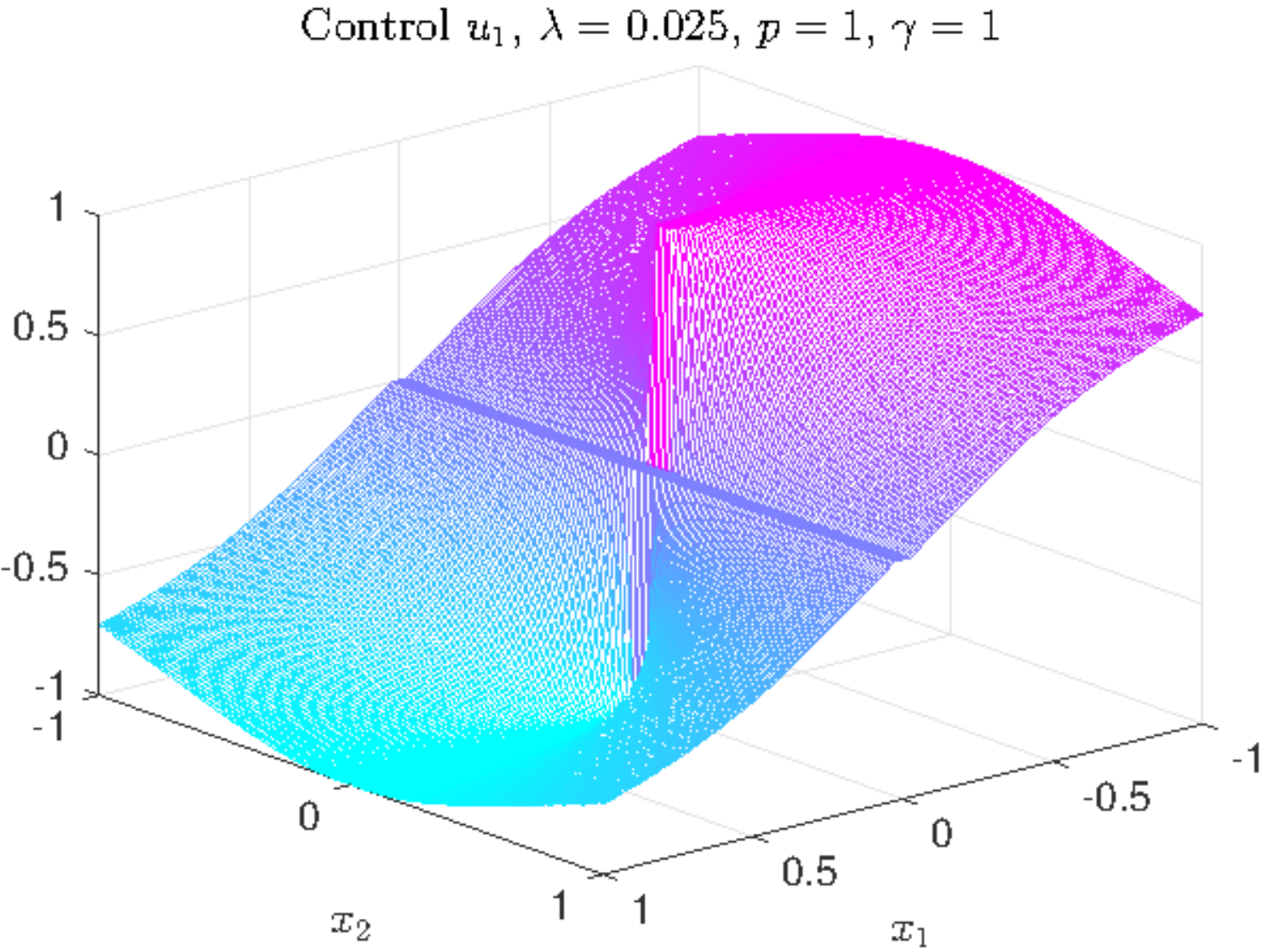}
\caption{Infinite horizon control of the Eikonal dynamics with $p=1$, $\gamma=1$, and $\lambda=0.025$. Value function (left) and optimal control $u_1$ (right). Reducing the value of $\lambda$ shrinks the sparsity region, and the optimal control converges to the minimum time optimal control depicted in Figure \ref{mtp}.}\label{eikp1small}
\end{figure}

\noindent A more dramatic change in the control structure is shown in Figure \ref{eikp05}, where results are presented for the case $p=0.5$. The control now becomes not only sparse, but also a bang-bang switching controller, i.e. a control which is  either active with full energy in only one coordinate at a time, or completely off. The sparsity region where the control is fully inactive remains the same as in the $p=1$ case (we conjecture that as in the sparse case, it only depends on the product $\lambda\gamma$). Outside this region, the optimal control belongs to the set of directions $U^*=\{(1,0),(0,1),(-1,0),(0,-1)\}$. These results are further illustrated in the last row of Figure \ref{eikp05}, where switching controllers lead to a zigzag optimal trajectory which stops once the sparsity region has been reached.

\begin{figure}[!h]
\includegraphics[width=0.495\textwidth]{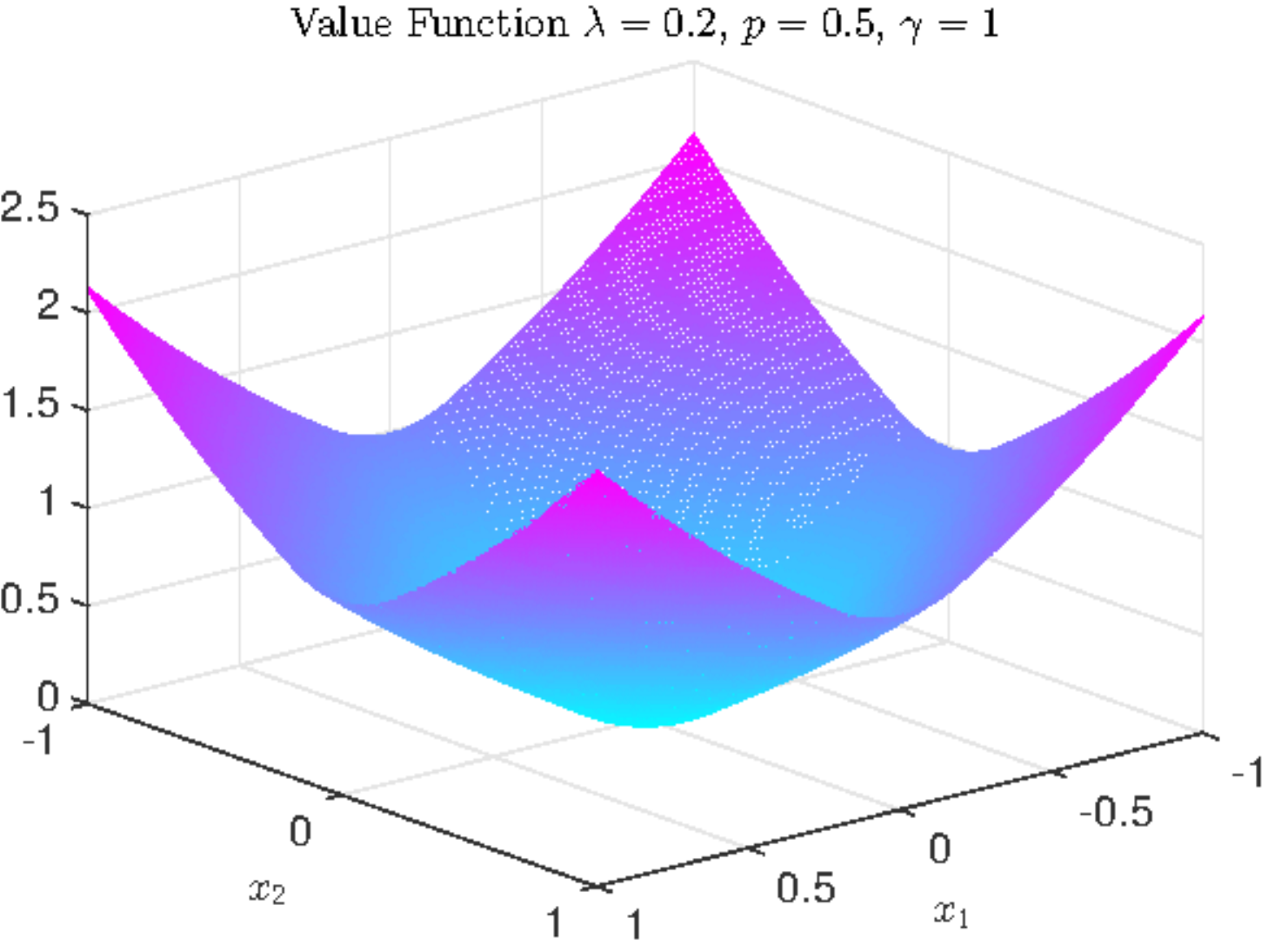}
\includegraphics[width=0.495\textwidth]{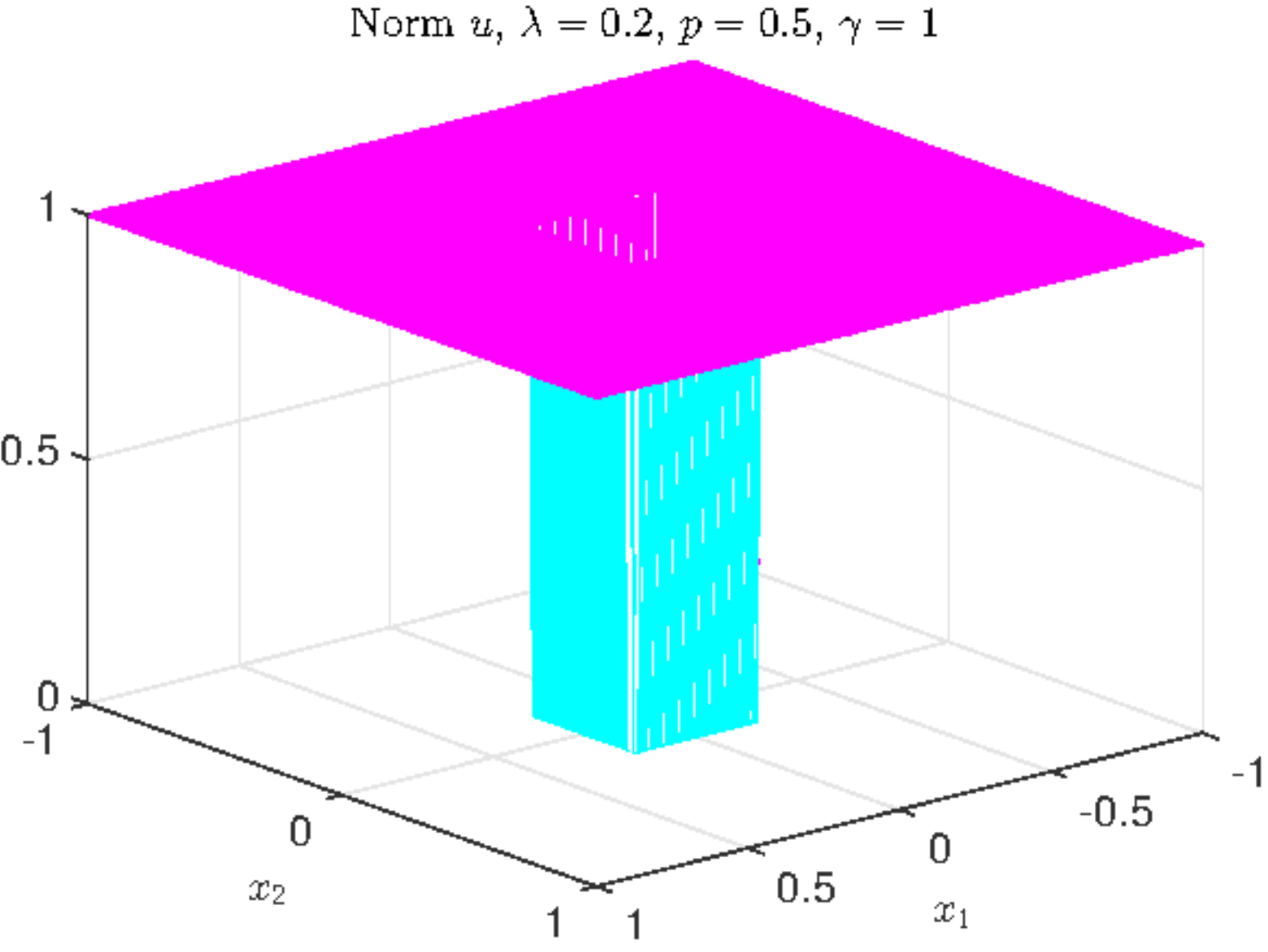}
\vskip 5mm
\includegraphics[width=0.495\textwidth]{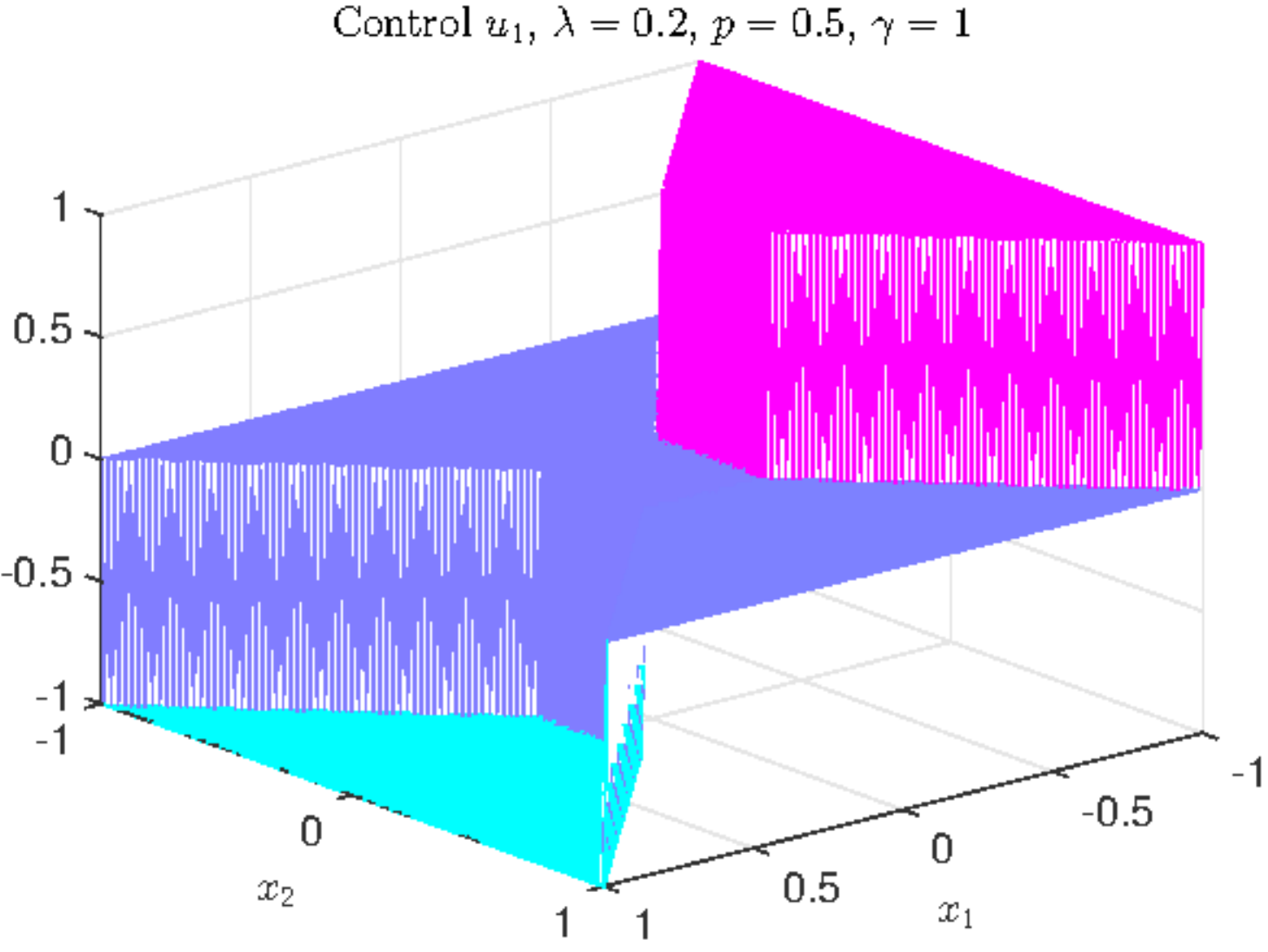}
\includegraphics[width=0.495\textwidth]{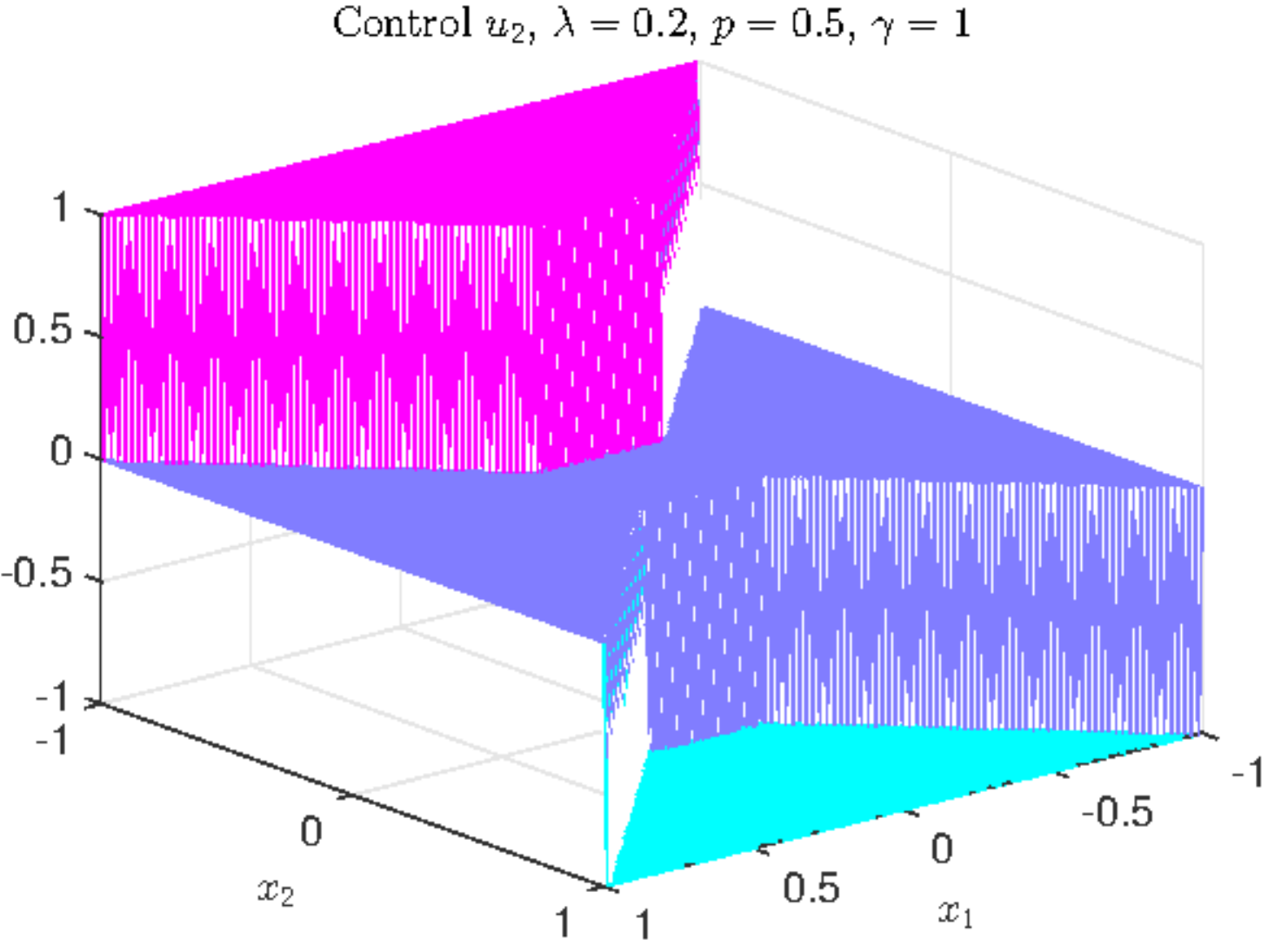}
\vskip 5mm
\includegraphics[width=0.495\textwidth]{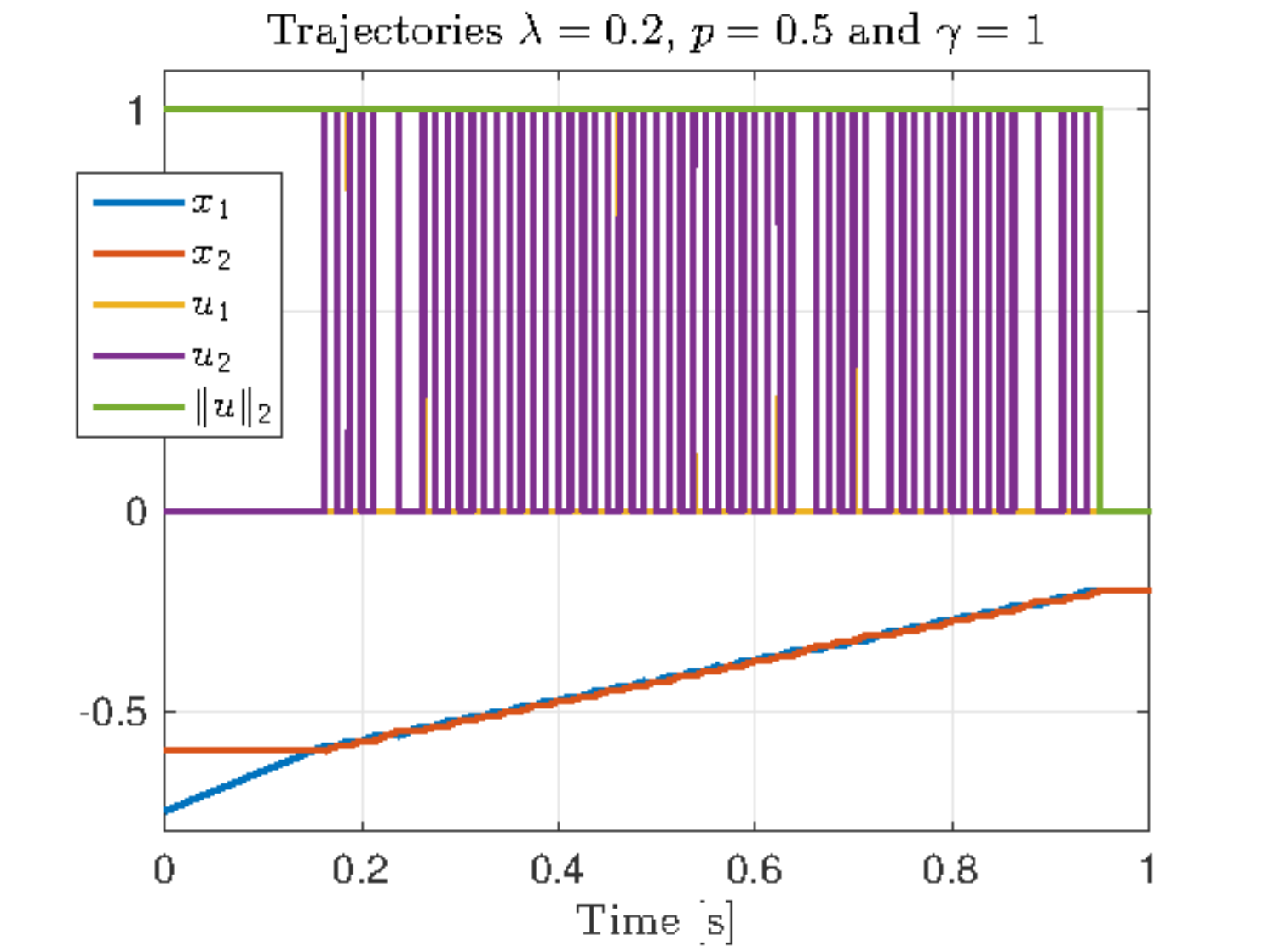}
\includegraphics[width=0.495\textwidth]{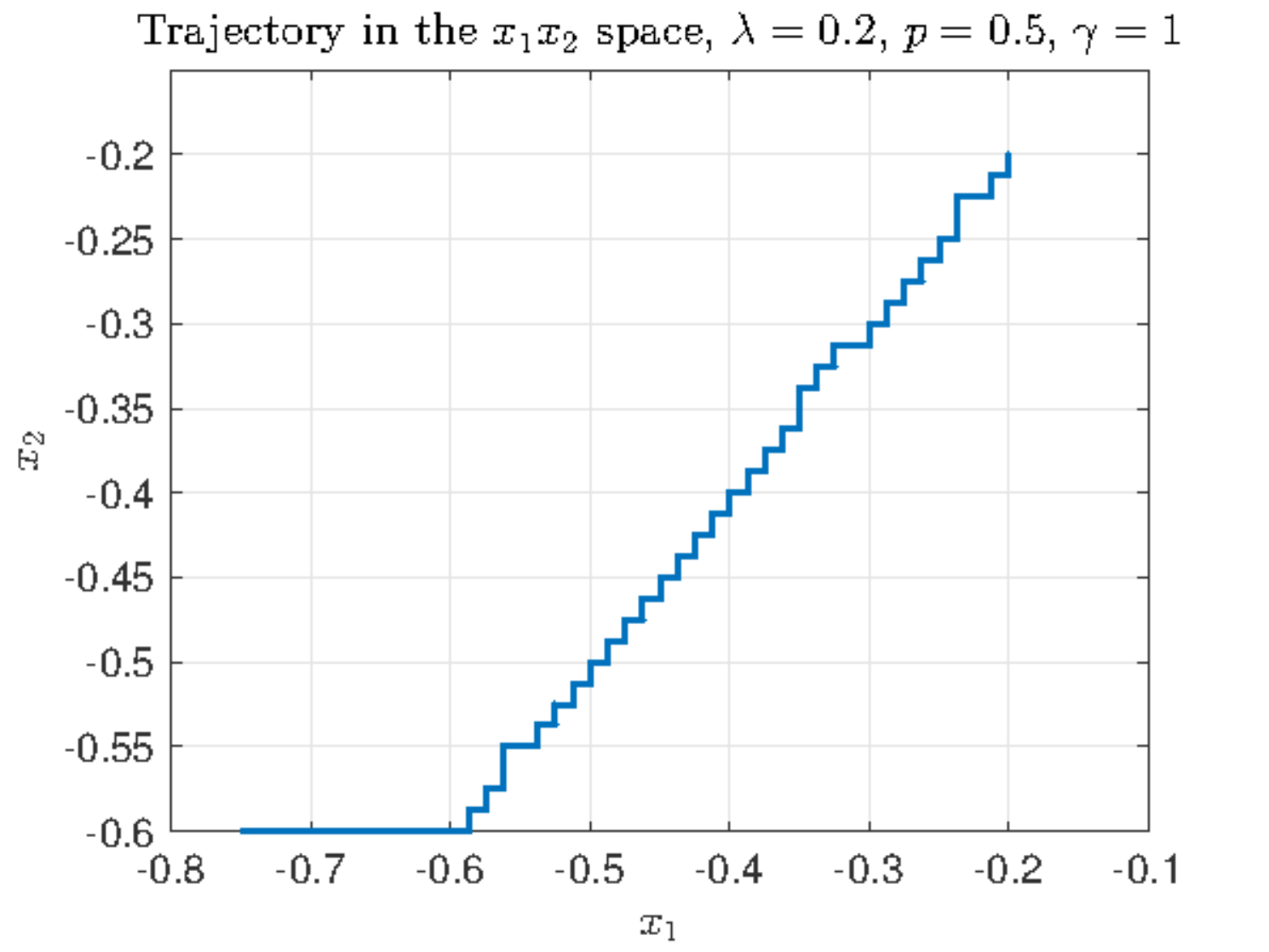}
\caption{Infinite horizon control of the Eikonal dynamics with $p=0.5$. Row 1 and 2: value function $V(x_1,x_2)$, $\|u\|_2$ -norm of the optimal control, optimal controls $u_1(x_1,x_2)$ and $u_2(x_1,x_2)$, over the state space  $\Omega=[-1,1]^2$. Row 3: trajectories for the initial condition $(x_1(0),x_2(0))=(-0.75,-0.6).$}\label{eikp05}
\end{figure}

\subsection{Test 2: nonlinear dynamics.}
One of the advantages of the numerical approximation of optimal controllers via the dynamic programming approach is the possibility of dealing with nonlinear dynamics at no additional effort. We perform similar tests as in the Eikonal case with nonlinear dynamics of the form
\begin{align*}
\dot x_1(s)&=x_1(s)(x_1(s)-q_1)+u_1(s)\\
\dot x_2(s)&=x_2(s)(x_2(s)-q_2)+u_2(s)\,,
\end{align*}
with $p_1,p_2>0$. In this case we set $q_1=0.6$ and $q_2=0.4$. The origin is locally asymptotically stable for $(x_1,x_2)\in (-q_1,q_1)\times (-q_2,q_2)$. We discuss the cases $p=2,1,0.5$.  As illustrated in Figure \ref{nonlinp2}, the $L^2$ cost functional leads to smooth trajectories tending to $0$ with both controls different from zero at all times. For the case $p=1$, shown in Figure \ref{nonlinp1}, the optimal controls are sparsified and bang-bang. Furthermore, once the state is sufficiently close to the origin, the control shuts off completely, and the free dynamics leads asymptotically to the origin. Finally, when $p=0.5$ as in \ref{nonlinp05}, the optimal control is further sparsified and becomes bang-bang and switching over a large portion of the state space. Compared to $p=1$, the optimal controls shut down to zero later when $p=0.5$.

\begin{figure}[!h]
\includegraphics[width=0.495\textwidth]{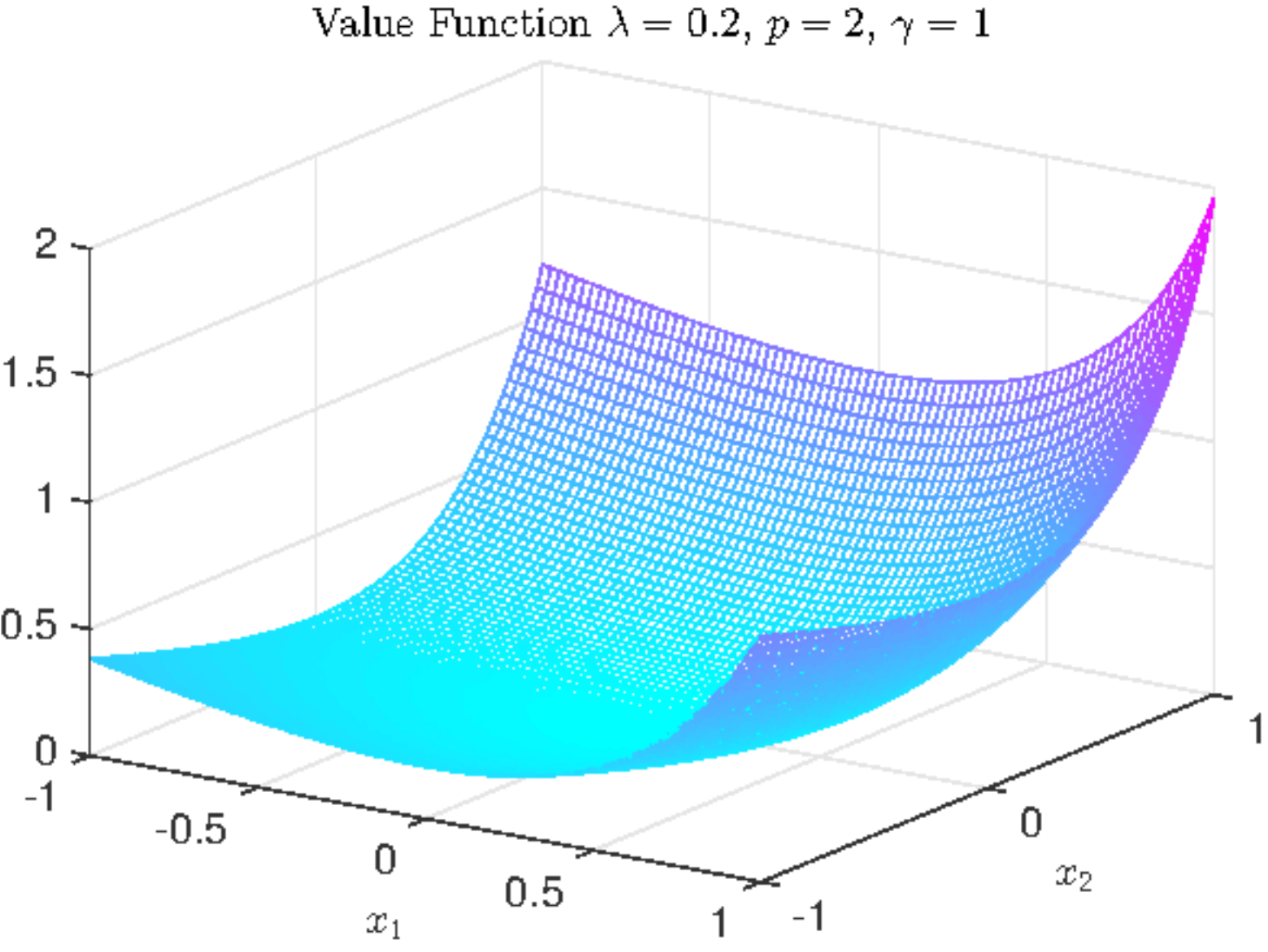}
\includegraphics[width=0.495\textwidth]{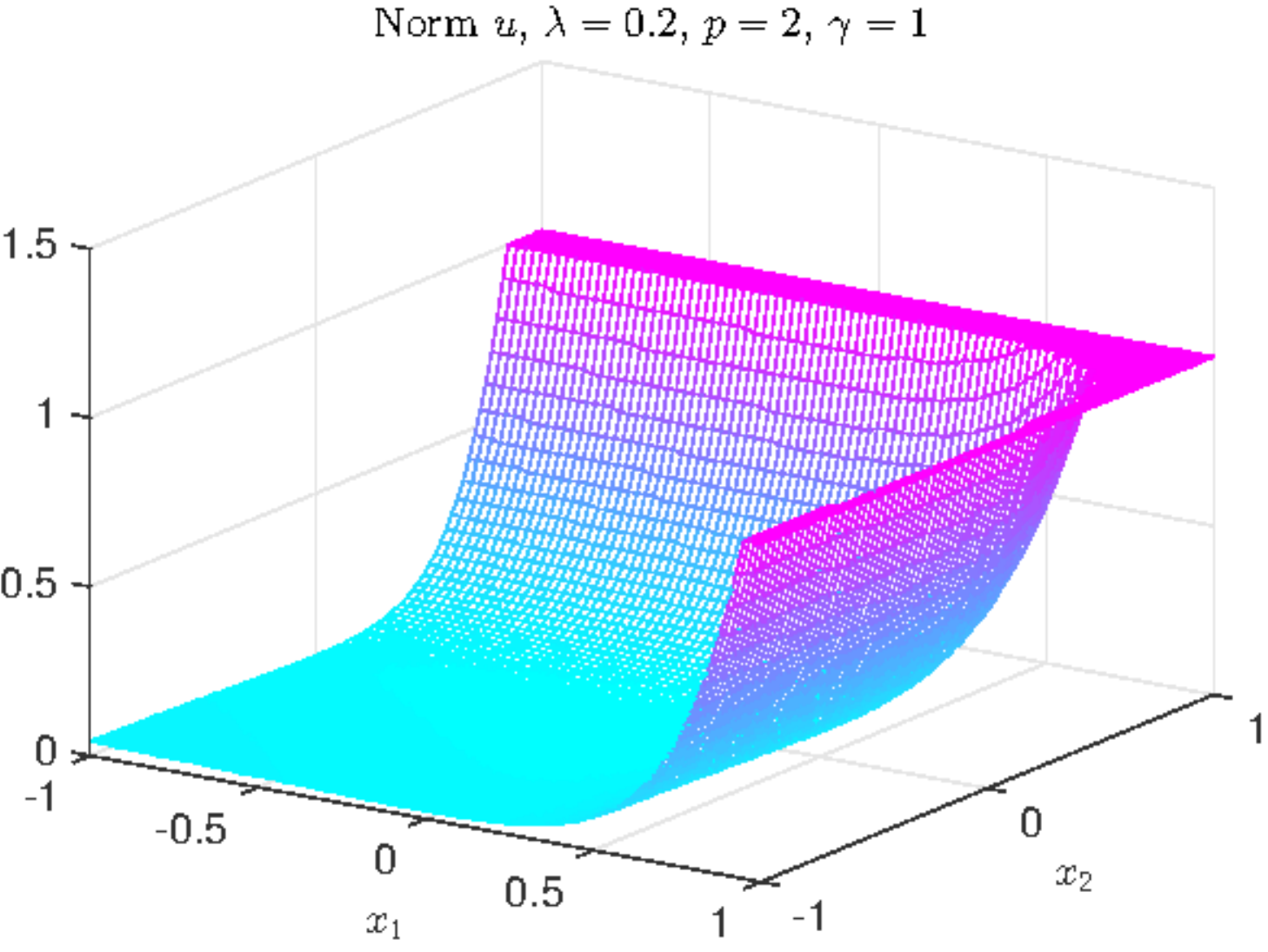}
\vskip 5mm
\includegraphics[width=0.495\textwidth]{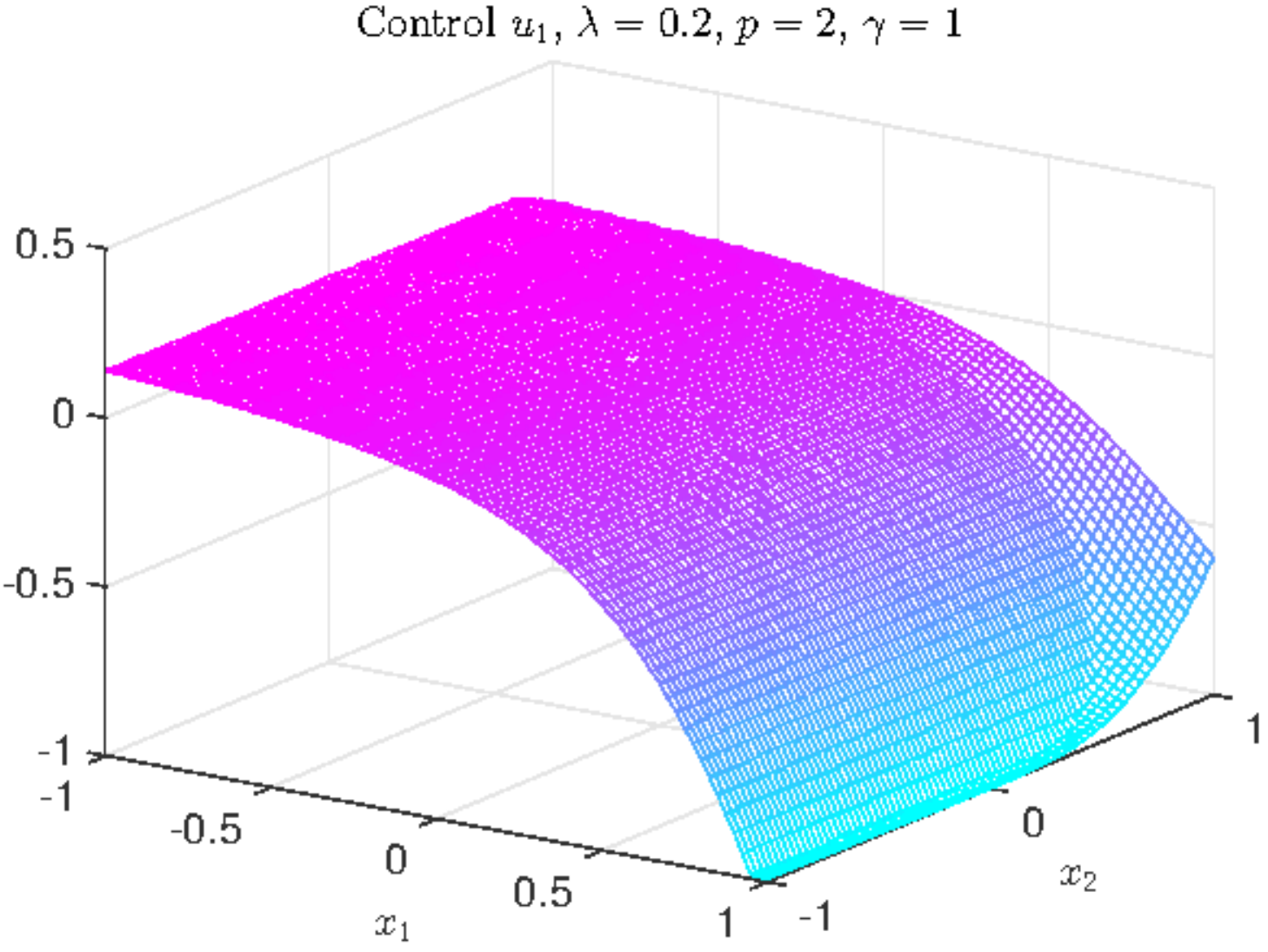}
\includegraphics[width=0.495\textwidth]{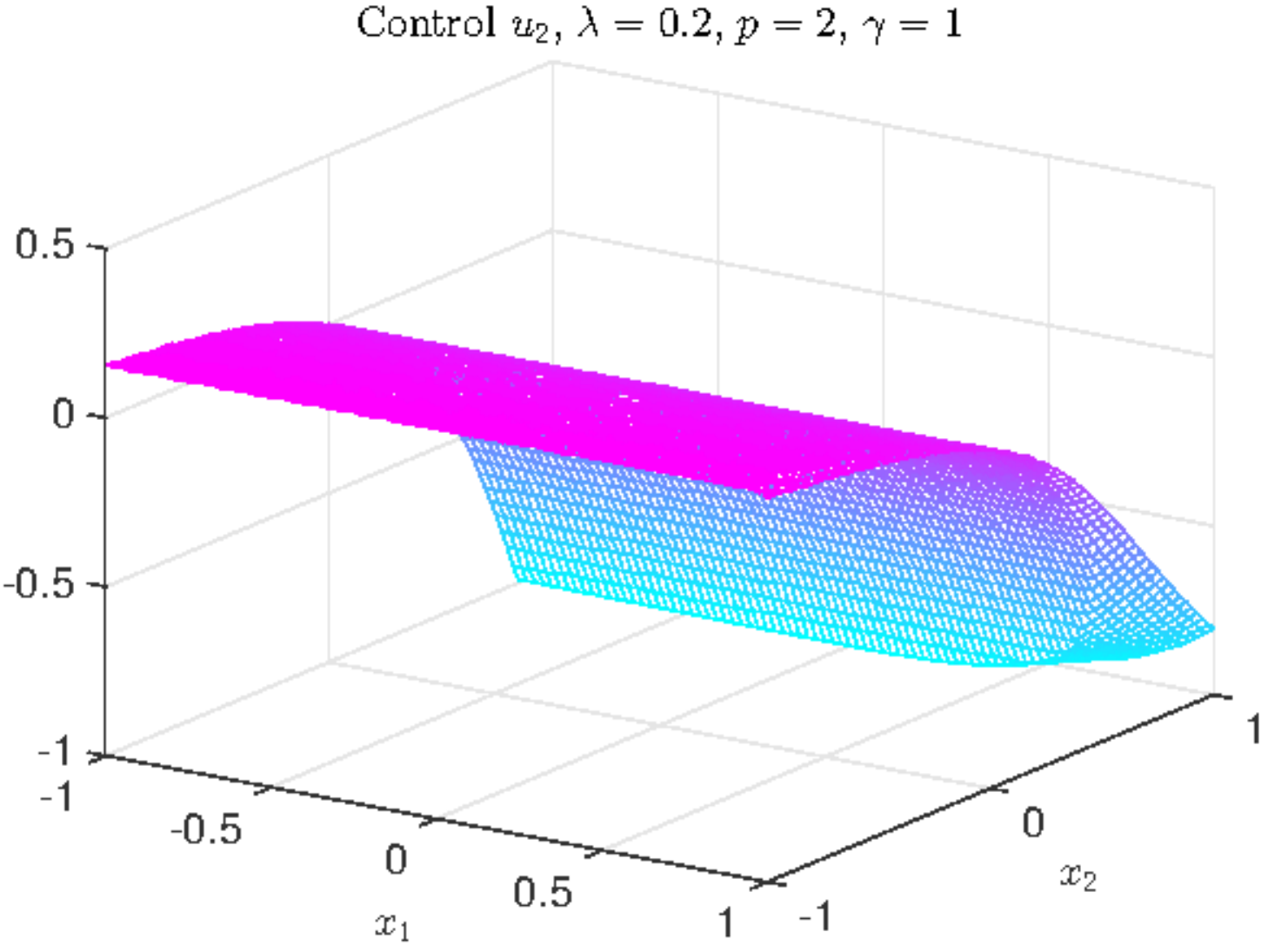}
\vskip 5mm
\centering
\includegraphics[width=0.495\textwidth]{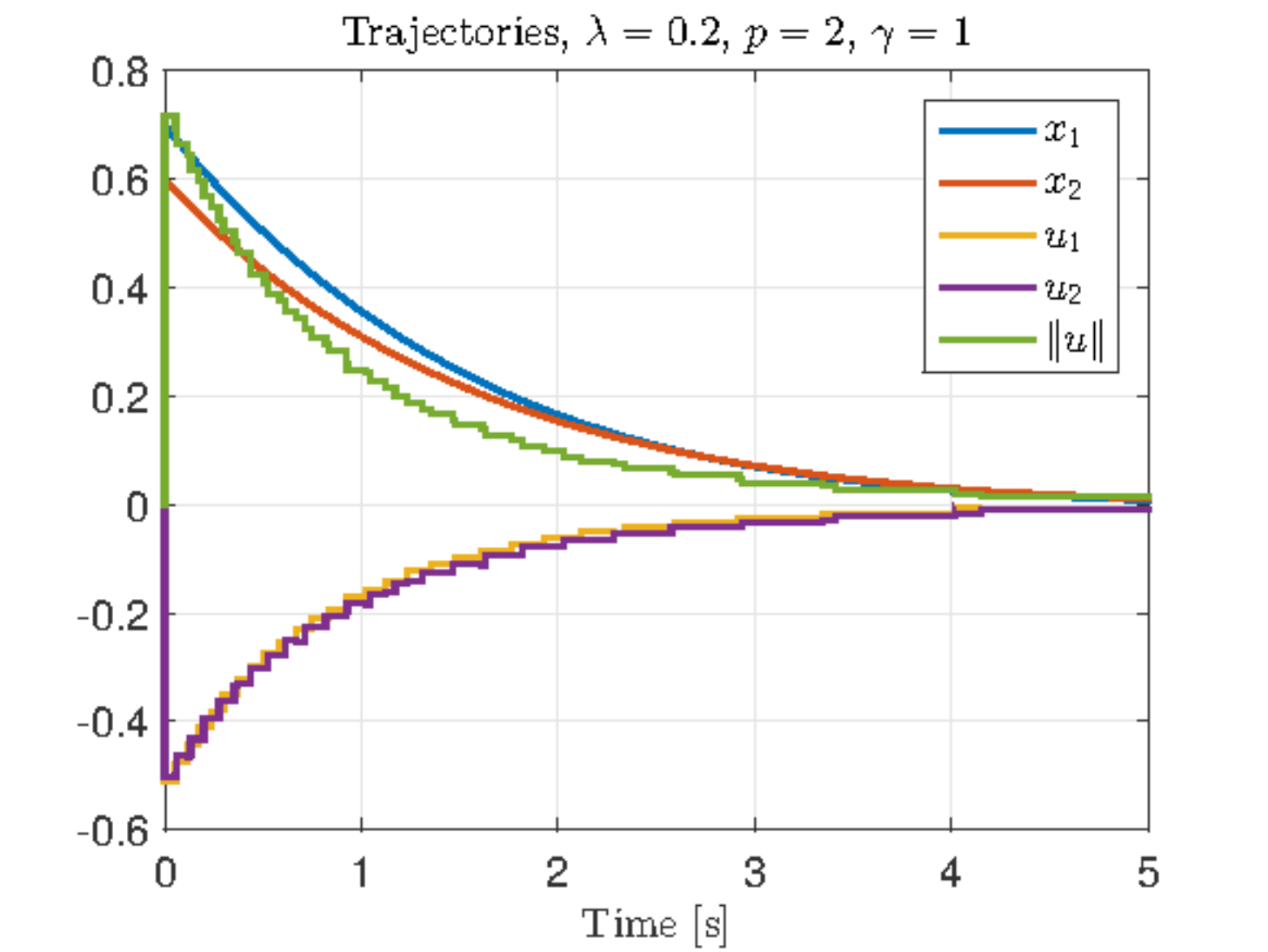}
\caption{Infinite horizon control of nonlinear dynamics with $p=2$. Row 1 and 2: value function $V(x_1,x_2)$, $\|u\|_2$ -norm of the optimal control, optimal controls $u_1(x_1,x_2)$ and $u_2(x_1,x_2)$, over the state space  $\Omega=[-1,1]^2$. Row 3: trajectories for the initial condition $(x_1(0),x_2(0))=(-0.75,-0.6).$}\label{nonlinp2}
\end{figure}

\begin{figure}[!h]
\includegraphics[width=0.495\textwidth]{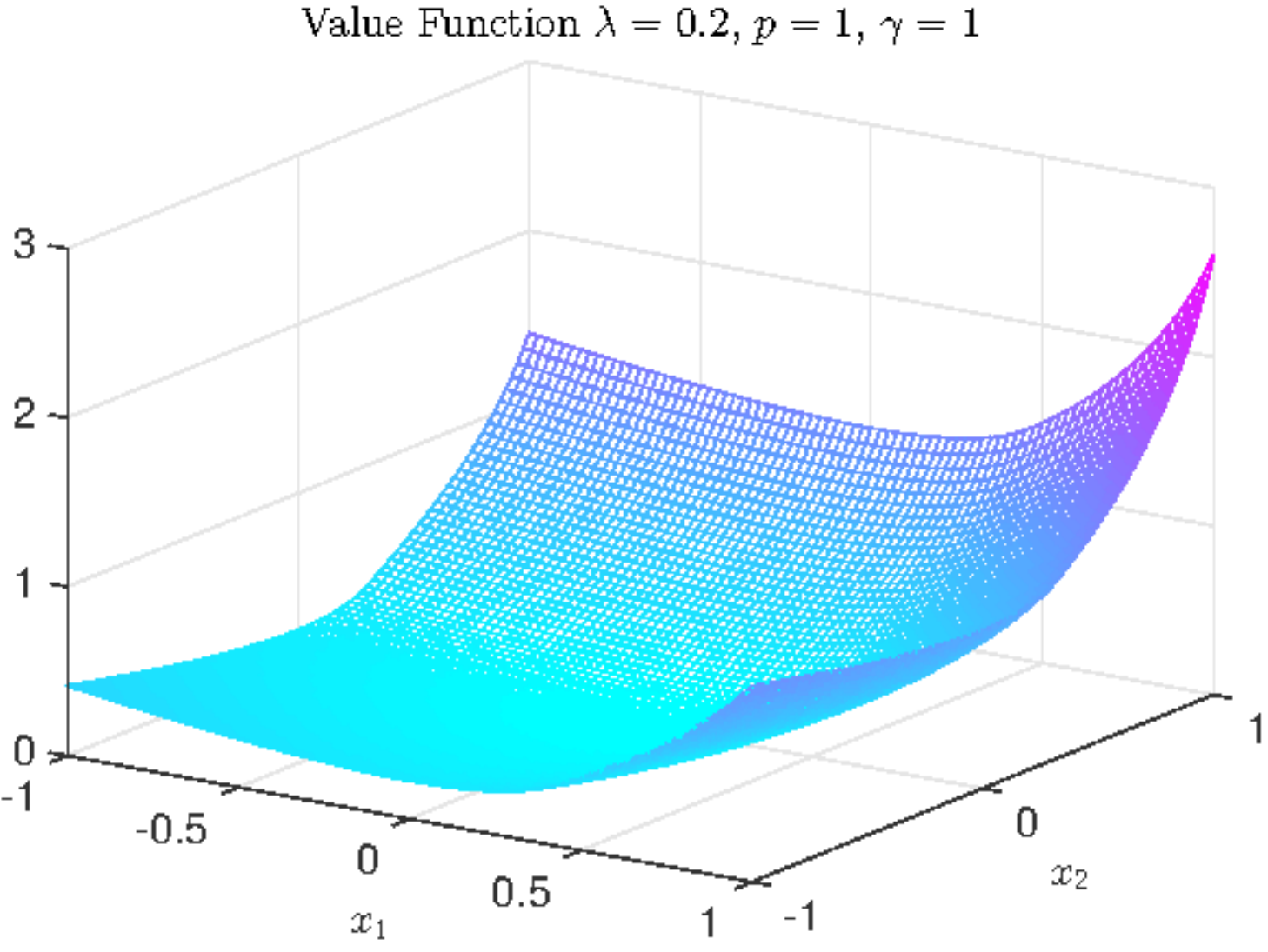}
\includegraphics[width=0.495\textwidth]{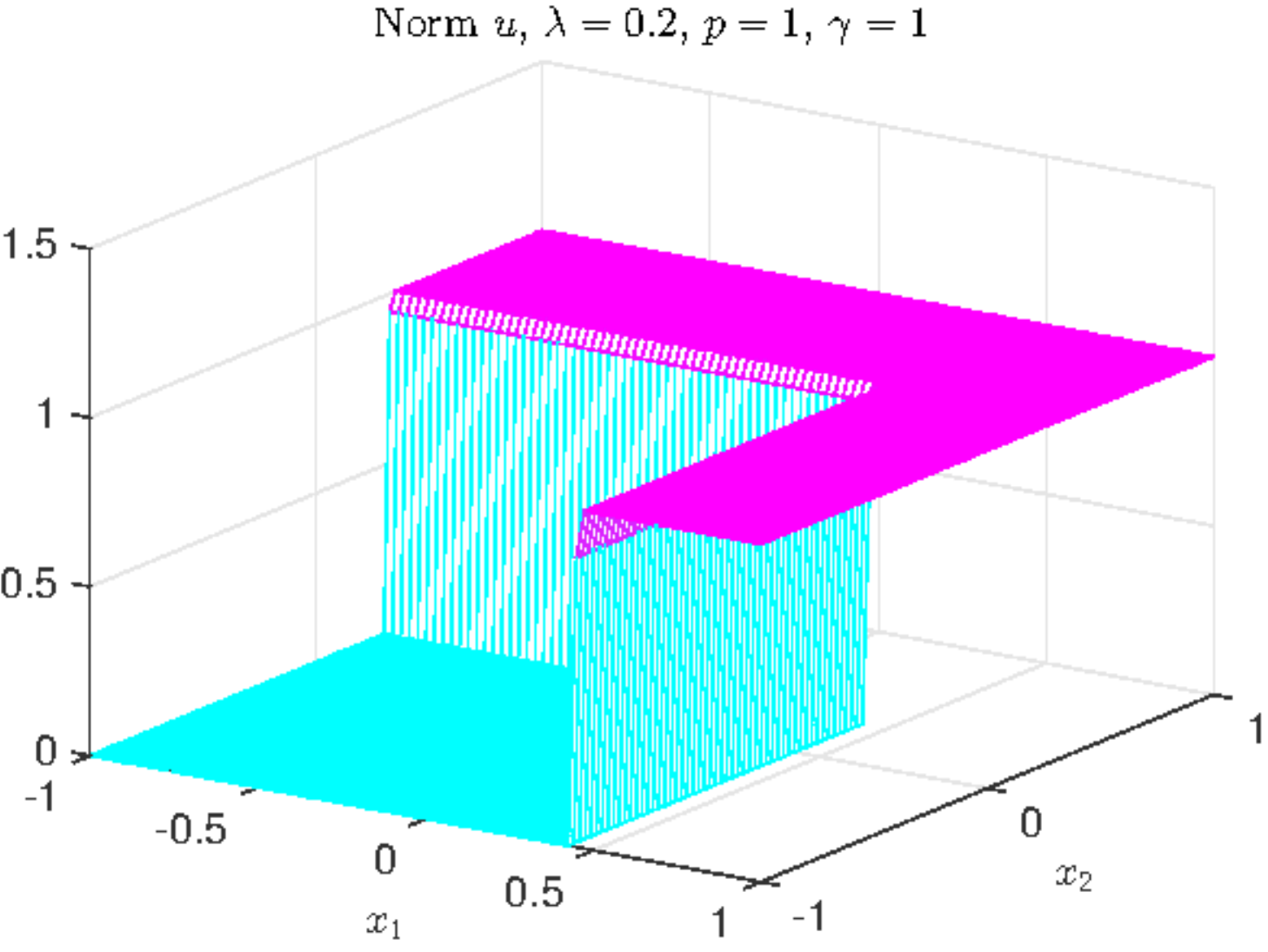}
\vskip 5mm
\includegraphics[width=0.495\textwidth]{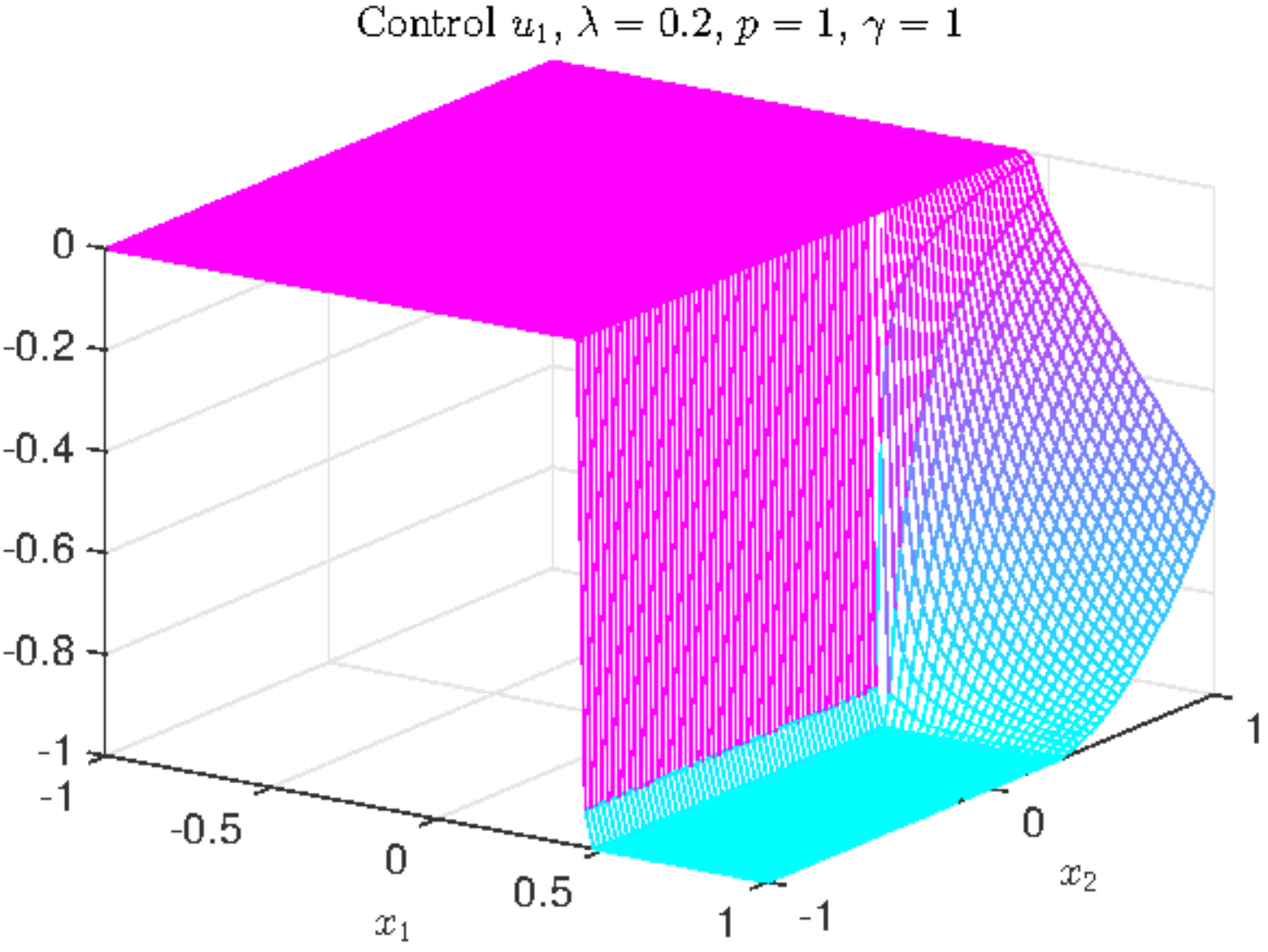}
\includegraphics[width=0.495\textwidth]{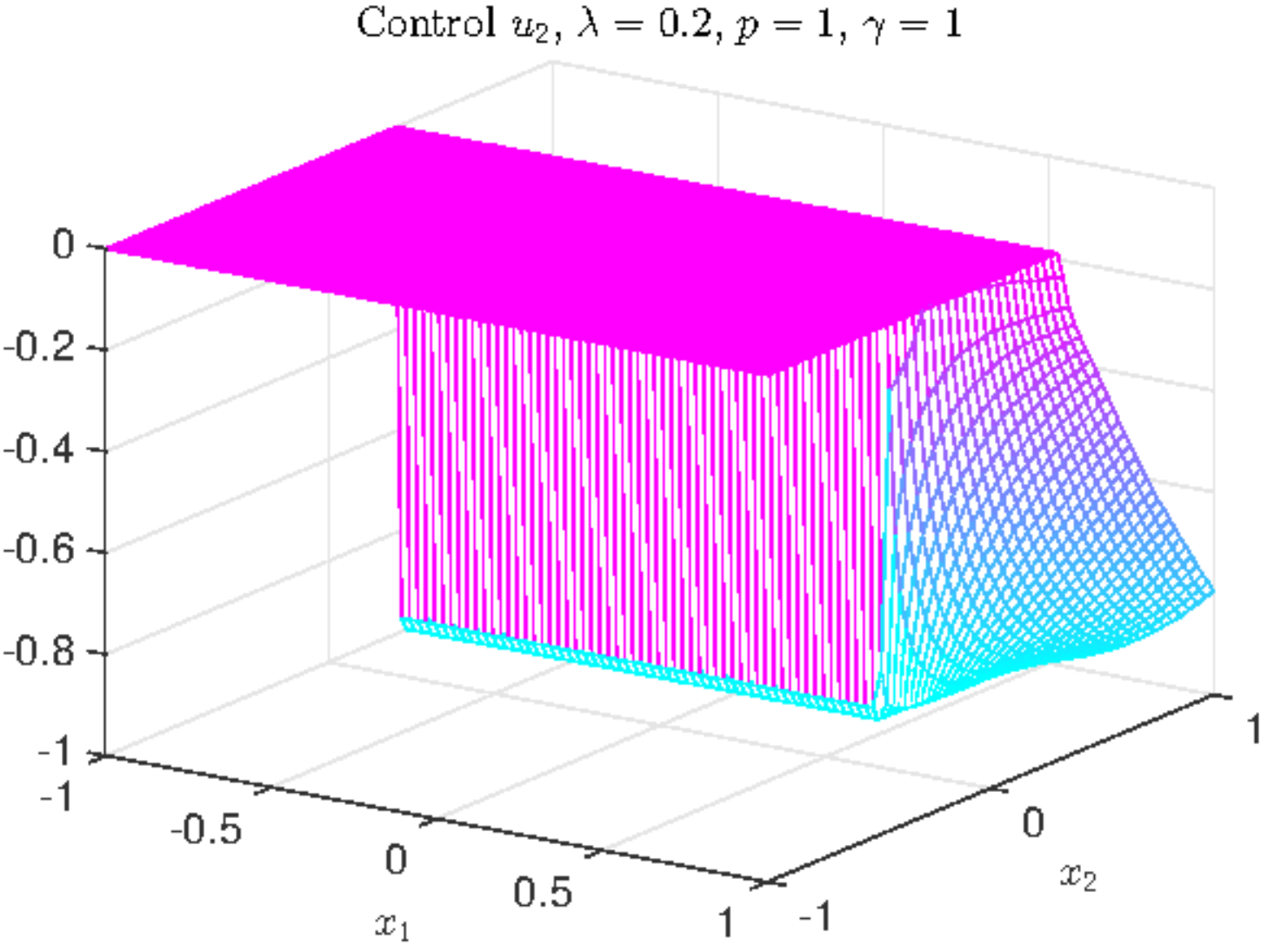}
\vskip 5mm
\centering
\includegraphics[width=0.495\textwidth]{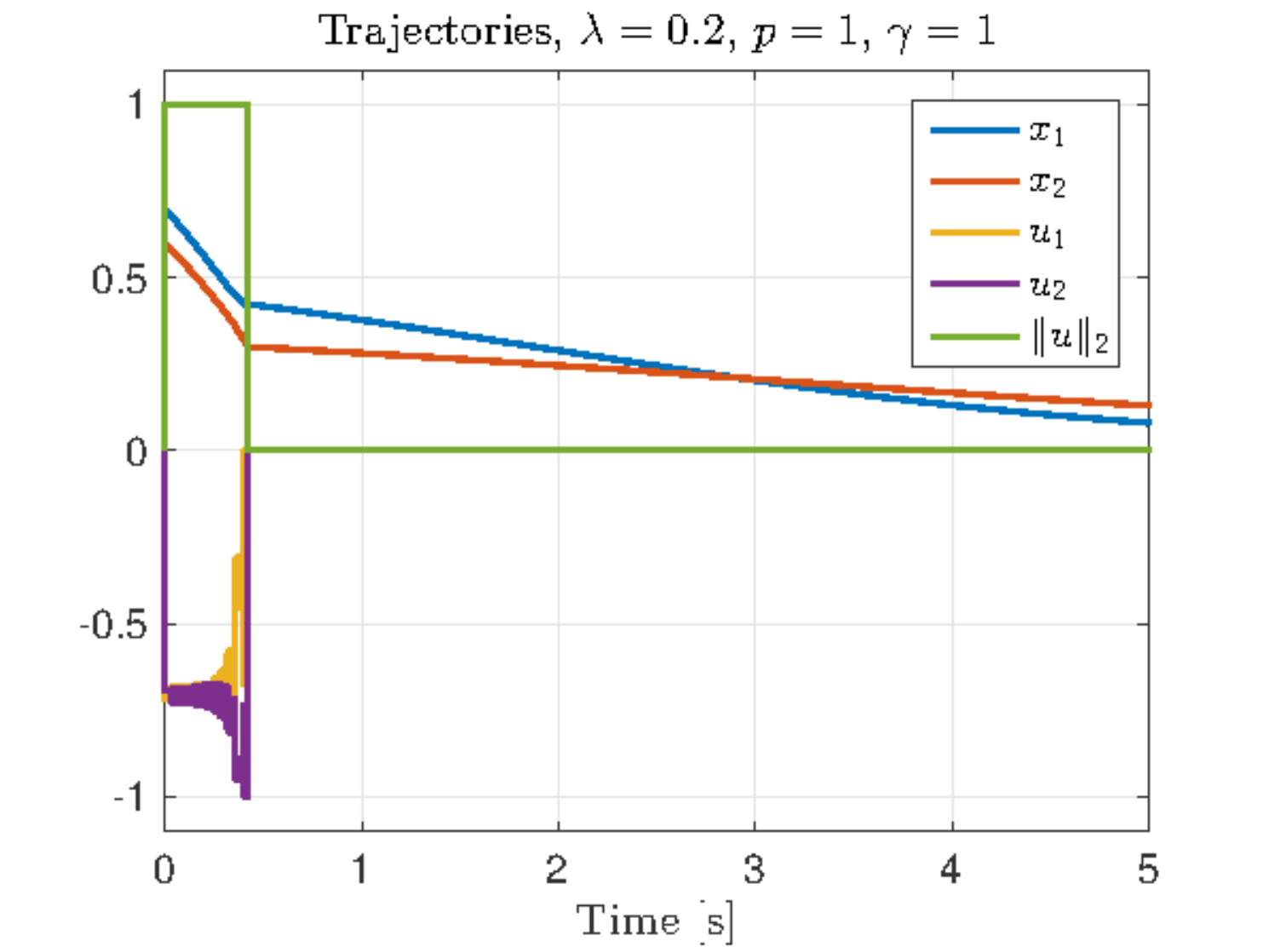}
\caption{Infinite horizon control of nonlinear dynamics with $p=1$. Row 1 and 2: value function $V(x_1,x_2)$, $\|u\|_2$ -norm of the optimal control, optimal controls $u_1(x_1,x_2)$ and $u_2(x_1,x_2)$, over the state space  $\Omega=[-1,1]^2$. Row 3: trajectories for the initial condition $(x_1(0),x_2(0))=(-0.75,-0.6).$}\label{nonlinp1}
\end{figure}

\begin{figure}[!h]
\includegraphics[width=0.495\textwidth]{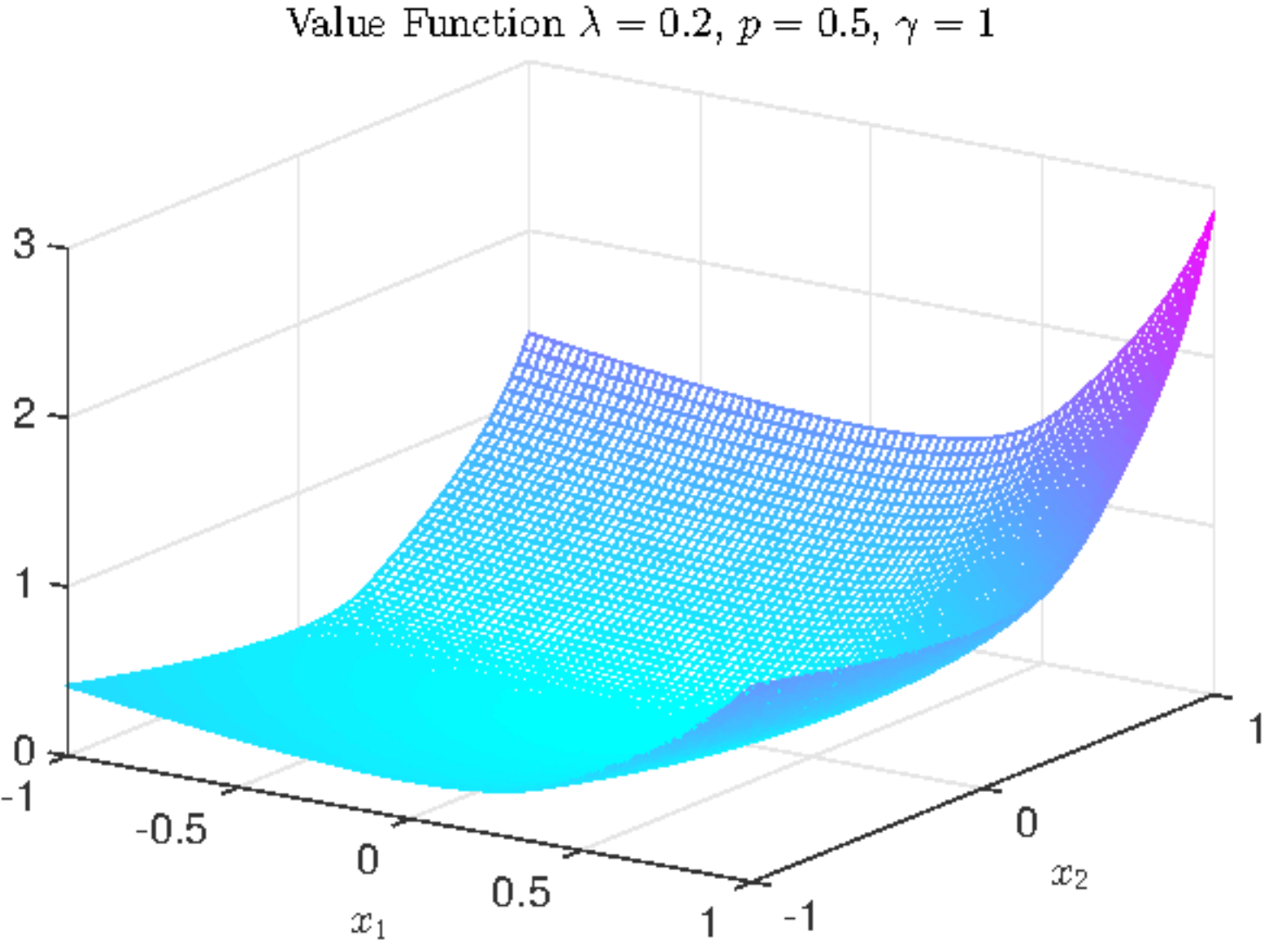}
\includegraphics[width=0.495\textwidth]{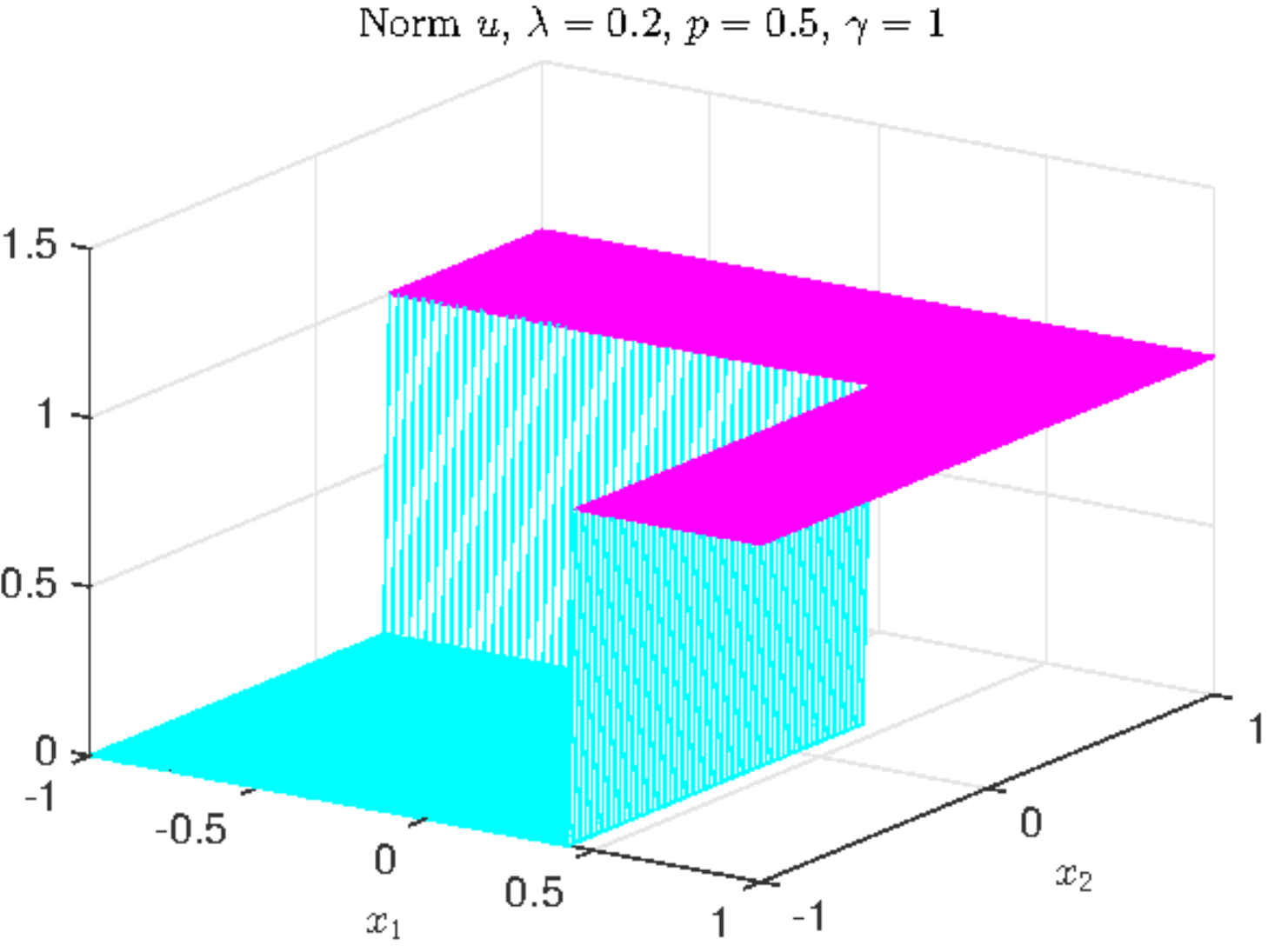}
\vskip 5mm
\includegraphics[width=0.495\textwidth]{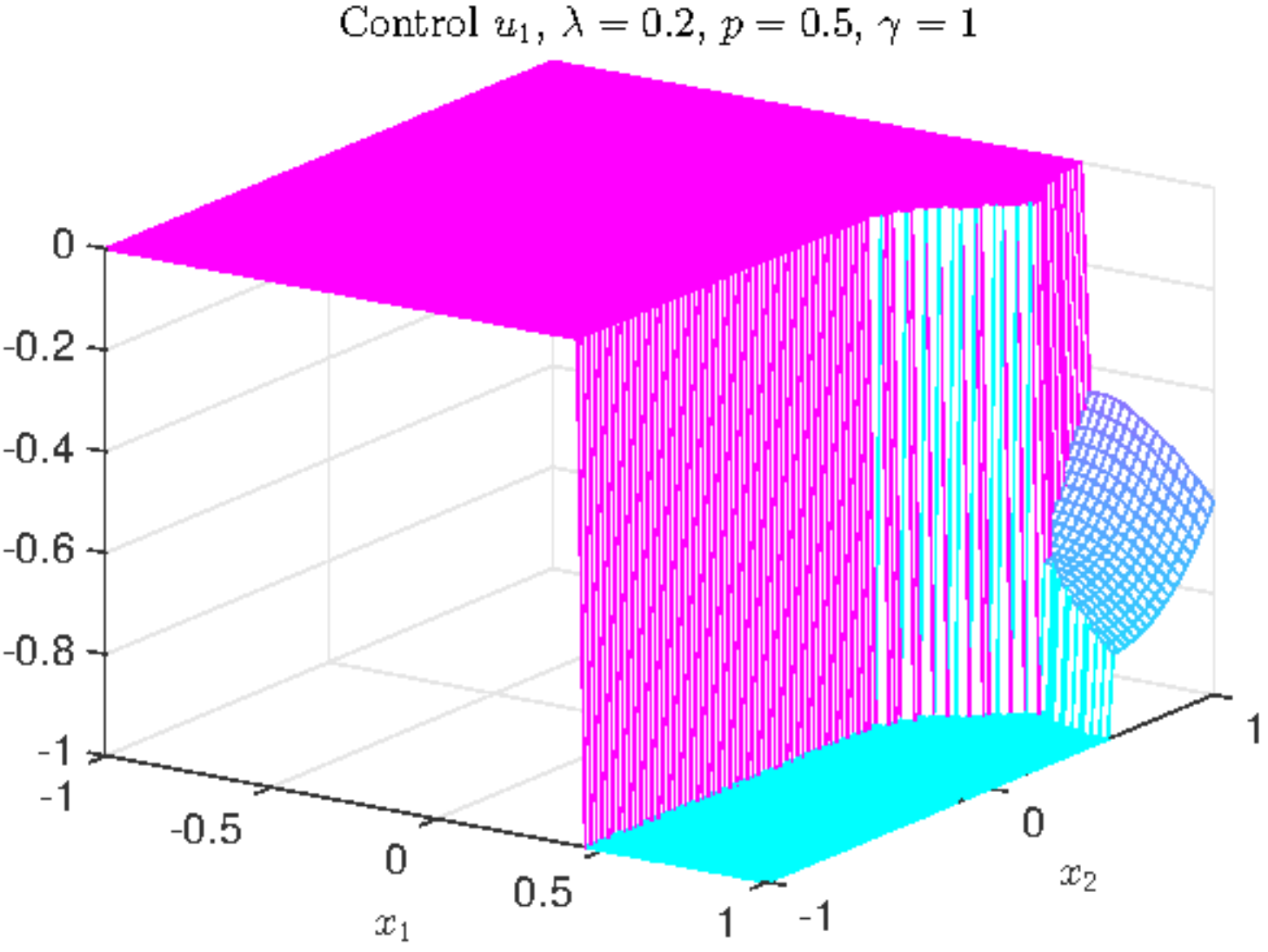}
\includegraphics[width=0.495\textwidth]{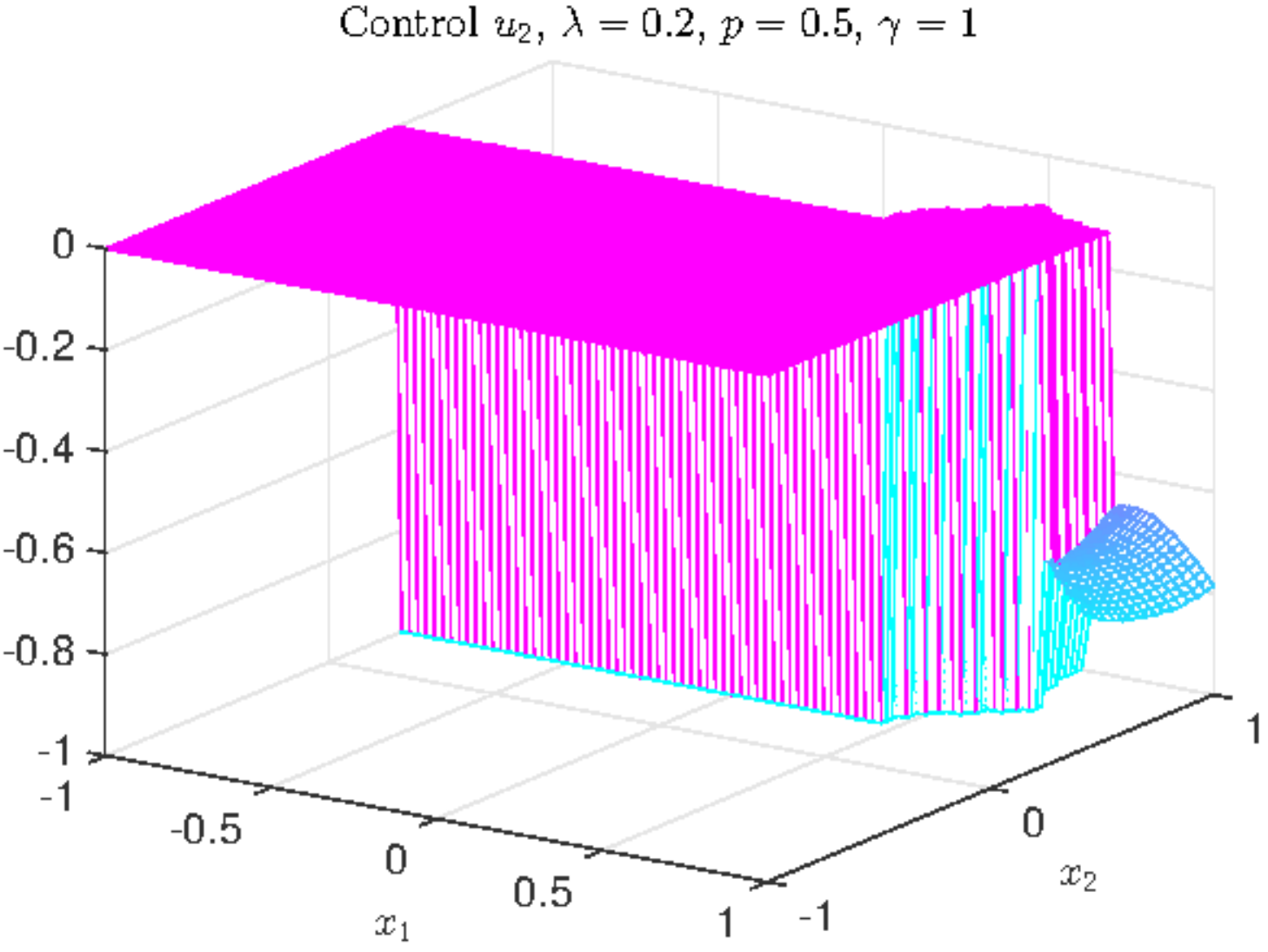}
\vskip 5mm
\centering
\includegraphics[width=0.495\textwidth]{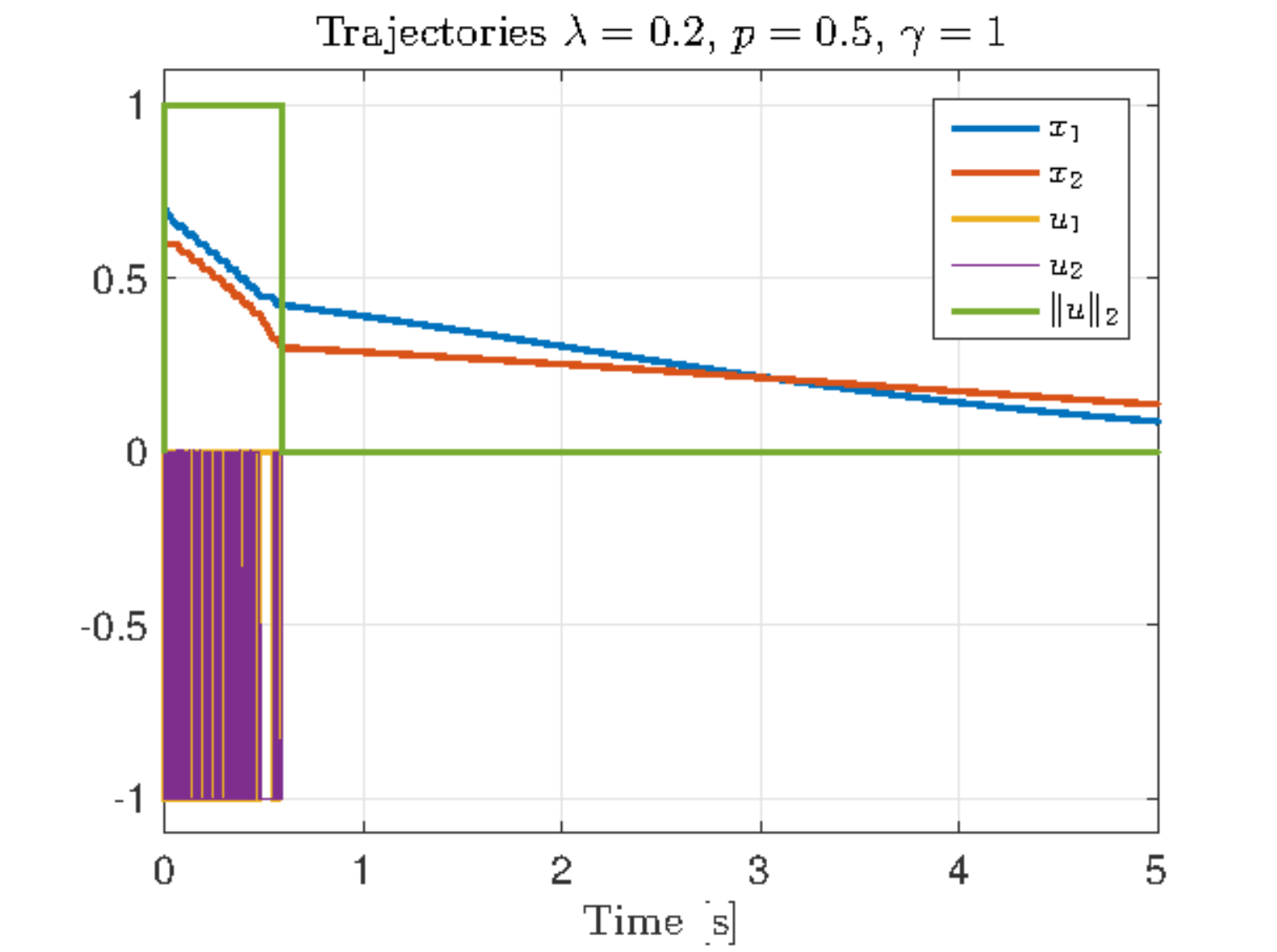}
\caption{Infinite horizon control of nonlinear dynamics with $p=0.5$. Row 1 and 2: value function $V(x_1,x_2)$, $\|u\|_2$ -norm of the optimal control, optimal controls $u_1(x_1,x_2)$ and $u_2(x_1,x_2)$, over the state space  $\Omega=[-1,1]^2$. Row 3: trajectories for the initial condition $(x_1(0),x_2(0))=(-0.75,-0.6).$}\label{nonlinp05}
\end{figure}

\section{Concluding remarks}

We have studied both theoretically and computationally,  sparsity properties of infinite horizon, optimal control problems with $L^p$-penalties, $p\in(0,1]$, in the controls. Existence results for fully  nonlinear dynamics were obtained in the convex case when $p=1$.  In the nonconvex case, existence results are presented systems which are affine in the controls. In both cases, first-order optimality conditions are also obtained. They are used as the basis to analyze the sparsity structure of the optimal controls.
In future work we shall  explore the relationship between sparse optimal control problems and stabilization properties. It can be noted, as in the undiscounted infinite horizon problem for Eikonal dynamics, that $L^1$ optimal control problems lead to controllers which can stabilize the system in finite time. Additionally, we plan to develop a  numerical approximation framework for optimal controllers in the nonconvex case.

\appendix  
\renewcommand\thesection{\Alph{section}}
\renewcommand\theequation{\Alph{section}.\arabic{equation}}
\section*{Appendix}
The proof of Theorem \ref{THMexistence} is given as follows.
\begin{proof}
Let $(T_k)_{k\in\N}$ be an arbitrary sequence of positive numbers with
\[
T_k<T_{k+1},\ \forall\,k\in\N\ \text{and}\ T_k\ra\infty\ \text{as}\ k\ra\infty.
\]
For each $k\in\N$, consider the following optimal control problems {\bf $(P_k)$} defined on the finite time interval $[0,T_k]$:
\begin{equation}\label{ocTk}
 \inf_{u\in L^{\infty}(0,T_k;U)} J_k(x,u),
\end{equation}
where
\[
J_k(x,u)=\int^{T_k}_0 e^{-\lambda s}\ell(y(s),u(s))ds
\]
and $(y(\cdot),u(\cdot))$ satisfies the dynamical system \eqref{ds}. Due to Lemma \ref{LEMocT}, there exists an optimal control $u_k$ for \eqref{ocTk}. Denote by $y_k$ the optimal trajectory corresponding to $u_k$ and
\[
\eta_k(s):=\int^s_0 e^{-\lambda(s'-s)}\ell(y_k(s'),u_k(s'))ds',\ \forall\,s\in [0,T_k].
\]
Then $\eta_k$ is optimal for \eqref{ocAugmented} with the final time $T_k$.
The aim is to construct an admissible control $\bar u$ defined on the infinite time interval $[0,\infty)$. The corresponding extended state will be denoted by $(\bar y,\bar \eta)$. The construction is described step by step as follows.

Consider at first the sequence $(y_k,\eta_k)_{k\in\N}$ on the time interval $[0,T_1]$. By the same arguments as in the proof of Lemma \ref{LEMocT}, there exists a subsequence $(y_{k,1},\eta_{k,1})$ of $(y_k,\eta_k)$ such that
\[
(y_{k,1},\eta_{k,1})\ra (\bar y,\bar \eta)  \text{ uniformly in } [0,T_1],\ \text{as}\ k\ra\infty;
\]
\[
(\dot y_{k,1},\dot{\eta}_{k,1})\ra (\dot{\bar y},\dot{\bar \eta})  \text{ weakly in } L^1(0,T_1;\R^{d+1}),\ \text{as}\ k\ra\infty,
\]
for some $(\bar y,\bar \eta)$ satisfying \eqref{dsAugmented} on $[0,T_1]$. Note that each control $u_{k,1}$ is an optimal control for a corresponding problem $(P_{n(k,1)})$ of the form \eqref{ocTk} for some integer $n(k,1)\geq 1$ on the time interval $[0,T_{n(k,1)}]$, and each $\eta_{k,1}$ is optimal for the problem \eqref{ocAugmented} with $T_{n(k,1)}$.

Now for the second step consider the sequence $(y_{k,1},\eta_{k,1})_{k\in\N}$ on the time interval $[0,T_2]$ for $k\geq 2$. Analogously to the previous step there exists $(y_{k,2},\eta_{k,2})$ of $(y_{k,1},\eta_{k,1})$ such that
\[
(y_{k,2},\eta_{k,2})\ra (\bar y,\bar \eta)  \text{ uniformly in } [0,T_2],\ \text{as}\ k\ra\infty;
\]
\[
(\dot y_{k,2},\dot{\eta}_{k,2})\ra (\dot{\bar y},\dot{\bar \eta})  \text{ weakly in } L^1(0,T_2;\R^{d+1}),\ \text{as}\ k\ra\infty,
\]
for $(\bar y,\bar \eta)$ satisfying \eqref{dsAugmented} on $[0,T_2]$. Here $(\bar y,\bar \eta)$ coincides with the one constructed in the previous step on $[0,T_1]$, and it is denoted again by the same symbol. Each control $u_{k,2}$ is an optimal control for a corresponding problem $(P_{n(k,2)})$ of the form \eqref{ocTk} for some integer $n(k,2)\geq 2$ on the time interval $[0,T_{n(k,2)}]$, and each $\eta_{k,2}$ is optimal for the problem \eqref{ocAugmented} with $T_{n(k,2)}$.

By repeating this procedure, we construct $(\bar y,\bar \eta)$ satisfying \eqref{dsAugmented} on the infinite time interval $[0,\infty)$. Simultaneously we obtain a countable family of $(y_{k,i},\eta_{k,i})$ for $i,k\in\N$, $k\geq i$. Each $u_{k,i}$ is an optimal control for problem $(P_{n(k,i)})$ of the form \eqref{ocTk} for some integer $n(k,i)\geq i$ on the time interval $[0,T_{n(k,i)}]$, and each $\eta_{k,i}$ is optimal for the problem \eqref{ocAugmented} with $T_{n(k,i)}$. Moreover, for all $i\in\N$,
\[
(y_{k,i},\eta_{k,i})\ra (\bar y,\bar \eta)  \text{ uniformly in } [0,T_i],\ \text{as}\ k\ra\infty;
\]
\[
(\dot y_{k,i},\dot{\eta}_{k,i})\ra (\dot{\bar y},\dot{\bar \eta})  \text{ weakly in } L^1(0,T_i;\R^{d+1}),\ \text{as}\ k\ra\infty,
\]
Let us take the diagonal sequence $(u_{k,k})_{k\in\N}$ and denote
\[
\tilde u_k=u_{k,k},\ \tilde y_k=y_{k,k},\ \tilde \eta_k=\eta_{k,k},\ \text{and}\ n_k=n(k,k),\ \forall\,k\in\N.
\]
Then the following properties hold:
\begin{enumerate}[(i)]
 \item 
 $\forall\,k\in\N$, the control $\tilde u_k$ is defined on the time interval $[0,T_{n_k}]$ with $n_k\geq k$, $\tilde u_k$ is an optimal control for the problem {\bf $(P_{n_k})$} of the form \eqref{ocTk}, and $\tilde \eta_k$ is optimal for the problem \eqref{ocAugmented} with $T_{n_k}$.
 \item
 $\forall\,i\in\N$, we have
\[
(\tilde y_k,\tilde \eta_k) \ra (\bar y,\bar \eta)  \text{ uniformly in } [0,T_i]\ \text{as}\ k\ra\infty,
\]
\[
(\dot{\tilde y}_k,\dot{\tilde \eta}_k) \ra (\dot{\bar y},\dot{\bar \eta})\ \text{weakly in}\ L^1(0,T_i;\R^{d+1})\ \text{as}\ k\ra\infty.
\]
\item 
There exists $\bar u\in L^{\infty}(0,\infty;U)$ such that
\[
\dot{\bar y}(s)=f(\bar y(s),\bar u(s))\ \text{and}\ \dot{\bar \eta}(s)\geq \lambda \bar \eta(s)+\ell(\bar y(s),\bar u(s)),\ \forall\,s\in (0,\infty).
\]
\end{enumerate}
We proceed with proving that $\bar u$ is an optimal control for the problem \eqref{ocinfini}. Arguing by contradiction, if $\bar u$ is not optimal for \eqref{ocinfini}, there exists $\e>0$ and $(\tilde y,\tilde u)$ satisfying \eqref{ds} such that
\begin{equation}\label{EqL10}
J(x,\tilde u)+\e< J(x,\bar u).
\end{equation}
By the assumption {\bf (H2)}, there exists $(y^*,u^*)$ satisfying \eqref{ds} such that 
\[
J(x,u^*)<+\infty.
\]
Since $\tilde u_k$ is optimal for ($P_{n_k}$), we have for any $k\in\N$ that
\[
J_{n_k}(x,\tilde u_k)\leq J_{n_k}(x,u^*)\leq J(x,u^*).
\]
Then for any $N>0$ and any $n_k$ with $T_{n_k}\geq N$,
\[
\int^N_0 e^{-\lambda s}\ell(\tilde y_k(s),\tilde u_k(s))ds\leq J_{n_k}(x,\tilde u_k)\leq J(x,u^*),
\]
and thus
\[
e^{-\lambda N}\tilde \eta_k(N)\leq J(x,u^*).
\]
By taking $k\ra\infty$, it holds that
\[
e^{-\lambda N}\bar \eta(N)\leq J(x,u^*).
\]
We thus obtain,
\[
\int^N_0 e^{-\lambda s}\ell(\bar y(s),\bar u(s))ds\leq e^{-\lambda N}\bar \eta(N)\leq J(x,u^*),\ \forall\,N>0,
\]
and therefore
\[
\int^{\infty}_0 e^{-\lambda s}\ell(\bar y(s),\bar u(s))ds\leq J(x,u^*).
\]
There exists $k_1\in\N$ such that for any $k\geq k_1$
\begin{equation}\label{EqL11}
\int^{\infty}_{T_{k}}e^{-\lambda s}\ell(\bar y(s),\bar u(s))ds<\frac{\e}{2}.
\end{equation}
Due to the fact that $\tilde \eta_k\ra \bar \eta$ uniformly, there exists $k_2\geq k_1$ such that for all $k\geq k_2$,
\[
e^{-\lambda T_{n_{k_1}}}\bar \eta(T_{n_{k_1}})\leq e^{-\lambda T_{n_{k_1}}}\tilde \eta_k(T_{n_{k_1}})+\frac{\e}{2},
\]
which implies that
\begin{equation}\label{EqL12}
J_{n_{k_1}}(x,\bar u)\leq J_{n_{k_1}}(x,\tilde u_k)+\frac{\e}{2}.
\end{equation}
Since $\tilde u_{k_2}$ is optimal for {\bf $(P_{n_{k_2}})$},
\begin{equation}\label{EqL13}
J_{n_{k_2}}(x,\tilde u_{k_2})\leq J_{n_{k_2}}(x,\tilde u)\leq J(x,\tilde u).
\end{equation}
Note that $n_{k_1}\leq n_{k_2}$, and thus together with \eqref{EqL12} and \eqref{EqL13}, we have
\begin{eqnarray*}
J_{n_{k_1}}(x,\bar u) &\leq& J_{n_{k_1}}(x,\tilde u_{k_2})+\frac{\e}{2}
\leq J_{n_{k_2}}(x,\tilde u_{k_2})+\frac{\e}{2} 
\leq J(x,\tilde u) + \frac{\e}{2}.
\end{eqnarray*}
Finally, by \eqref{EqL11} we deduce that
\begin{eqnarray*}
J(x,\bar u) &=& J_{n_{k_1}}(x,\bar u)+ \int^{\infty}_{T_{n_{k_1}}}e^{-\lambda s}\ell(\bar y(s),\bar u(s))ds \\
&\leq& J_{n_{k_1}}(x,\bar u)+ \frac{\e}{2}
\leq J(x,\tilde u) + \e,
\end{eqnarray*}
which contradicts \eqref{EqL10}. Hence $\bar u$ is an optimal control for \eqref{ocinfini}, which ends the proof.
\end{proof}

\end{document}